\documentclass{article}

\usepackage[utf8]{inputenc}
\usepackage{amsmath, amssymb, amsthm, bm, mathrsfs}
\usepackage[top=1.5cm, bottom=1.5cm, left=2cm, right=2cm]{geometry}
\usepackage{csquotes, authblk, cases, enumitem}
\usepackage{graphicx, float, caption, subcaption, adjustbox}
\usepackage[hidelinks]{hyperref}
\usepackage[numbers]{natbib}
\usepackage[section, plain]{algorithm}
\usepackage{algpseudocode}
\usepackage{array, tabularx, xcolor, colortbl}
\usepackage{color, xcolor}
\usepackage[para]{footmisc}
\usepackage[normalem]{ulem}
\counterwithin{figure}{section}
\counterwithin{table}{section}
\numberwithin{equation}{section}
\graphicspath{{Figures/}}
\definecolor{darkred}{rgb}{0.8,0,0}

\DeclareMathOperator{\Span}{\operatorname{span}}

\DeclareMathOperator{\diag}{\operatorname{diag}}

\newenvironment{frcseries}{\fontfamily{frc}\selectfont}{}

\theoremstyle{definition}
\newtheorem{definition}{Definition}[section]
\newtheorem{remark}[definition]{Remark}
\newtheorem{example}[definition]{Example}
\newtheorem{theorem}[definition]{Theorem}
\newtheorem{lemma}[definition]{Lemma}
\newtheorem{corollary}[definition]{Corollary}

\title{Mass lumping and outlier removal strategies for complex geometries in isogeometric analysis}
\author[1]{Yannis Voet \thanks{yannis.voet@epfl.ch}}
\author[1]{Espen Sande \thanks{espen.sande@epfl.ch}}
\author[1]{Annalisa Buffa \thanks{annalisa.buffa@epfl.ch}}
\affil[1]{\small MNS, Institute of Mathematics, École polytechnique fédérale de Lausanne, Station 8, CH-1015 Lausanne, Switzerland}
\date{\today}

\begin{document}

\maketitle

\begin{abstract}
Mass lumping techniques are commonly employed in explicit time integration schemes for problems in structural dynamics and both avoid solving costly linear systems with the consistent mass matrix and increase the critical time step. In isogeometric analysis, the critical time step is constrained by so-called ``outlier'' frequencies, representing the inaccurate high frequency part of the spectrum. Removing or dampening these high frequencies is paramount for fast explicit solution techniques. In this work, we propose mass lumping and outlier removal techniques for nontrivial geometries, including multipatch and trimmed geometries. Our lumping strategies provably do not deteriorate (and often improve) the CFL condition of the original problem and are combined with deflation techniques to remove persistent outlier frequencies. Numerical experiments reveal the advantages of the method, especially for simulations covering large time spans where they may halve the number of iterations with little or no effect on the numerical solution.

\end{abstract}

\section{Introduction and background}
Isogeometric analysis is a discretization technique for solving partial differential equations (PDEs) that relies on spline functions such as B-splines and non-uniform rational B-splines (NURBS) both for parameterizing the geometry and representing the solution \citep{hughes2005isogeometric, cottrell2009isogeometric}. Spline functions used in isogeometric analysis offer distinctive advantages over Lagrange polynomials used in classical finite element discretizations, including exact representation of common geometries and superior approximation properties \cite{bazilevs2006isogeometric, bressan2019approximation, sande2020explicit}. In structural dynamics, the advantages of isogeometric analysis were already evidenced in \citep{cottrell2006isogeometric, hughes2008duality, hughes2014finite} where maximally smooth spline approximations removed the so-called ``optical branches'' from the discrete spectrum, a typical artifact of classical finite element discretizations. As a matter of fact, nearly all discrete eigenvalues approximate the continuous eigenvalues with high accuracy except for the largest ones. They form noticeable spikes in the upper part of the spectrum and were coined ``outlier'' eigenvalues for this very reason \cite{cottrell2006isogeometric}. These inaccurate eigenvalues grow with the mesh size and spline order \cite{gallistl2017stability} thereby severely constraining the critical time step of explicit time integration schemes, often referred to as the Courant–Friedrichs–Lewy (CFL) condition. For example, for undamped dynamical systems, the critical time step of the central difference method is
\begin{equation}
\label{eq: CFL_central_difference}
    \Delta t_c= \frac{2}{\sqrt{\lambda_n}}
\end{equation}
where $\lambda_n$ is the largest eigenvalue of the discrete system \citep{hughes2012finite, bathe2006finite}.

Since the advent of isogeometric analysis, much effort has focused on removing the outliers from the discrete spectrum. A nonlinear spline parametrization was first proposed in \cite{cottrell2006isogeometric} by uniformly distributing control points and changing the geometry parametrization but defeats the spirit of isogeometric analysis. The method also reportedly undermines the accuracy of the low frequencies and modes \cite{hiemstra2021removal}. A similar strategy was later proposed in \cite{chan2018multi} by constructing ``smoothed'' knot vectors that are approximations to optimal knot vectors (related to some $n$-width optimal spline spaces) but suffers from similar drawbacks. Instead, in \cite{sande2019sharp}, the authors used the $n$-width optimal spline spaces defined in \cite{floater2019optimal} to remove outliers, and these spaces were later proven to be outlier-free in \cite{manni2022application} without loss of accuracy in the low frequencies. These optimal spaces mimic the true eigenfunctions by imposing certain higher-order derivatives to be zero at the boundary. Bases for these spaces were constructed in \cite{takacs2016approximation,floater2019optimal,sogn2019robust,manni2022application} by using certain symmetry properties of B-splines. A similar strategy was numerically observed to be outlier-free in \cite{hiemstra2021removal} for the Laplacian, however, some outliers were still observed for the biharmonic problem with the spline spaces proposed in \cite{hiemstra2021removal}. By mimicking the properties of the $n$-width optimal spline spaces in \cite{floater2017optimal}, outliers for the biharmonic problem were later completely removed in \cite{manni2023outlier}. In both \cite{hiemstra2021removal} and \cite{manni2023outlier}, the basis construction resembles the MDB-spline extraction technique developed in \cite{toshniwal2020multi}. A related outlier removal technique was devised in \cite{deng2021boundary}, where the authors weakly impose the constraints coming from the true eigenfunctions by penalizing high order derivatives near the boundary. 
This same strategy was later employed in \cite{nguyen2022variational} in the multipatch setting by weakly enforcing $C^{p-1}$ continuity at patch interfaces to remove interior outliers. These penalization techniques often take the form of mass scaling, which is also widely employed for classical finite element methods \citep{macek1995mass, olovsson2005selective, stoter2022variationally}. Rather than completely removing outliers, they strongly mitigate them but also require heuristics or dedicated algorithms for computing the generally unknown penalization parameters. We note that these approaches could be combined to mitigate both boundary and interior outliers resulting from $C^0$ continuity at patch interfaces. Also in the context of classical finite element methods, an eigenvalue deflation technique was proposed in \cite{tkachuk2014local,gonzalez2020large} by explicitly computing the largest eigenvalues and eigenmodes. However, the authors do not discuss the computational overhead in their experiments and applying their method to the global assembled mass matrix is generally infeasible due to the large number of inaccurate high frequencies in classical $C^0$ finite element methods.

Outlier removal strategies are generally motivated for one-dimensional problems and then extended to higher dimensions via a tensor product construction. This construction, however, inherently limits the applicability of the methods to trivial single-patch geometries and separable coefficient functions. One issue in extending these approaches to nontrivial problems lies in the definition of outliers. For nontrivial geometries, outliers are often smoothened out and the spectrum generally does not feature any spikes, although changes in curvature are sometimes noticeable. Outliers then lose the intrinsic property that characterized them and their identification becomes ambiguous. Straightforwardly applying the constructions proposed in \citep{hiemstra2021removal, manni2022application, deng2021boundary} to nontrivial geometries generally does not yield satisfactory results, suggesting that the support of the outlier eigenfunctions might stretch into the domain's interior.

For applications in explicit dynamics, restrictive CFL conditions are not the only issue. Due to the inherent cost of ``exactly'' solving linear systems with the mass matrix \citep{collier2012cost, collier2013cost}, obtaining an easily invertible (preferably diagonal or tridiagonal) mass matrix is paramount. Mass lumping has historically been used for approximating the consistent mass matrix by a diagonal (lumped) mass matrix. Among the scores of methods proposed in the 1970s for classical finite element methods, only a handful are applicable to isogeometric analysis due to the non-interpolatory nature of the basis functions. Among them is the classical row-sum technique \citep{zienkiewicz2005finite, hughes2012finite}. Provably, the row-sum technique does not worsen the CFL condition of the consistent mass \cite{voet2023mathematical}. However, for isogeometric analysis, the method performs poorly in higher dimensions and strong numerical evidence suggests it reduces the converge rate to quadratic order, independently of the spline order \citep{cottrell2006isogeometric}; a property only proved for 1D problems and low spline orders \cite{cottrell2006isogeometric}. Since then, much research effort has focused on devising more accurate and potentially high order mass lumping schemes for isogeometric analysis. Cottrell et al. \cite{cottrell2009isogeometric} first suggested constructing diagonal mass matrices by using dual basis functions as test functions in a Petrov-Galerkin framework. However, computer implementations come with all sorts of difficulties and initially the idea did not gain much momentum until it was taken up again in \cite{anitescu2019isogeometric} with promising results, sparking a surge of interest in dual bases. In \citep{nguyen2023towards}, high order convergence is achieved by combining (approximate) dual basis functions with the row-sum technique. Unfortunately, the method is currently limited to single-patch geometries and extensions to multipatch or trimmed geometries are not in sight. Even in the single-patch case, several technical difficulties still hold it back, such as the imposition of boundary conditions, which is the focus of ongoing research \cite{hiemstra2023higher}. In another line of research, families of banded and Kronecker product matrices were constructed in \cite{voet2023mathematical} by increasing the bandwidth of the row-sum lumped mass matrix and were shown to significantly improve the accuracy. Unfortunately, such improvements are only realized on trivial geometries and are tied to an improvement of the constant instead of the convergence rate.

Mass lumping may sometimes dramatically impact the CFL condition. For trimmed geometries \cite{marussig2018review}, Leidinger \cite{leidinger2020explicit} first showed that the CFL condition was not affected by small trimmed elements if the mass matrix was lumped provided the spline's order and smoothness were sufficiently large with respect to the order of the differential operator. This finding was further supported by the studies in \citep{coradello2021accurate, stoter2023critical, radtke2024analysis} but also raised concerns over the accuracy of the smallest eigenvalues and modes. Devising accurate mass lumping/scaling techniques for trimmed geometries is an ongoing challenge.

Fast solution methods in explicit dynamics usually combine outlier removal, mass scaling and mass lumping in a sometimes ad hoc fashion. In this article, we propose such a strategy for complex geometries in isogeometric analysis and is substantiated by a strong mathematical foundation. Our contributions are twofold: firstly, we extend the mass lumping techniques proposed in \cite{voet2023mathematical} to nontrivial single-patch, multipatch and trimmed geometries and secondly propose an algebraic outlier removal strategy compatible with any lumping technique. Although this method has been independently proposed for classical finite element discretizations \cite{tkachuk2014local,gonzalez2020large}, we generalize it and thoroughly discuss its implementation and computational cost. The article is structured as follows: after describing the model problem and its discretization in Section \ref{se: model_problem}, we recall in Section \ref{se: mass_lumping} the mass lumping techniques devised in \cite{voet2023mathematical} and later extend them to nontrivial single-patch and multipatch problems. Although these techniques provably do not worsen the CFL condition of the original problem, they may not significantly improve it either. Thus, they are combined in Section \ref{se: outlier_removal} with outlier removal techniques that deflate the spectrum from persistent outlier eigenvalues and generalize the method proposed in \cite{tkachuk2014local,gonzalez2020large}. Contrary to classical $C^0$ finite elements, isogeometric discretizations typically feature a small number of rapidly increasing eigenvalues towards the end of the spectrum, for which deflation techniques are well-suited. Section \ref{se: numerical_experiments} gathers some numerical experiments illustrating the theoretical findings and demonstrating the advantages of the method. Finally, conclusions are drawn in Section \ref{se: conclusion}.

\section{Model problem and its discretization}
\label{se: model_problem}
In this article, we consider hyperbolic PDEs from structural dynamics. Their simplest instance is the classical wave (or acoustic) equation, which we will select as model problem. Let $\Omega \subset \mathbb{R}^d$ be an open connected domain in $d$-dimensional space with Lipschitz boundary and let $I=[0,T]$ be the time domain with $T>0$ denoting the final time. We look for $u \colon \Omega \times [0,T] \to \mathbb{R}$ such that 
\begin{align}
 \rho(\mathbf{x})\partial_{tt} u(\mathbf{x},t)-\kappa(\mathbf{x})\Delta u(\mathbf{x},t) &=f(\mathbf{x},t) & &\text{ in } \Omega \times (0,T], \label{eq: wave_equation} \\
 u(\mathbf{x},t)&=0 & &\text{ on } \partial \Omega \times (0,T], \nonumber\\
 u(\mathbf{x},0)&=u_0(\mathbf{x}) & &\text{ in } \Omega,  \nonumber\\
 \partial_t u(\mathbf{x},0)&=v_0(\mathbf{x}) & &\text{ in } \Omega, \nonumber
\end{align}
where $u_0$ and $v_0$ are some initial displacement and velocity, respectively, $\rho$ and $\kappa$ are some positive-valued coefficient functions and we prescribe homogeneous Dirichlet boundary conditions for simplicity. In a standard Galerkin discretization, we look for an approximation $u_h(.,t)$ of $u(.,t)$ in a finite dimensional subspace $V_h$ and test against all functions in $V_h$, which leads to solving the semi-discrete problem (see for instance \citep{hughes2012finite, quarteroni2009numerical})
\begin{align}
\label{eq: semi_discrete_pb}
\begin{split}
M\ddot{\mathbf{u}}(t) + K\mathbf{u}(t) &= \mathbf{f}(t) \qquad \text{for } t \in [0,T], \\
\mathbf{u}(0) &= \mathbf{u}_0,\\
\dot{\mathbf{u}}(0) &= \mathbf{v}_0.
\end{split}
\end{align}
where $K, M \in \mathbb{R}^{n \times n}$ are the stiffness and mass matrices, respectively. The time-dependent right-hand side vector $\mathbf{f}(t) \in \mathbb{R}^{n}$ accounts for the function $f$ and potential non-homogeneous Neumann and Dirichlet boundary conditions. Finally, $\mathbf{u}(t) \in \mathbb{R}^{n}$ is the coefficient vector of the approximate solution $u_h(\mathbf{x}, t)$ in a basis of $V_h$. Isogeometric analysis consists in choosing spline functions from computer-aided-design (CAD) such as B-splines both for representing the approximate solution and describing the geometry \citep{hughes2005isogeometric, cottrell2009isogeometric}. Such functions follow a standardized construction in a so-called parametric domain $\hat{\Omega}=(0,1)^d$ before being defined in the physical domain $\Omega$. In dimension $d=1$, the B-spline basis $\{\hat{B}_i\}_{i=1}^n$ is constructed recursively from a \emph{knot vector} $\Xi:=(\xi_1,\dots,\xi_{n+p+1})$, which is a sequence of non-decreasing real numbers. The integers $p$ and $n$ denote the spline degree and spline space dimension, respectively. A knot vector is called \emph{open} if
\begin{equation*}
    \xi_1=\dots=\xi_{p+1} < \xi_{p+2} \leq \dots \leq \xi_n < \xi_{n+1} = \dots = \xi_{n+p+1}.
\end{equation*}
Internal knots of multiplicity $1 \leq m \leq p$ give rise to $C^k$ continuous spline spaces, denoted $\mathcal{S}^{k}_{p,\Xi}$, where $k=p-m$. Our work primarily focuses on (but is not restricted to) maximally smooth $C^{p-1}$ spaces obtained when the multiplicity of each internal knot is $1$; i.e. the so-called isogeometric $k$-method. In dimension $d \geq 2$, the spline space is defined as a tensor product of univariate spaces, which all follow a similar construction. The degree, space dimension and continuity along each direction are collected in the vectors $\mathbf{p}=(p_1,p_2,\dots,p_d)$, $\mathbf{n}=(n_1,n_2,\dots,n_d)$ and $\mathbf{k}=(k_1,k_2,\dots,k_d)$, respectively, and we denote the resulting spline space $\mathcal{S}^{\mathbf{k}}_{\mathbf{p},\mathbf{\Xi}}$ (where the dependency on the knot vectors $\Xi_1,\dots,\Xi_d$ is specified by $\mathbf{\Xi}$). In dimension $d \geq 2$, it becomes convenient to label basis functions with multi-indices $\mathbf{i}=(i_1,i_2,\dots,i_d)$, which are often identified with ``linear'' indices in the global numbering. This identification permits a slight abuse of notation when writing 
\begin{equation*}
\hat{B}_i=\hat{B}_{\mathbf{i}}=\hat{B}_{1 i_1}\hat{B}_{2 i_2}\dots \hat{B}_{d i_d}
\end{equation*}
where $\hat{B}_{lj}$ denotes the $j$th function in the $l$th direction and $1 \leq i \leq n:=\prod_{l=1}^d n_l$ is a global index which only depends on $\mathbf{i}$ and $\mathbf{n}$. In the isogeometric paradigm, these functions also describe the geometry via the spline parametrization $F \colon \hat{\Omega} \to \Omega$, which maps the parametric domain to the physical domain. Geometries described by such a map are called \emph{single-patch}. The basis functions over the physical domain are then defined as $B_i = \hat{B}_i \circ F^{-1}$ and the spline spaces over the parametric and physical domains are
\begin{equation*}
    \hat{V}_h = \Span\{\hat{B}_{\mathbf{i}} \colon \mathbf{1} \leq \mathbf{i} \leq \mathbf{n}\} \quad \text{and} \quad V_h = \Span\{B_{\mathbf{i}} \colon \mathbf{1} \leq \mathbf{i} \leq \mathbf{n}\},
\end{equation*}
respectively, where $\mathbf{1}$ is the vector of all ones and vector inequalities are understood componentwise. For single-patch geometries, the entries of the stiffness and mass matrices are
\begin{equation}
\label{eq: matrix_entries}
     K_{ij}=\int_{\hat{\Omega}} (\nabla \hat{B}_i(\hat{\mathbf{x}}))^T G(\hat{\mathbf{x}}) \nabla \hat{B}_j(\hat{\mathbf{x}}) \quad \text{and} \quad M_{ij}=\int_{\hat{\Omega}} c(\hat{\mathbf{x}}) \hat{B}_i(\hat{\mathbf{x}})\hat{B}_j(\hat{\mathbf{x}}) \quad 1 \leq i,j \leq n
\end{equation}
where $G(\hat{\mathbf{x}}):=\kappa(F(\hat{\mathbf{x}}))|\det(J_F)|(J_F^TJ_F)^{-1}$, $c(\hat{\mathbf{x}}):=\rho(F(\hat{\mathbf{x}}))|\det(J_F)|$ and $J_F=J_F(\hat{\mathbf{x}})$ denotes the Jacobian matrix of $F$. As with any other standard Galerkin method, $K$ and $M$ are both symmetric and while $M$ is positive definite, $K$ is generally only positive semidefinite (unless Dirichlet boundary conditions are prescribed on some portion of the boundary). In isogeometric analysis, $M$ is additionally nonnegative owing to the pointwise nonnegativity of the B-spline basis functions.

NURBS functions enable the exact representation of a broader class of geometries (including conic sections) and lead to similar expressions and properties for the system matrices. However, the range of geometries they may describe is still far too limited for most industrial applications. For complex geometries, it may be necessary to divide the physical domain into $N_p$ subdomains (or patches); i.e.
\begin{equation*}
    \Omega = \bigcup_{r=1}^{N_p} \Omega_r
\end{equation*}
where each subdomain (or patch) $\Omega_r$ is described by its own map $F_r \colon \hat{\Omega} \to \Omega_r$. Thus, a \emph{multipatch} geometry is just a collection of patches. The construction of spline spaces over multipatch geometries is rather straightforward, though the notation becomes more cumbersome due to prolifying indices. Patches in isogeometric analysis are analogous to elements in classical finite element discretizations. Therefore, the assembly of the stiffness and mass matrices for isogeometric multipatch discretizations is analogously expressed as
\begin{equation*}
    K=\sum_{r=1}^{N_p} R_r^TK_rR_r \quad \text{and} \quad M=\sum_{r=1}^{N_p} R_r^TM_rR_r
\end{equation*}
where $K_r$ and $M_r$ are the local stiffness and mass matrices of the $r$th patch and $R_r$ maps its local degrees of freedom to global ones.

Despite the additional flexibility of multipatch geometries, they still fall short in describing the highly complex shapes of industrial CAD models \cite{leidinger2019explicit}. Such models commonly consist in multiple \emph{trimmed} NURBS patches. Trimming is a Boolean operation whereby parts of a geometry are joined, intersected or simply discarded. While these operations change the visualization of the model, they do not change its mathematical description. The analysis on trimmed geometries is particularly challenging for a variety of reasons, including integration on trimmed boundaries, imposition of essential boundary conditions, stability and conditioning issues \citep{marussig2018review, de2023stability}. Trimming also alters the structure of systems matrices, which heavily impacts the design of assembly algorithms and preconditioning techniques. As a matter of fact, many developments in isogeometric analysis that rely on a tensor product structure are not applicable to trimmed geometries.

For discretizing \eqref{eq: semi_discrete_pb} in time, explicit methods are usually preferred for fast dynamic processes such as blasts or impacts due to the physical restriction on the step size. Scores of methods have been proposed in the literature, including the central difference, Wilson-$\theta$ and generalized $\alpha$ methods to name just a few. Most of them are commonly included in textbooks \citep{hughes2012finite, bathe2006finite}, which the reader may consult for details. The critical time step depends on the method (see e.g. \eqref{eq: CFL_central_difference} for the central difference method in the undamped case) but in the undamped case all explicit methods require solving a linear system with the mass matrix at least once in each iteration. Although very efficient preconditioning strategies have been devised \citep{gao2014fast, wozniak2017parallel, chan2018multi, loli2022easy}, practitioners still mostly resort to mass lumping, which we describe in the next section.

\section{Mass lumping}
\label{se: mass_lumping}
\subsection{Row-sum mass lumping and its generalization}
\label{se: structured_mass_lumping}
As for preconditioning matrices, lumped mass matrices are easily invertible approximations of the consistent mass. However, contrary to preconditioning, mass lumping consists in directly substituting the consistent mass with its approximation and therefore, it should not only provide a sufficiently accurate approximation but also guarantee an improvement of the CFL condition. Several techniques have been developed but despite the apparent downgrade in convergence rate, the row-sum technique remains very popular owing to its simplicity and straightforward implementation. We define it through the application of a lumping operator \cite[][Definition 3.7]{voet2023mathematical}.

\begin{definition}[Lumping operator]
\label{def: lumping}
Let $B \in \mathbb{R}^{n \times n}$. The lumping operator $\mathcal{L} \colon \mathbb{R}^{n \times n} \to \mathbb{R}^{n \times n}$ is defined as
\begin{equation*}
\mathcal{L}(B)=\diag(d_1,\dots,d_n)
\end{equation*}
where $d_i=\sum_{j=1}^n |b_{ij}|$ for $i=1,\dots,n$.
\end{definition}

\begin{remark}
Note that the absolute values in the above definition are unnecessary for isogeometric analysis due to the positivity of the B-spline basis. However, in the case of a more general basis, they ensure that $\mathcal{L}(M) \succeq M$ and subsequently guarantee an improvement of the CFL \cite[][Corollary 3.10]{voet2023mathematical}. This result generally does not hold without them, as is well-known in the finite element community. Of course, the accuracy of lumping with absolute values when the basis is allowed to be negative is a whole separate issue and should be studied on a case-by-case basis.
\end{remark}

Since most mass lumping/scaling methods are defined algebraically as modifications to the consistent mass matrix, it is convenient to introduce an order relation between symmetric matrices.

\begin{definition}[Loewner partial order]
For two symmetric matrices $A,B \in \mathbb{R}^{n \times n}$, we write $A \succeq B$ (respectively $A \succ B$) if $A-B$ is positive semidefinite (respectively positive definite).
\end{definition}

The Loewner partial order is a natural choice in our context since it is the matrix equivalent of bounding bilinear forms. Indeed, if $M$ and $\tilde{M}$ are the Gram matrices for the symmetric bilinear forms $b,\tilde{b} \colon V_h \times V_h \to \mathbb{R}$ in some basis of $V_h$, then
\begin{equation*}
    \tilde{b}(u_h,u_h) \geq b(u_h,u_h) \quad \forall u_h \in V_h \iff \tilde{M} \succeq M.
\end{equation*}
When the construction of $\tilde{M}$ is completely algebraic, we rarely have an explicit representation of $\tilde{b}(u_h,u_h)$, but we might still bound it by understanding the relation between $\tilde{M}$ and $M$ in the Loewner ordering. The argument has already been used to good effect in \cite{voet2023mathematical} where it was shown that for two symmetric positive definite matrices $A,B \in \mathbb{R}^{n \times n}$, the generalized eigenvalues of $(A,\mathcal{L}(B))$ are always smaller or equal to those of $(A,B)$, which is a consequence of the fact that $\mathcal{L}(B) \succeq B$ \cite[][Corollary 3.10]{voet2023mathematical}. The authors in \cite{voet2023mathematical} used eigenvalue bounds to prove the result but there exist multiple short and elegant ways of reaching the same conclusion. For instance, if $B$ is nonnegative and one thinks of it as a weighted adjacency matrix, then $\mathcal{L}(B)$ is the degree matrix and $\mathcal{L}(B) - B$ is its graph Laplacian, which plays a prominent role in spectral clustering techniques \cite{von2007tutorial}. A direct computation then shows that $\mathbf{x}^T(\mathcal{L}(B) - B)\mathbf{x}=\frac{1}{2} \sum_{i,j=1}^n (x_i-x_j)^2 b_{ij} \geq 0$ for any $\mathbf{x} \in \mathbb{R}^n$. This technique will be employed later in our paper.

Unfortunately, in the context of isogeometric analysis, the row-sum technique performs poorly in higher dimensions, even for moderate spline degrees. Therefore, in \cite{voet2023mathematical}, the authors considered increasing the bandwidth in order to improve the accuracy, while still underestimating the generalized eigenvalues of $(A,B)$. It led to defining a finite partially ordered sequence of banded matrices with increasing bandwidth
\begin{equation*}
    \mathcal{L}(B)=P_1 \succeq P_2 \succeq \dots \succeq P_{n-1} \succeq P_n=B
\end{equation*}
where $P_1$ coincides with the usual row-sum lumped mass matrix, $P_2$ is tridiagonal, $P_3$ is pentadiagonal and so forth. This strategy was then generalized to higher dimensions for Kronecker product matrices by lumping the factor matrices defining the Kronecker product.

\subsection{Block mass lumping}
In general, for nontrivial geometries and coefficients, the mass matrix cannot be expressed as a Kronecker product and the techniques described in \citep{deng2021boundary, manni2022application, hiemstra2021removal} are not straightforwardly applicable. In \cite{voet2023mathematical}, the authors suggested computing the nearest Kronecker product approximation and substituting it to the consistent mass. However, the loss of accuracy induced by the approximation depends on the singular value decay of a reordered matrix and is independent of the discretization parameters. Not only may the approximation be rather crude but combining it with mass lumping also does not give a theoretical guarantee of improving the CFL, although it was observed in all practically relevant cases. The primary objective of this section is to generalize the mass lumping techniques in \cite{voet2023mathematical} to complex geometries and address the aforementioned issues. These techniques guarantee an improvement of the CFL condition and their implementation is non-intrusive. Although they do not improve the convergence rate, they might still significantly improve the accuracy and remain relevant, especially at a time when most practical simulations employ quadratic or cubic discretizations. The key observation for achieving these results is that, although the mass matrix generally cannot be expressed as a Kronecker product, it still inherits some favorable structure from the tensor product basis functions which we will exploit for designing algebraic mass lumping techniques. We will first describe this structure in detail for the single-patch case and then propose a block generalization of the methods presented in \cite{voet2023mathematical}. From there, the multipatch case naturally follows.

Thanks to the local support property of the basis functions commonly used in isogeometric analysis (e.g. B-splines and NURBS), maximally smooth discretizations of 1D problems lead to banded matrices, where the bandwidth is the spline degree $p$ \cite{hofreither2018black}. For higher dimensional problems, the tensor product construction of the basis functions leads to hierarchical banded matrices that were defined inductively in \cite{hofreither2018black}. Without loss of generality, the framework introduced in this section assumes a lexicographical type labeling of the degrees of freedom, which may always be recovered after a suitable reordering of the system matrices.

\begin{definition}[$1$-level banded matrix]
\label{def: 1_level_banded}    
A matrix $B \in \mathbb{R}^{n \times n}$ is called \emph{1-level banded} (or simply \emph{banded}) with bandwidth $b$ if
\begin{equation*}
    |i-j|>b \implies b_{ij}=0 \quad i,j=1,\dots,n.
\end{equation*}
\end{definition}

\begin{definition}[$d$-level banded matrix]
\label{def: d_level_banded}    
A block matrix $\mathcal{B} \in \mathbb{R}^{n_1n_2\dots n_d \times n_1n_2\dots n_d}$ partitioned as
\begin{equation*}
\mathcal{B}
=
\begin{pmatrix}
    B_{1,1} & \cdots & B_{1,n_1} \\
    \vdots & \ddots & \vdots \\
    B_{n_1,1} & \cdots & B_{n_1,n_1}
\end{pmatrix}
\end{equation*}
is called \emph{d-level banded} with bandwidths $(b_1,b_2,\dots,b_d)$ if each block $B_{i,j} \in \mathbb{R}^{n_2\dots n_d \times n_2\dots n_d}$ is $(d-1)$-level banded with bandwidths $(b_2,\dots,b_d)$ and
\begin{equation*}
    |i-j|>b_1 \implies B_{i,j}=0 \quad i,j=1,\dots,n_1.
\end{equation*}
\end{definition}

This hierarchical notion of bandedness is related to the standard notion of bandedness through the bandwidths and block sizes. In the rest of the paper we will denote $r_k=\prod_{j=k+1}^d n_j$ the size of the blocks on the $k$th hierarchical level (with $k<d$). On the finest level, the ``blocks'' reduce to scalars and we set $r_d=1$.

\begin{lemma}
\label{lem: bandwidth}
Let $\mathcal{B} \in \mathbb{R}^{n_1n_2\dots n_d \times n_1n_2\dots n_d}$ be $d$-level banded with bandwidths $\mathbf{b}=(b_1,b_2,\dots,b_d)$ and block sizes $\mathbf{r}=(r_1,r_2,\dots,r_d)$. Then $\mathcal{B}$ has bandwidth
\begin{equation}
\label{eq: bandedness}
    \mathbf{b} \cdot \mathbf{r} = \sum_{i=1}^d b_ir_i.
\end{equation}
\end{lemma}
\begin{proof}
The proof is by induction on the dimension $d$. For $d=1$, $\mathcal{B}$ is $1$-level banded (i.e. banded) with bandwidth $b_d$ and the property \eqref{eq: bandedness} holds since $r_d=1$. We now verify the property for dimension $d$ assuming it holds for dimension $d-1$. Let $\mathcal{B}$ be $d$-level banded with bandwidths $\mathbf{b}=(b_1,b_2,\dots,b_d)$. Then, by definition,
\begin{equation*}
\mathcal{B}
=
\begin{pmatrix}
    B_{1,1} & \cdots & B_{1,n_1} \\
    \vdots & \ddots & \vdots \\
    B_{n_1,1} & \cdots & B_{n_1,n_1}
\end{pmatrix}
\end{equation*}
and each block $B_{i,j}$ is $(d-1)$-level banded with bandwidths $(b_2,\dots,b_d)$ and $B_{i,j}=0$ if $|i-j|>b_1$. Thus, the bandwidth of $\mathcal{B}$ is given by the sum of $b_1r_1$ and the bandwidth of $B_{i,j}$. Since $B_{i,j}$ is $(d-1)$-level banded, the induction hypothesis completes the proof.
\end{proof}

Definitions \ref{def: 1_level_banded} and \ref{def: d_level_banded} only provide a description of the sparsity. In order to build a working framework, they must be complemented with Definitions \ref{def: 1_SPSD_SPD_partition} and \ref{def: d_SPSD_SPD_partition} that describe both symmetry and spectral properties. 

\begin{definition}
\label{def: 1_SPSD_SPD_partition}
The sets of symmetric positive semidefinite (SPSD) and symmetric positive definite (SPD) matrices of size $n$ are defined, respectively, as
\begin{equation*}
    \mathcal{S}_n=\{B \in \mathbb{R}^{n \times n} \colon B=B^T, B \succeq 0\} \quad \text{and} \quad \mathcal{S}_n^+=\{B \in \mathbb{R}^{n \times n} \colon B=B^T, B \succ 0\}.
\end{equation*}
\end{definition}
The next definition provides a hierarchical generalization.
\begin{definition}
\label{def: d_SPSD_SPD_partition}
For block matrices $\mathcal{B} \in \mathbb{R}^{n_1n_2\dots n_d \times n_1n_2\dots n_d}$ partitioned as
\begin{equation*}
\mathcal{B}
=
\begin{pmatrix}
    B_{1,1} & \cdots & B_{1,n_1} \\
    \vdots & \ddots & \vdots \\
    B_{n_1,1} & \cdots & B_{n_1,n_1}
\end{pmatrix}
\end{equation*}
where $B_{i,j} \in \mathbb{R}^{n_2\dots n_d \times n_2\dots n_d}$, we define the sets
\begin{align*}
    &\mathcal{S}_{(n_1,n_2,\dots,n_d)}=\{\mathcal{B} \in \mathcal{S}_{n} \colon B_{i,j} \in \mathcal{S}_{(n_2,\dots,n_d)}\}, \\
    &\mathcal{S}_{(n_1,n_2,\dots,n_d)}^+=\{\mathcal{B} \in \mathcal{S}_{n}^+ \colon B_{i,j} \in \mathcal{S}_{(n_2,\dots,n_d)}\},
\end{align*}
where $n=\prod_{i=1}^d n_i$.
\end{definition}
Given a vector $\mathbf{n}=(n_1,n_2,\dots,n_d)$, we will often denote the corresponding sets $\mathcal{S}_\mathbf{n}$ and $\mathcal{S}_\mathbf{n}^+$, respectively. Clearly, if $\mathcal{B} \in \mathcal{S}_\mathbf{n}$ (or $\mathcal{S}_\mathbf{n}^+$) is $d$-level banded with bandwidths $\mathbf{b}=(b_1,b_2,\dots,b_d)$, then $\mathbf{b} \leq \mathbf{n}-\mathbf{1}$ componentwise, where $\mathbf{1}$ is the vector of all ones.
\begin{lemma}
For any vector $\mathbf{n}=(n_1,n_2,\dots,n_d) \in \mathbb{N}^d$,
\begin{align*}
    &\mathcal{S}_{(n_1,n_2,\dots,n_d)} = \mathcal{S}_{(n_1,n_2,\dots,n_{d-1},r_{d-1})} 
 \subseteq \mathcal{S}_{(n_1,n_2,\dots,n_{d-2},r_{d-2})} \subseteq \dots \subseteq \mathcal{S}_{(n_1, r_1)} \subseteq \mathcal{S}_{n}, \\
    &\mathcal{S}_{(n_1,n_2,\dots,n_d)}^+ = \mathcal{S}_{(n_1,n_2,\dots,n_{d-1},r_{d-1})}^+
 \subseteq \mathcal{S}_{(n_1,n_2,\dots,n_{d-2},r_{d-2})}^+ \subseteq \dots \subseteq \mathcal{S}_{(n_1, r_1)}^+ \subseteq \mathcal{S}_{n}^+.
\end{align*}
\end{lemma}
\begin{proof}
We prove the statements from left to right. Let $\mathcal{B} \in \mathcal{S}_{(n_1,n_2,\dots,n_d)}$. By definition, on the $(d-1)$th hierarchical level the matrices are in $\mathcal{S}_{n_d}=\mathcal{S}_{r_{d-1}}$ and the first equality trivially follows. On the $(d-2)$th level the matrices are in $\mathcal{S}_{(n_{d-1}, r_{d-1})} \subseteq \mathcal{S}_{r_{d-2}}$. Thus, $\mathcal{B} \in \mathcal{S}_{(n_1,n_2,\dots,n_{d-2},r_{d-2})}$. We then repeatedly apply the same argument by moving up the hierarchical structure and realizing that at level $k$ (with $0 \leq k \leq d-2$) the matrices are in $\mathcal{S}_{(n_{k+1}, r_{k+1})} \subseteq \mathcal{S}_{r_k}$. The proof of the second statement is completely analogous.
\end{proof}

Thus, $d$-level banded matrices in $\mathcal{S}_{\mathbf{n}}$ (or $\mathcal{S}_{\mathbf{n}}^+$) have a hierarchical block banded structure with SPSD blocks. The following lemma shows that the isogeometric mass matrix falls in this category.

\begin{lemma}
\label{lem: mass_properties}
Let $\mathcal{M} \in \mathbb{R}^{n_1 \dots n_d \times n_1 \dots n_d}$ be a $d$-dimensional isogeometric single-patch mass matrix with associated dimensions vector $\mathbf{n}=(n_1,n_2,\dots,n_d)$. Then 
\begin{enumerate}[noitemsep]
    \item $\mathcal{M}$ is $d$-level banded,
    \item $\mathcal{M} \in \mathcal{S}_{\mathbf{n}}^+$.
\end{enumerate}
\end{lemma}
\begin{proof}
We prove the two statements below. 
\begin{enumerate}[noitemsep]
    \item The key observation is noticing that the mass matrices $\mathcal{M}$ and $\hat{\mathcal{M}}$ in the physical and parametric domains, respectively, have the same sparsity pattern and therefore the same hierarchical block bandedness. Indeed, their entries are defined as (see \eqref{eq: matrix_entries})
    \begin{equation*}
        \mathcal{M}_{ij}=\int_{\hat{\Omega}} c(\hat{\mathbf{x}}) \hat{B}_i(\hat{\mathbf{x}})\hat{B}_j(\hat{\mathbf{x}}) \quad \text{and} \quad \hat{\mathcal{M}}_{ij}=\int_{\hat{\Omega}} \hat{B}_i(\hat{\mathbf{x}})\hat{B}_j(\hat{\mathbf{x}}).
    \end{equation*}
    From the positivity of the B-spline (or NURBS) basis and the fact that $c(\hat{\mathbf{x}}):=\rho(F(\hat{\mathbf{x}}))|\det(J_F(\hat{\mathbf{x}}))| > 0$, it follows that $\mathcal{M}_{ij}=0 \iff \hat{\mathcal{M}}_{ij}=0$. Thus, $\mathcal{M}$ and $\hat{\mathcal{M}}$ have the same sparsity pattern. Moreover, since $\hat{\mathcal{M}}$ is the mass matrix in the parametric domain
    \begin{equation*}
        \hat{\mathcal{M}} = \bigotimes_{i=1}^d \hat{M}_i
    \end{equation*}
    where $\hat{M}_i \in \mathbb{R}^{n_i \times n_i}$ is banded with bandwidth $b_i$ for $i=1,\dots,d$. Thus, by definition, $\hat{\mathcal{M}}$ is $d$-level banded with bandwidths $(b_1,b_2,\dots,b_d)$ and consequently so is $\mathcal{M}$.
    \item We start at the top of the hierarchical structure and work our way downward. Firstly, since the mass matrix is symmetric positive definite $\mathcal{M} \in \mathcal{S}_{n}^+$. Secondly, it may be written as 
    \begin{equation}
    \label{eq: block_structure}
    \mathcal{M}
    =
    \begin{pmatrix}
        M_{1,1} & \cdots & M_{1,n_1} \\
        \vdots & \ddots & \vdots \\
        M_{n_1,1} & \cdots & M_{n_1,n_1}
    \end{pmatrix}
    \end{equation}
    where $M_{i,j} \in \mathbb{R}^{r_1 \times r_1}$. We will show that $M_{i,j} \in \mathcal{S}_{r_1}$ for all $i,j=1,\dots n_1$. Let $\hat{\mathbf{B}}_l \in \mathbb{R}^{n_l}$ denote the vector of basis functions $\{\hat{B}_{li}\}_{i=1}^{n_l}$ along the $l$th direction in the parametric domain. The matrix $M_{i,j}$ is then given by
    \begin{equation}
    \label{eq: M_{ij}_symmetry}
        M_{i,j}=\int_{\hat{\Omega}} \underbrace{c(\hat{\mathbf{x}}) \hat{B}_{1i}\hat{B}_{1j}}_{g_{ij}} \bigotimes_{l=2}^d \hat{\mathbf{B}}_l\hat{\mathbf{B}}_l^T = \int_{\hat{\Omega}} g_{ij} \bigotimes_{l=2}^d \hat{\mathbf{B}}_l\hat{\mathbf{B}}_l^T.
    \end{equation}
    From \eqref{eq: M_{ij}_symmetry}, $M_{i,j}$ is evidently symmetric. Moreover, thanks to the pointwise nonnegativity of the basis functions, $g_{ij} \geq 0$ and consequently, for any vector $\mathbf{x} \in \mathbb{R}^{r_1}$,
    \begin{equation}
    \label{eq: M_{ij}_positive_semidefiniteness}
        \mathbf{x}^TM_{i,j}\mathbf{x}=\int_{\hat{\Omega}} g_{ij} (\mathbf{x}^T\bigotimes_{l=2}^d \hat{\mathbf{B}}_l)^2 = \int_{\hat{\Omega}} g_{ij} v^2 \geq 0
    \end{equation}
    where $v=\mathbf{x}^T\bigotimes_{l=2}^d \hat{\mathbf{B}}_l$ is a function in a finite element subspace. Thus, \eqref{eq: M_{ij}_symmetry} and \eqref{eq: M_{ij}_positive_semidefiniteness} together show that $M_{i,j} \in \mathcal{S}_{r_1}$. Finally, since $\mathcal{M} \in \mathcal{S}_{n}^+$ and $M_{i,j} \in \mathcal{S}_{r_1}$ for all $i,j=1,\dots n_1$, then $\mathcal{M} \in \mathcal{S}_{(n_1,r_1)}^+$ by definition. We now repeat the same argument by first showing that each $M_{i,j}$ can itself be expressed as a block matrix similarly to \eqref{eq: block_structure} with blocks of size $r_2 \times r_2$. By repeating the arguments in \eqref{eq: M_{ij}_symmetry} and \eqref{eq: M_{ij}_positive_semidefiniteness} one easily shows that each of these blocks is in $\mathcal{S}_{r_2}$ and consequently $M_{i,j} \in \mathcal{S}_{(n_2,r_2)}$. Finally, moving up one level, we deduce that $\mathcal{M} \in \mathcal{S}_{(n_1,n_2,r_2)}^+$. By recursively applying the same arguments, we finally prove that $\mathcal{M} \in \mathcal{S}_{\mathbf{n}}^+$.
\end{enumerate}
\end{proof}

\begin{remark}
For single-patch isogeometric discretizations, the dimensions vector $\mathbf{n}$ corresponds to the number of basis functions along each parametric direction and is always uniquely determined by the number of subdivisions, order, smoothness and boundary conditions. Moreover, for maximally smooth discretizations, the bandwidths are equal to the spline degrees such that $\mathbf{b}=\mathbf{p}$.   
\end{remark}

The definitions above also accommodate vector-valued PDEs such as linear elasticity. In this context, the mass matrix is a $(d+1)$-level banded matrix of bandwidths $(0,b_1,\dots,b_d)$ (i.e. a block diagonal matrix). Moreover, if each component of the solution is discretized using the same scalar spaces, then $\mathcal{M} \in \mathcal{S}_{(d, \mathbf{n})}^+$, where $\mathbf{n}$ is the dimensions vector for a scalar problem.

We stress that Lemma \ref{lem: mass_properties} is a sole consequence of the tensor product construction of the basis functions and their pointwise nonnegativity and does not depend on the geometry mapping. In particular, it shows that the isogeometric mass matrix is \emph{not only} symmetric positive definite, but actually enjoys additional structure. For instance $\mathcal{S}_{(n_1,n_2)}^+$, typically encountered for $2$-dimensional discretizations, is the set of SPD block matrices with SPSD blocks. This structure is key to extending mass lumping techniques to nontrivial problems in dimension $d \geq 2$. We now define the block analogue of the lumping operator introduced in \cite{voet2023mathematical}.

\begin{definition}[Block lumping operator]
\label{def: block_lumping}
Let $\mathcal{B} \in \mathbb{R}^{n_1n_2 \times n_1n_2}$ be a block matrix partitioned as
\begin{equation*}
\mathcal{B}
=
\begin{pmatrix}
    B_{1,1} & \cdots & B_{1,n_1} \\
    \vdots & \ddots & \vdots \\
    B_{n_1,1} & \cdots & B_{n_1,n_1}
\end{pmatrix}
\end{equation*}
where $B_{i,j} \in \mathbb{R}^{n_2 \times n_2}$. The block lumping operator $\mathcal{L}$ is defined as
\begin{equation*}
\mathcal{L}(\mathcal{B})=\diag(D_1,\dots,D_{n_1}):=
    \begin{pmatrix}
        D_1 & \cdots & 0 \\
        \vdots & \ddots & \vdots \\
        0 & \cdots & D_{n_1}
    \end{pmatrix}
\end{equation*}
where $D_i=\sum_{j=1}^{n_1} B_{i,j}$ for $i=1,\dots,n_1$.
\end{definition}

Whereas the (scalar) lumping operator in Definition \ref{def: lumping} returns a diagonal matrix, the block lumping operator returns a block diagonal matrix. We establish some useful consequences of this definition for the sets $\mathcal{S}_{\mathbf{n}}$ and $\mathcal{S}_{\mathbf{n}}^+$.

\begin{lemma}
\label{lem: well_posedness}
For any vector $\mathbf{n}=(n_1,n_2,\dots,n_d) \in \mathbb{N}^d$,
\begin{equation*}
    \mathcal{L}(\mathcal{S}_{\mathbf{n}}) \subseteq \mathcal{S}_{\mathbf{n}} \quad \text{and} \quad \mathcal{L}(\mathcal{S}_{\mathbf{n}}^+) \subseteq \mathcal{S}_{\mathbf{n}}^+.
\end{equation*}
\end{lemma}
\begin{proof}
Let $\mathcal{B} \in \mathcal{S}_{\mathbf{n}}$. The result for $d=1$ is obvious from Definition \ref{def: lumping}. Now assume that $d \geq 2$ and let $\mathcal{L}(\mathcal{B})$ be constructed following Definition \ref{def: block_lumping}. The proof simply follows from the stability of $\mathcal{S}_{\mathbf{n}}$ under addition: since $B_{i,j} \in \mathcal{S}_{(n_2,\dots,n_d)}$ for all $i,j=1,\dots n_1$, then 
\begin{equation*}
    D_i=\sum_{j=1}^{n_1} B_{i,j} \in \mathcal{S}_{(n_2,\dots,n_d)}.
\end{equation*}
Since $\mathcal{L}(\mathcal{B})$ is block diagonal with SPSD blocks, it is itself SPSD and $\mathcal{L}(\mathcal{B}) \in \mathcal{S}_{\mathbf{n}}$. The proof of the second statement is completely analogous (noting that for matrices $\mathcal{B} \in \mathcal{S}_{\mathbf{n}}^+$ all diagonal blocks and diagonal sub-blocks down the hierarchy are positive definite).
\end{proof}

The next lemma is the block generalization of \citep[][Lemma 3.9]{voet2023mathematical}.

\begin{lemma}
\label{lem: block_positivity}
Let $\mathcal{B} \in \mathcal{S}_{\mathbf{n}}^+$ with $\mathbf{n}=(n_1,\dots,n_d) \in \mathbb{N}^d$ and $d \geq 2$. Then,
\begin{equation*}
    \mathcal{L}(\mathcal{B}) \succeq \mathcal{B}.
\end{equation*}
\end{lemma}
\begin{proof}
Let $\mathcal{B} \in \mathcal{S}_{\mathbf{n}}^+$,
\begin{equation*}
    \mathcal{B}
    =
    \begin{pmatrix}
        B_{1,1} & \cdots & B_{1,n_1} \\
        \vdots & \ddots & \vdots \\
        B_{n_1,1} & \cdots & B_{n_1,n_1}
    \end{pmatrix},
    \quad \mathbf{x}
    =\begin{pmatrix}
        \mathbf{x}_1 \\
        \vdots \\
        \mathbf{x}_{n_1}
    \end{pmatrix}.
\end{equation*}
Then, using the fact that $B_{i,j}=B_{j,i}^T=B_{j,i}$ and $B_{i,j} \succeq 0$ for all $i,j=1,\dots,n_1$,
\begin{align*}
    &\mathbf{x}^T(\mathcal{L}(\mathcal{B})-\mathcal{B})\mathbf{x} \\
    &= \sum_{i=1}^{n_1} \mathbf{x}_i^T \left(\sum_{j=1}^{n_1} B_{i,j}\right)\mathbf{x}_i-\sum_{i,j=1}^{n_1} \mathbf{x}_{i}^TB_{i,j}\mathbf{x}_j \\
    &=\frac{1}{2}\left(\sum_{i=1}^{n_1} \mathbf{x}_i^T \left(\sum_{j=1}^{n_1} B_{i,j}\right)\mathbf{x}_i-2\sum_{i,j=1}^{n_1} \mathbf{x}_{i}^TB_{i,j}\mathbf{x}_j+\sum_{j=1}^{n_1} \mathbf{x}_j^T \left(\sum_{i=1}^{n_1} B_{j,i}\right)\mathbf{x}_j\right) \\
    &=\frac{1}{2}\sum_{i,j=1}^{n_1} \mathbf{x}_i^TB_{i,j}\mathbf{x}_i-2\mathbf{x}_i^TB_{i,j}\mathbf{x}_j+\mathbf{x}_j^TB_{i,j}\mathbf{x}_j \\
    &=\frac{1}{2}\sum_{i,j=1}^{n_1} (\mathbf{x}_i-\mathbf{x}_j)^TB_{i,j}(\mathbf{x}_i-\mathbf{x}_j) \geq 0,
\end{align*}
which proves that $\mathcal{L}(\mathcal{B})-\mathcal{B} \succeq 0$.
\end{proof}

\begin{remark}
\label{rem: extension}
Interestingly, Lemma \ref{lem: block_positivity} also holds for the larger set of symmetric block matrices with SPSD blocks. Moreover, denoting $\mathbf{e}$ the vector of all ones, $(1,\mathbf{e})$ is an eigenpair of $(\mathcal{B}, \mathcal{L}(\mathcal{B}))$ regardless of whether $\mathcal{B}$ is nonnegative. For nonnegative matrices and $d=1$, our results simply reduce to those of \cite{voet2023mathematical}. 
\end{remark}

In \cite{voet2023mathematical}, the authors considered the matrix splitting $B=D_i+R_i$, where $D_i$ consists of all super and sub-diagonals strictly smaller than $i$ and $R_i$ is the remainder. Lumped matrices were then defined by lumping the remainder $R_i$ and adding it to $D_i$. The block lumped matrices introduced in Definition \ref{def: block_lumped_matrices} are the block analogue of those constructed in \cite{voet2023mathematical} and feature blockwise operations instead of entrywise operations.

\begin{definition}[Block lumped matrices]
\label{def: block_lumped_matrices}
Let $\mathcal{B} \in \mathcal{S}_{\mathbf{n}}^+$ with $\mathbf{n} \in \mathbb{N}^d$ and $d \geq 2$ and consider the matrix splitting $\mathcal{B}=\mathcal{D}_i+\mathcal{R}_i$ where $\mathcal{D}_i$ consists of all super and sub block diagonals strictly smaller than $i$ and $\mathcal{R}_i$ is the remainder. We define the sequence of matrices $\mathcal{P}_i=\mathcal{D}_i+\mathcal{L}(\mathcal{R}_i)$ for $i=1,\dots,n_1$. In particular, we observe that $\mathcal{P}_1=\mathcal{L}(\mathcal{B})$ and $\mathcal{P}_{n_1}=\mathcal{B}$.
\end{definition}

By construction, $\mathcal{P}_i$ only reduces the highest hierarchical level of bandedness: if $\mathcal{B}$ is $d$-level banded with bandwidths $(b_1,b_2,\dots,b_d)$, then $\mathcal{P}_i$ (with $i \leq b_1+1$) is $d$-level banded with bandwidths $(i-1,b_2,\dots,b_d)$. Figure \ref{fig: block_splitting} shows an example for a $2$-level banded matrix $\mathcal{P}_2$ with bandwidths $(1,3)$ (block tridiagonal matrix) constructed from a $2$-level banded matrix $\mathcal{B}$ with bandwidths $(3,3)$ (block septadiagonal matrix). The next theorem provides a generalization of \citep[][Theorem 3.21]{voet2023mathematical}. From now on, we will always assume that the eigenvalues are ordered increasingly.

\begin{figure}[htbp]
    \centering
    \includegraphics[scale=0.25]{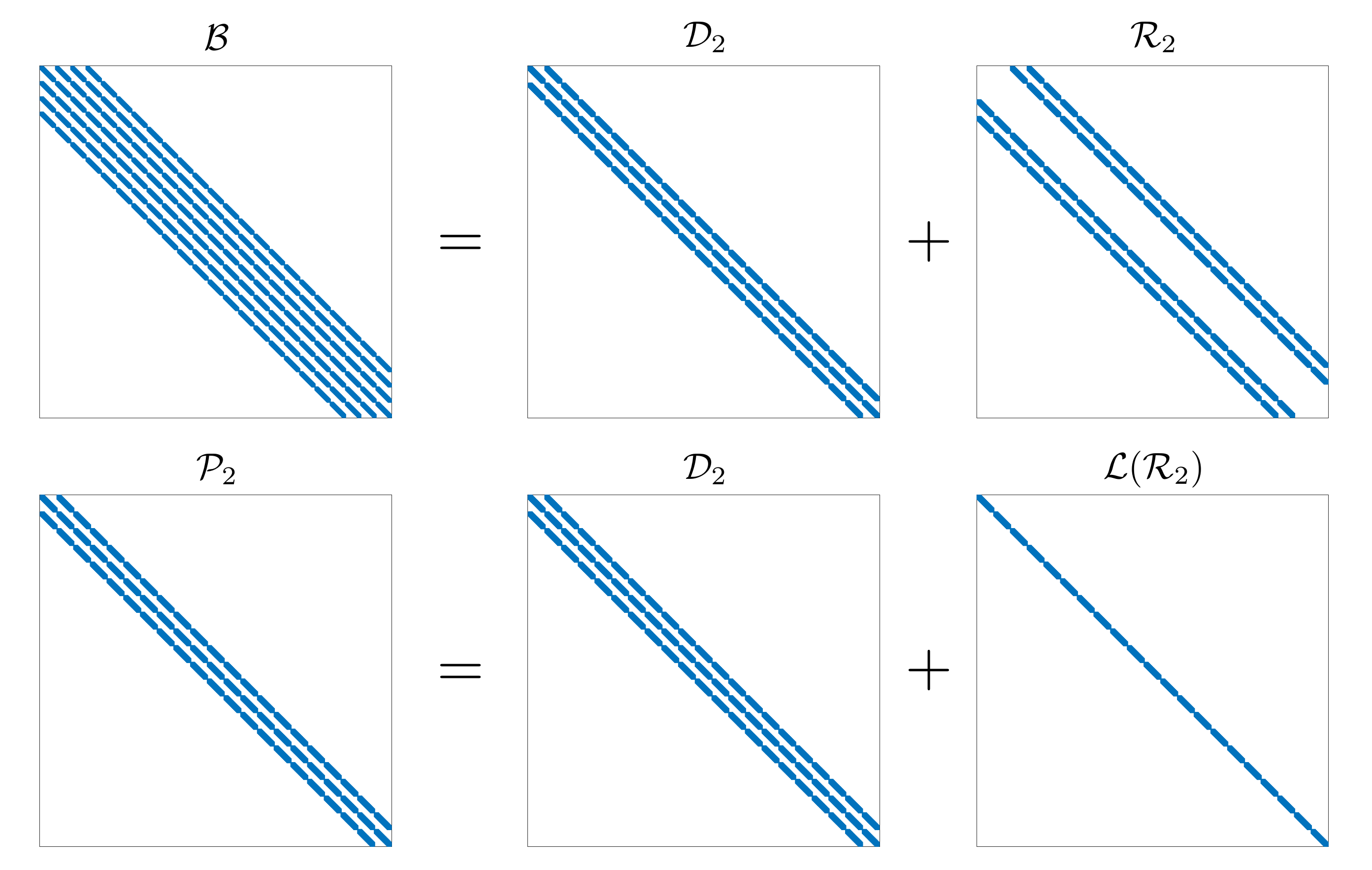}
    \caption{Block tridiagonal matrix $\mathcal{P}_2$ constructed from a block septadiagonal matrix $\mathcal{B}$}
    \label{fig: block_splitting}
\end{figure}

\begin{theorem}
\label{th: family_block_lumped_mass_prec}
Let $\mathcal{B} \in \mathcal{S}_{\mathbf{n}}^+$ with $\mathbf{n}=(n_1,\dots,n_d) \in \mathbb{N}^d$ and $d \geq 2$. Then, the sequence of matrices $\{\mathcal{P}_i\}_{i=1}^{n_1}$ constructed from $\mathcal{B}$ according to Definition \ref{def: block_lumped_matrices} satisfies the following properties:
\begin{enumerate}[noitemsep]
    \item $\Lambda(\mathcal{B}, \mathcal{P}_i) \subset (0,1]$ for all $i=1,\dots,n_1$,
    \item $\lambda_k(\mathcal{B}, \mathcal{P}_i) \leq \lambda_k(\mathcal{B}, \mathcal{P}_{i+1})$ for all $k$ and any given $i=1,\dots, n_1-1$,
    \item $\lambda_n(\mathcal{B}, \mathcal{P}_i)=1$ for all $i=1,\dots,n_1$.
\end{enumerate}
\end{theorem}
\begin{proof}
The proof is analogous to \citep[][Theorem 3.21]{voet2023mathematical} using Lemma \ref{lem: block_positivity} and Remark \ref{rem: extension}.
\end{proof}
The proof arguments of Theorem \ref{th: family_block_lumped_mass_prec} also show that the block lumped matrices satisfy
\begin{equation}
\label{eq: Loewner_sequence_single_patch}
    \mathcal{L}(\mathcal{B})=\mathcal{P}_1 \succeq \mathcal{P}_2 \succeq \dots \succeq \mathcal{P}_{n_1-1} \succeq \mathcal{P}_{n_1}=\mathcal{B}.
\end{equation}
This ordering implies that for a matrix $\mathcal{A} \in \mathcal{S}_n$,
\begin{equation*}
    \lambda_k(\mathcal{A}, \mathcal{P}_1) \leq \lambda_k(\mathcal{A}, \mathcal{P}_2) \leq \dots \leq \lambda_k(\mathcal{A}, \mathcal{P}_{n_1-1}) \leq \lambda_k(\mathcal{A}, \mathcal{P}_{n_1}) \quad 1 \leq k \leq n.
\end{equation*}
Clearly, block diagonal matrices such as $\mathcal{P}_1$ are very appealing given that linear systems can be solved in parallel for each block. The block tridiagonal case can still be treated efficiently by a block forward elimination and backward substitution algorithm; see e.g. \citep[][Section 4.5.1]{golub2013matrix} for the details.

\subsection{Hierarchical mass lumping}
\label{se: hierarchical_mass_lumping}
Mass lumping can also be applied down the hierarchical structure in many different ways. The following lemma sets the foundation and generalizes Lemma \ref{lem: well_posedness}.

\begin{lemma}
\label{lem: Pi_in_Sn+}
If $\mathcal{B} \in \mathcal{S}_{\mathbf{n}}^+$ with $\mathbf{n}=(n_1,\dots,n_d) \in \mathbb{N}^d$, then $\mathcal{P}_k \in \mathcal{S}_{\mathbf{n}}^+$ for all $k=1,\dots,n_1$.   
\end{lemma}
\begin{proof}
The case $d=1$ was already proved in \citep[][Theorem 3.21]{voet2023mathematical}. For $d \geq 2$, by definition $\mathcal{P}_k=\mathcal{D}_k+\mathcal{L}(\mathcal{R}_k)$ and is partitioned similarly to \eqref{eq: block_structure} with blocks $P_{i,j}$ for $i,j=1,\dots,n_1$. Since $\mathcal{S}_{\mathbf{n}}$ is closed under addition, $P_{i,j} \in S_{(n_2,\dots,n_d)}$ for all $i,j=1,\dots,n_1$. Moreover, $\mathcal{P}_k \succeq \mathcal{B} \succ 0$ implies that $\mathcal{P}_k \in \mathcal{S}_n^+$ and consequently $\mathcal{P}_k \in S_{\mathbf{n}}^+$.
\end{proof}

Although $\mathcal{P}_1=\mathcal{L}(\mathcal{B})$ is block diagonal, for high-dimensional problems the size of each block may still be quite large. However, following Lemma \ref{lem: Pi_in_Sn+}, $\mathcal{P}_1 \in \mathcal{S}_{\mathbf{n}}^+$, which suggests recursively applying the lumping operator on its diagonal blocks, which are in $\mathcal{S}_{(n_2,\dots,n_d)}^+$. On the second to last level, all diagonal blocks are in $\mathcal{S}_{n_d}^+$. At this stage, if $\mathcal{B}$ is nonnegative, applying the standard row-sum technique results in the standard row-sum lumped mass matrix. As we progress down the hierarchical structure, the number of diagonal blocks increases but their size decreases. Indeed, on the $k$th level, the number of diagonal blocks is $q_k=\prod_{j=1}^k n_j$ and their size is $r_k=\prod_{j=k+1}^d n_j$ such that for all $k$ the product $q_kr_k=n$ is the size of the full matrix. The following definition formalizes the procedure.

\begin{definition}[Hierarchical lumped matrices]
\label{def: hierarchical_lumped_matrices}
Let $\mathcal{B} \in \mathcal{S}_{\mathbf{n}}^+$ with $\mathbf{n}=(n_1,\dots,n_d) \in \mathbb{N}^d$ and $d \geq 2$. Set $\mathcal{H}_1=\mathcal{L}(\mathcal{B})$ and let $\mathcal{H}_k$ for $1 \leq k \leq d-1$ be such that
\begin{equation*}
    \mathcal{H}_k= \diag(D_{k,1},\dots,D_{k,q_k})
\end{equation*}
where $q_k=\prod_{j=1}^k n_j$. Then, $\mathcal{H}_{k+1}$ is defined from $\mathcal{H}_k$ as 
\begin{equation*}
    \mathcal{H}_{k+1} = \diag(\mathcal{L}(D_{k,1}),\dots,\mathcal{L}(D_{k,q_k})).
\end{equation*}
\end{definition}
Figure \ref{fig: sparsity_pattern_hierarchical_ML}, for example, shows the sparsity pattern of a matrix $\mathcal{B} \in \mathcal{S}_{(n_1,n_2,n_3)}^+$ together with its hierarchical lumped mass matrices $\mathcal{H}_k$ for $k=1,2,3$. By construction, hierarchical mass lumping reduces the bandwidth down the hierarchical structure: if $\mathcal{B}$ is $d$-level banded with bandwidths $(b_1,b_2,\dots,b_d)$, then $\mathcal{H}_k$ is $d$-level banded with bandwidths $(0,\dots,0,b_{k+1},\dots,b_d)$. Note that the numbering of $\mathcal{H}_k$ refers to the hierarchical level and hence typically ranges from $1$ to $d$ (contrary to the numbering of $\mathcal{P}_i$, which ranges from $1$ to $n_1$). Similarly to \eqref{eq: Loewner_sequence_single_patch}, hierarchical lumped matrices also satisfy an order relation.

\begin{corollary}
\label{cor: hierarchical_order}
Let $\mathcal{B} \in \mathcal{S}_{\mathbf{n}}^+$ with $\mathbf{n} \in \mathbb{N}^d$ and $d \geq 2$. Then, the sequence of matrices $\{\mathcal{H}_k\}_{k=1}^d$ constructed from $\mathcal{B}$ according to Definition \ref{def: hierarchical_lumped_matrices} satisfies \begin{equation*}
    \mathcal{H}_d \succeq \mathcal{H}_{d-1} \succeq \dots \succeq \mathcal{H}_1.
\end{equation*} 
\end{corollary}
\begin{proof}
The proof is an obvious consequence of Lemma \ref{lem: block_positivity}.    
\end{proof}

\begin{figure}[htbp]
     \centering
     \begin{subfigure}[t]{0.22\textwidth}
    \centering
    \includegraphics[width=\textwidth]{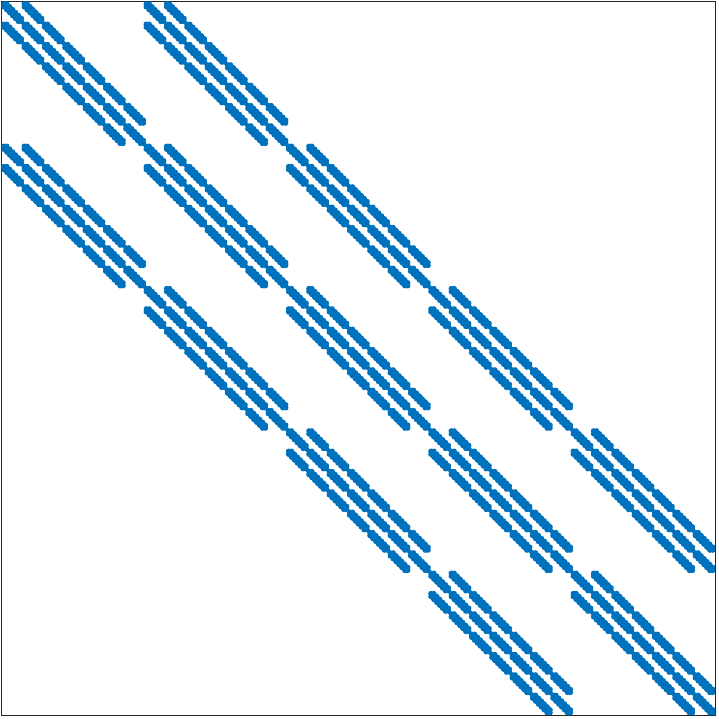}
    \caption{$\mathcal{B}$}
    \label{fig: 3D_Laplace_magnet_sparsity_M_n4_4_3_p1}
     \end{subfigure}
     \hfill
     \begin{subfigure}[t]{0.22\textwidth}
    \centering
    \includegraphics[width=\textwidth]{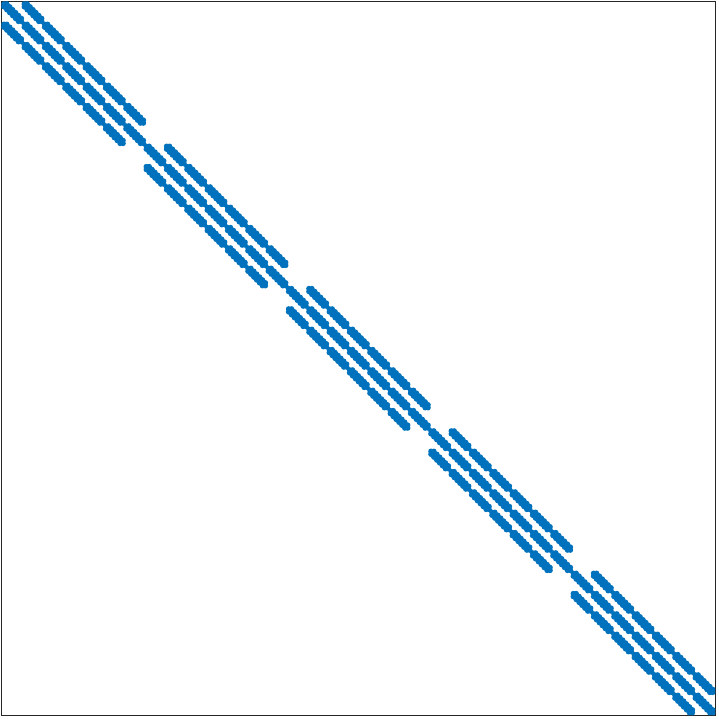}
    \caption{$\mathcal{H}_1$}
    \label{fig: 3D_Laplace_magnet_sparsity_H1_n4_4_3_p1}
     \end{subfigure}
     \hfill
    \begin{subfigure}[t]{0.22\textwidth}
    \centering
    \includegraphics[width=\textwidth]{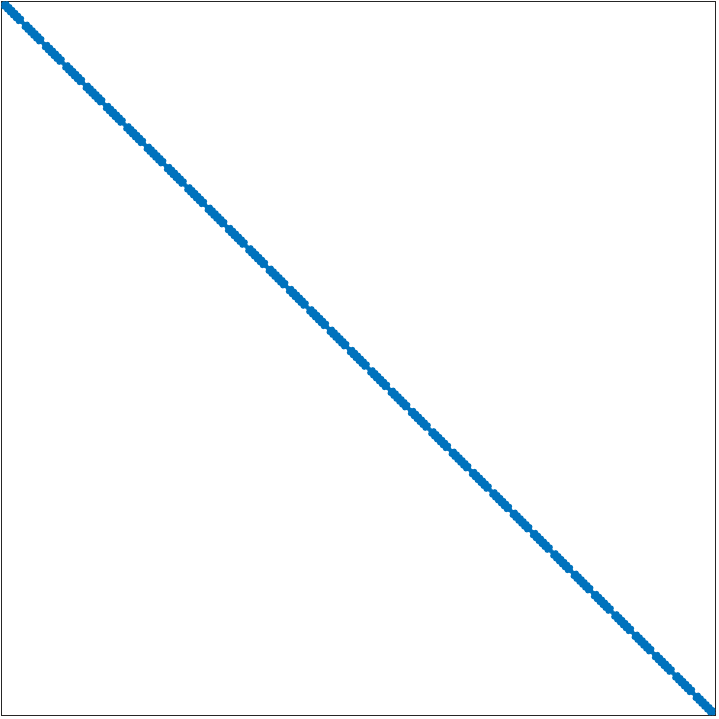}
    \caption{$\mathcal{H}_2$}
    \label{fig: 3D_Laplace_magnet_sparsity_H2_n4_4_3_p1}
     \end{subfigure}
     \hfill
    \begin{subfigure}[t]{0.22\textwidth}
    \centering
    \includegraphics[width=\textwidth]{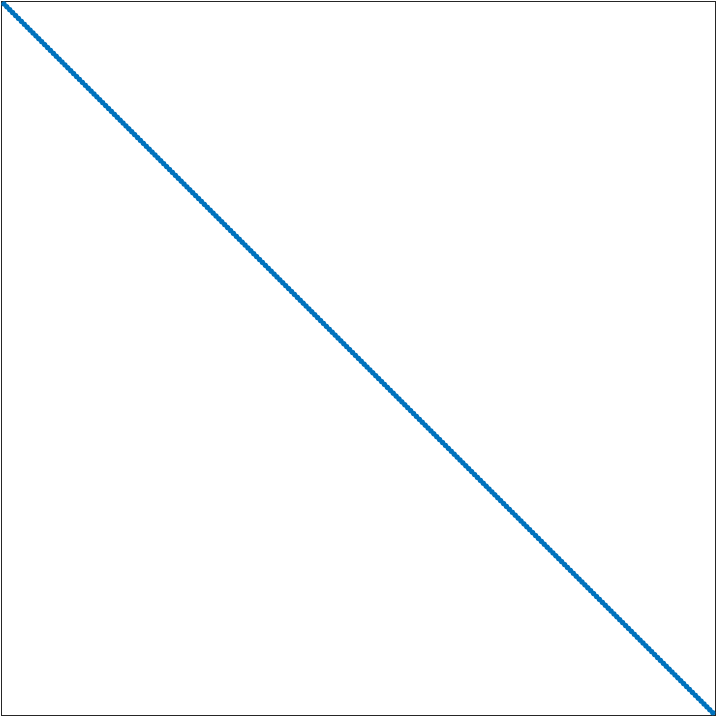}
    \caption{$\mathcal{H}_3$}
    \label{fig: 3D_Laplace_magnet_sparsity_H3_n4_4_3_p1}
     \end{subfigure}
     \hfill
    \caption{Sparsity patterns}
    \label{fig: sparsity_pattern_hierarchical_ML}
\end{figure}

\subsection{Multipatch mass lumping}
\label{se: multipatch_mass_lumping}
We recall that in the multipatch setting, the stiffness and mass matrices are expressed as
\begin{equation*}
    \mathcal{K}=\sum_{r=1}^{N_p} R_r^T\mathcal{K}_rR_r \quad \text{and} \quad \mathcal{M}=\sum_{r=1}^{N_p} R_r^T\mathcal{M}_rR_r
\end{equation*}
where $N_p$ is the number of patches, $\mathcal{K}_r$ and $\mathcal{M}_r$ are the local stiffness and mass matrices of the $r$th patch and $R_r$ maps its local degrees of freedom to global ones. Since $\mathcal{M}_r$ are single-patch mass matrices, it motivates the following definition of multipatch mass lumping.

\begin{definition}[Multipatch lumped matrices]
\label{def: multipatch_mass_lumping}
Let $\mathcal{B}=\sum_{r=1}^{N_p} R_r^T\mathcal{B}_rR_r$ be a multipatch matrix, where $\mathcal{B}_r \in \mathcal{S}_{\mathbf{n}}^+$ for all $r=1,\dots,N_p$. We define $\mathcal{P}_{\mathbf{i}}=\sum_{r=1}^{N_p} R_r^T\mathcal{P}_{r,i_r}R_r$ as a multipatch lumped matrix, where $\mathcal{P}_{r,i_r}$ is constructed from $\mathcal{B}_r$ following Definition \ref{def: block_lumped_matrices} and $\mathbf{i}=(i_1,\dots,i_{N_p})$ is a multi-index.
\end{definition}

For notational convenience, we will assume that the discretization parameters are identical for each patch such that we may choose $i_r=i$ for all patches $r=1,\dots,N_p$ and simply denote $\mathcal{P}_{i}$ the resulting multipatch lumped mass matrix. Although this notation conflicts with the single-patch case, it will always be clear from the context whether $\mathcal{P}_{i}$ refers to a single-patch or multipatch lumped mass matrix. The next lemma generalizes our previous findings to the multipatch case.

\begin{lemma}
\label{lem: multipatch_order}
Let $\mathcal{B}=\sum_{r=1}^{N_p} R_r^T\mathcal{B}_rR_r$, where $\mathcal{B}_r \in \mathcal{S}_{\mathbf{n}}^+$ for all $r=1,\dots,N_p$. Then the sequence of matrices $\{\mathcal{P}_{i}\}_{i=1}^{n_1}$ constructed from $\mathcal{B}$ following Definition \ref{def: multipatch_mass_lumping} satisfies
\begin{equation*}
    \mathcal{P}_1 \succeq \mathcal{P}_2 \succeq \dots \succeq \mathcal{P}_{n_1-1} \succeq \mathcal{P}_{n_1}=\mathcal{B}.
\end{equation*}
\end{lemma}
\begin{proof}
First recall that for any symmetric matrices $A,B \in \mathbb{R}^{n \times n}$ and any $V \in \mathbb{R}^{n \times m}$, if $A \succeq B$, then $V^TAV \succeq V^TBV$ \citep[][Theorem 7.7.2(a)]{horn2012matrix}. The result then immediately follows since for any $1 \leq r \leq N_p$ and any index $1 \leq i < n_1$,
\begin{align*}
    \mathcal{P}_{r,i} &\succeq \mathcal{P}_{r,i+1}, & &\text{see \eqref{eq: Loewner_sequence_single_patch}}\\
    \implies R_r^T\mathcal{P}_{r,i}R_r &\succeq R_r^T\mathcal{P}_{r,i+1}R_r, & & \\
    \implies \mathcal{P}_i=\sum_{r=1}^{N_p} R_r^T\mathcal{P}_{r,i}R_r &\succeq \sum_{r=1}^{N_p} R_r^T\mathcal{P}_{r,i+1}R_r=\mathcal{P}_{i+1}. & &
\end{align*}
The statement then follows from repeated application of the previous relation.
\end{proof}
For high-dimensional problems, it might be useful to resort to hierarchical mass lumping techniques on the single-patch level, as described in Section \ref{se: hierarchical_mass_lumping}. There is an obvious analogue of Definition \ref{def: multipatch_mass_lumping} and Lemma \ref{lem: multipatch_order} for such cases.

The purpose of mass lumping is first and foremost to reduce the block bandedness of the mass matrix and guarantee a CFL condition that cannot be worse than the original one. However, generally speaking, mass lumping does not significantly improve the CFL while it might undermine the accuracy. Multipatch problems are not exempt and the issue already originates on the single-patch level, before local matrices are merged into the global one. This merging process is identical to the assembly procedure of classical finite element methods, which is not surprising given the analogy between patches and elements. Thus, the proof of the following lemma is analogous to the standard finite element case (see e.g. \citep{irons1971bound, hughes1979implicit, wathen1987realistic}).

\begin{lemma}
\label{lem: bounds_generalized_eig}
Let $\mathcal{A}=\sum_{r=1}^{N_p} R_r^T\mathcal{A}_rR_r$ and $\mathcal{B}=\sum_{r=1}^{N_p} R_r^T\mathcal{B}_rR_r$, where $\mathcal{A}_r \in \mathcal{S}_n$ and $\mathcal{B}_r \in \mathcal{S}_n^+$ for all $r=1,\dots,N_p$. Then,
\begin{equation*}
    \min_r \lambda_{\min}(\mathcal{A}_r,\mathcal{B}_r) \leq \lambda_{\min}(\mathcal{A},\mathcal{B}), \qquad \lambda_{\max}(\mathcal{A},\mathcal{B}) \leq \max_r \lambda_{\max}(\mathcal{A}_r,\mathcal{B}_r).
\end{equation*}
\end{lemma}
The upper bound of Lemma \ref{lem: bounds_generalized_eig} may be quite tight and is an incentive for acting on the single-patch matrices prior to assembly.

\subsection{Solving linear systems with the lumped mass matrix}
\label{se: linear_system_solvers}
In the single-patch case, linear systems with the lumped mass matrices are conveniently solved using sparse Cholesky factorizations; i.e. $P=LL^T$, where $L$ is a lower triangular matrix. It is well-known that the bandwidth of the Cholesky factor $L$ cannot be larger than the bandwidth of $P$. For this reason, many techniques for minimizing the fill-in actually minimize the bandwidth of a permuted matrix. The mass lumping techniques discussed in this work reduce the bandwidth of the consistent mass, thereby significantly accelerating sparse direct solvers. Indeed, a direct application of Lemma \ref{lem: bandwidth} shows that the bandwidth of $\mathcal{H}_k$ is
\begin{equation*}
    (0,\dots,0,b_{k+1},\dots,b_d) \cdot (r_1,\dots,r_k,r_{k+1},\dots,r_d) = \sum_{i=k+1}^d b_ir_i.
\end{equation*}
In comparison to the bandwidth of $\mathcal{B}$, the bandwidth of $\mathcal{H}_k$ suppresses the first $k$ largest contributors to the sum. For multi-dimensional problems, it is a compelling argument for first reducing the bandwidth at the top of the hierarchy and then working our way downward.

While solving linear systems with the lumped mass matrix in the single-patch case is relatively straightforward, the multipatch case deserves some more explanations. The multipatch lumped mass matrix (after potentially a symmetric permutation) is a generalized saddle point matrix \cite{benzi2005numerical}, expressed as 
\begin{equation*}
    \mathcal{P}=
    \begin{pmatrix}
        D & C \\
        C^T & X
    \end{pmatrix}
    \quad \text{with} \quad D=\diag(D_1,\dots,D_{N_p}).
\end{equation*}
We consider the linear system
\begin{equation*}
    \begin{pmatrix}
        D & C \\
        C^T & X
    \end{pmatrix}
    \begin{pmatrix}
        \mathbf{x} \\
        \mathbf{y}
    \end{pmatrix}
    =
    \begin{pmatrix}
        \mathbf{f} \\
        \mathbf{g}
    \end{pmatrix}.
\end{equation*}
After block Gaussian elimination, we solve the upper block triangular system
\begin{equation}
\label{eq: upper_triangular_system}
    \begin{pmatrix}
        D & C \\
        0 & S
    \end{pmatrix}
    \begin{pmatrix}
        \mathbf{x} \\
        \mathbf{y}
    \end{pmatrix}
    =
    \begin{pmatrix}
        \mathbf{f} \\
        \tilde{\mathbf{g}}
    \end{pmatrix}
\end{equation}
where $S=X-C^TD^{-1}C$ is the Schur complement of $D$ in $\mathcal{P}$ and $\tilde{\mathbf{g}}=\mathbf{g}-C^TD^{-1}\mathbf{f}$. Contrary to the consistent mass matrix (which features a similar structure), the Schur complement of the lumped mass matrix can be formed explicitly and cheaply owing to its sparsity and simple block diagonal structure. Once the Schur complement is formed, which is done once and for all, \eqref{eq: upper_triangular_system} can be solved by backward substitution. This strategy is advantageous given that within time stepping schemes, one must solve a \emph{sequence} of linear systems and not just a single one.

\section{Outlier removal}
\label{se: outlier_removal}
\subsection{Deflation techniques}
\label{se: deflation}
Mass lumping generally mitigates but does not completely eliminate outlier frequencies from the spectrum. Thus, it is usually combined with dedicated outlier removal techniques. Unfortunately, the methods described in \citep{sande2019sharp, deng2021boundary, manni2022application, hiemstra2021removal} are only applicable to highly structured problems rarely met in practical applications. Moreover, numerical experiments show that predefined penalization terms barely help remove outliers and must instead be tailored to the specific problem at hand. For this reason, we present in this section an algebraic outlier removal technique based on low-rank perturbations that is robust with respect to the geometry. Our strategy consists in deflating the spectrum from its largest eigenvalues, while preserving the smallest ones. Of course, this choice of scaling assumes that the dynamics are completely resolved by the low-frequency part of the spectrum, which is often (nearly) the case. We first recall some preliminary results, providing the theoretical foundation of the method.

\begin{lemma}[{\citep[][Theorem VI.1.15]{stewart1990matrix}}]
\label{lem: classical_problem}
Let $A,B \in \mathbb{R}^{n \times n}$ be symmetric matrices with $B$ positive definite. Then, all generalized eigenvalues of $(A,B)$ are real and there exists an invertible matrix $U \in \mathbb{R}^{n \times n}$ such that
\begin{equation*}
    U^TAU=D, \qquad U^TBU=I,
\end{equation*}
where $D=\diag(\lambda_1, \dots, \lambda_n)$ is a real diagonal matrix containing the eigenvalues.
\end{lemma}

\begin{definition}[Scaled matrix pencil]
\label{def: deflated_pencil}
Let $A,B \in \mathbb{R}^{n \times n}$ be symmetric matrices with $B$ positive definite and $f$, $g$ be two functions defined on the spectrum of $(A,B)$. The scaled pencil $(\bar{A},\bar{B})$ is defined as
\begin{align*}
    \bar{A} &= A+Vf(D_2)V^T, \\
    \bar{B} &= B+Vg(D_2)V^T,
\end{align*}
where $V=BU_2 \in \mathbb{R}^{n \times r}$, with $U_2=[\mathbf{u}_{n-r+1}, \dots, \mathbf{u}_n]$ the matrix formed by the last $r$ $B$-orthonormal eigenvectors of $(A,B)$ and $D_2=\diag(\lambda_{n-r+1}, \dots, \lambda_n) \in \mathbb{R}^{r \times r}$ the diagonal matrix formed by the last $r$ eigenvalues with $r \ll n$.    
\end{definition}

The next theorem shows that this definition provides the desired scaling.

\begin{theorem}[Deflation of matrix pencils]
\label{th: low_rank_pert_AB}
Let $A,B \in \mathbb{R}^{n \times n}$ be symmetric matrices with $B$ positive definite and $(\bar{A},\bar{B})$ be the scaled pencil introduced in Definition \ref{def: deflated_pencil}. Then,
\begin{itemize}[noitemsep]
    \item The eigenvectors of $(A,B)$ and $(\bar{A},\bar{B})$ are the same.
    \item The eigenvalues of $(\bar{A},\bar{B})$ are given by:
    \begin{equation*}
    \bar{\lambda}_{i_k}=
    \begin{cases}
    \lambda_k & \text{ for } k=1,\dots,n-r, \\
    \frac{\lambda_k+f(\lambda_k)}{1+g(\lambda_k)} & \text{ for } k=n-r+1,\dots,n.
    \end{cases}
    \end{equation*}
\end{itemize}
\end{theorem}
\begin{proof}
We first note that the matrices $\bar{A}$ and $\bar{B}$ can be rewritten as
\begin{align*}
    \bar{A} &= A+BU\diag(0, f(D_2))U^TB, \\
    \bar{B} &= B+BU\diag(0, g(D_2))U^TB,
\end{align*}
where $\diag(0, f(D_2)), \diag(0, g(D_2)) \in \mathbb{R}^{n \times n}$ are the block diagonal matrices obtained by appending zeros to $f(D_2)$ and $g(D_2)$, respectively, and $U$ is the matrix of eigenvectors. Verifying that $\mathbf{u}_i$ is an eigenvector of $(\bar{A},\bar{B})$ for $i=1,\dots,n$ is straightforward. Moreover, since the matrix pencils $(\bar{A},\bar{B})$ and $(U^T\bar{A}U, U^T\bar{B}U)$ are equivalent \citep[][Chapter 15]{parlett1998symmetric},
\begin{align*}
\Lambda(\bar{A},\bar{B})&=\Lambda(U^T\bar{A}U, U^T\bar{B}U) \\
&=\Lambda(D+\diag(0, f(D_2)),I+\diag(0, g(D_2))) \\
&=\{\lambda_k\}_{k=1}^{n-r} \cup \left\{\frac{\lambda_k+f(\lambda_k)}{1+g(\lambda_k)}\right\}_{k=n-r+1}^n,
\end{align*}
where the second equality follows from Lemma \ref{lem: classical_problem}.
\end{proof}

\begin{remark}
In numerical linear algebra, \emph{deflation} refers to the removal of unwanted eigenvalues. The result of Theorem \ref{th: low_rank_pert_AB} is analogous to deflation ``by substraction'', which originated from the early work of Hotelling \cite{hotelling1943some}. The reader may refer to \citep{parlett1998symmetric, saad2011numerical} for an overview of deflation techniques.
\end{remark}

The previous theorem allows to map the largest eigenvalues of $(A,B)$ to virtually any real number. However, the transformation must be carefully chosen such that it does not reduce too much the outlier frequencies. Indeed, since the eigenvectors are not affected by the transformation, if outlier frequencies are mapped to low frequencies, their spurious eigenvectors will artificially enter the solution, which might have disastrous consequences for the dynamics. We give below some suitable choices for the functions $f$ and $g$ that avoid this issue.

\begin{enumerate}[noitemsep]
    \item Set $f(\lambda)=\lambda_{n-r}-\lambda$ and $g(\lambda)=0$. 
    \item Set $f(\lambda)=0$ and $g(\lambda)=\frac{\lambda}{\lambda_{n-r}}-1$.
    \item More generally, choose any function $g(\lambda)$ (defined on the spectrum of $(A,B)$) and set $f(\lambda)=\lambda_{n-r}(1+g(\lambda))-\lambda$.
\end{enumerate}
In the first case, a negative semidefinite perturbation is added to the stiffness matrix while in the second case, a positive semidefinite perturbation is added to the mass matrix. The latter has already been proposed in \cite{tkachuk2014local,gonzalez2020large} as a mass scaling strategy (referred therein as \emph{spectral scaling} and \emph{mass tailoring}, respectively). The nature of the perturbation in the third case depends on the specific choice of functions. In the remaining part of the article, we will confine ourselves to the choices of $f$ and $g$ listed above, which all lead to the same transformed eigenvalues, given by
\begin{equation*}
    \bar{\lambda}_{k}=
    \begin{cases}
    \lambda_k & \text{ for } k=1,\dots,n-r, \\
    \lambda_{n-r} & \text{ for } k=n-r+1,\dots,n.
    \end{cases}
\end{equation*}

The result, graphically illustrated in Figure \ref{fig: truncation}, consists in shaving off the upper part of the spectrum. Note that in our context $\lambda_{n-r}$ is the largest ``regular'' (or non-outlier) eigenvalue. In principle, it could be replaced with a cutoff value, as suggested in \cite{tkachuk2014local,gonzalez2020large}, to avoid computing an additional eigenvalue. However, choosing $\lambda_{n-r}$ preserves the eigenvalue numbers and prevents spurious eigenfunctions from moving to the lower to intermediate frequency range. Thus, we prefer computing this additional eigenvalue, which in practice barely introduces any overhead.
\begin{figure}[htbp]
    \centering
    \includegraphics[scale=0.4]{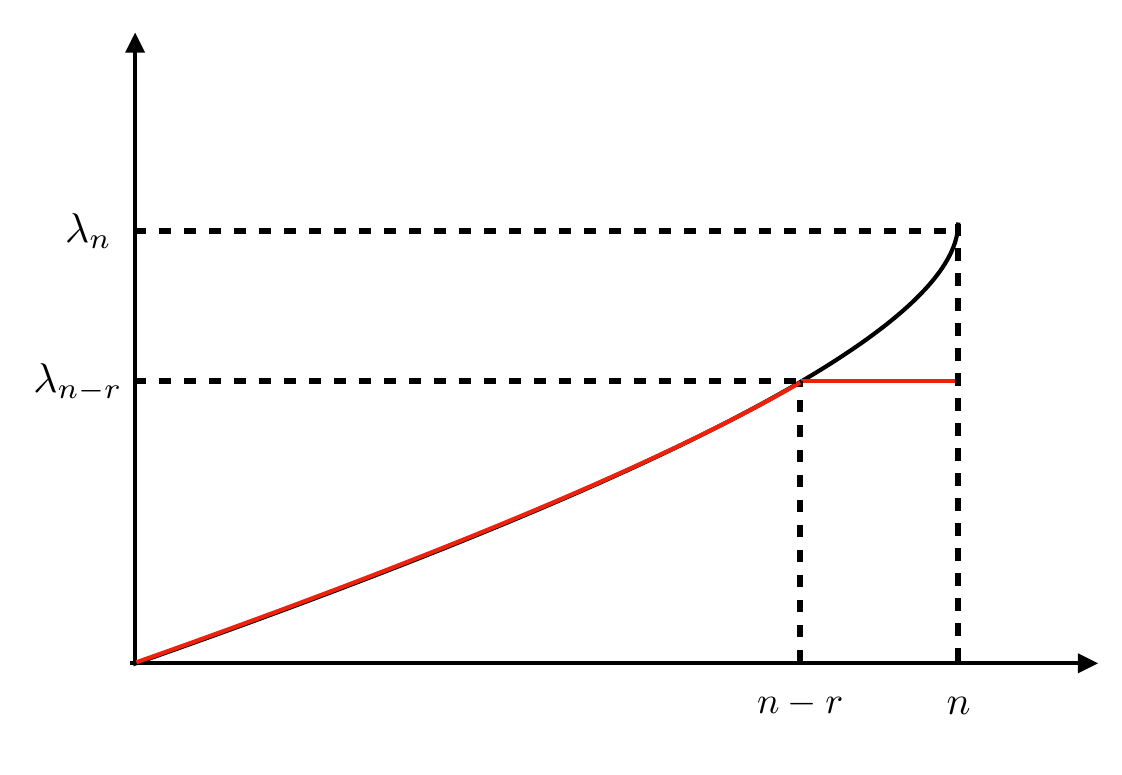}
    \caption{Truncation of the largest eigenvalues}
    \label{fig: truncation}
\end{figure}
Clearly, once the $r+1$ largest eigenpairs have been computed, the increase in critical time step is known. For instance, for the central difference method in the undamped case (see \eqref{eq: CFL_central_difference})
\begin{equation*}
    \frac{\bar{\Delta t_c}}{\Delta t_c}=\sqrt{\frac{\lambda_n}{\lambda_{n-r}}}.
\end{equation*}
 On the one hand, Theorem \ref{th: low_rank_pert_AB} circumvents the lack of robustness from predefined perturbations terms. On the other hand, the cost for its implementation is also much higher since it involves a few eigenpairs, which must be computed on a case-by-case basis. We first focus on the rationale of our method and discuss its computational cost more thoroughly in Section \ref{se: eigenvalue_comp}.

As we have seen, we may deflate the spectrum by either perturbing the mass, stiffness or both. In order to align ourselves with common practice, which tends to only modify the mass matrix, we propose a two-step lumping-scaling strategy: we first approximate the mass matrix with one of the lumping strategies proposed in Section \ref{se: mass_lumping} (or, alternatively, any suitable ad hoc mass lumping technique) and then scale the lumped mass matrix to remove persistent outliers. Note that the scaled mass matrix is generally completely dense and must obviously never be formed explicitly (in contrast to the method proposed in \cite{tkachuk2014local}). Instead, it is represented implicitly by only storing the lumped mass matrix and the terms $V$ and $g(D_2)$ defining the low-rank perturbation. Moreover, since the perturbation is low-rank, the scaled mass matrix may be easily inverted thanks to the Woodbury matrix identity \cite{woodbury1950inverting} leading to
\begin{equation}
\label{eq: woodbury_identity}
    \bar{B}^{-1}=B^{-1}-U_2(g(D_2)^{-1}+I_r)^{-1}U_2^T.
\end{equation}
Since $g(D_2)$ is diagonal (and nonsingular), $(g(D_2)^{-1}+I_r)^{-1}$ can be formed explicitly. Thus, solving a linear system with the scaled mass matrix only requires solving a linear system with the lumped mass matrix and computing a matrix-vector multiplication with a low-rank matrix. While the former uses standard techniques, as described in Section \ref{se: linear_system_solvers}, the latter only requires $O(rn)$ additional flops. Thus, if the perturbation's rank is relatively small, these matrix-vector multiplications do not introduce any significant overhead. Due to repeated matrix-vector multiplications with the stiffness matrix in the time stepping scheme, the cost incurred by instead scaling the stiffness matrix is the same and eventually the choice is just a matter of taste. Although stability concerns have been raised for the Woodbury identity \cite{yip1986note}, instabilities were never experienced during our tests. The method described herein is very general and could even be beneficial for practitioners using the consistent mass. Indeed, the scaling does not affect the smallest eigenpairs and therefore preserves the higher order convergence characterizing the consistent mass. However, the eigenpairs must then be computed to high precision, which is often not necessary when using a lumped mass whose small frequencies are only second order accurate.

Nevertheless, the strategy may seem rather impractical given that it requires explicit knowledge of the outlier eigenvalues and associated eigenvectors, whose number grows under mesh refinement \citep{hiemstra2021removal}. Surprisingly, this aspect was completely neglected in earlier works \cite{tkachuk2014local, gonzalez2020large}. Several arguments underpin our strategy. Firstly, in contract to classical $C^0$ finite elements for which similar methods were proposed, maximally smooth $C^{p-1}$ spline discretizations feature far fewer outlier eigenvalues, as outlined in the Appendix. Secondly, although practitioners often rule out (approximate) eigenvalue computations as prohibitively expensive, as we will discuss in the next subsection, the workload for computing a few of the largest eigenpairs with the Lanczos method is very similar to performing a few iterations of an explicit algorithm for dynamical simulations. Thus, it might be worthwhile spending a few iterations to remove outliers if we might save up on hundreds of iterations later on, especially for long-time simulations. The Lanczos method is the state-of-the-art solver for sparse symmetric generalized eigenvalue problems. It is briefly summarized in the next section to support our argument.

\subsection{Eigenvalue and eigenvector computations}
\label{se: eigenvalue_comp}
The Lanczos method for generalized eigenproblems can be derived from the one for standard eigenproblems, after transforming the generalized eigenproblem to standard form; e.g. via the Cholesky factorization of $B$. However, further transformations are needed for computational efficiency and numerical stability. A basic version is presented in Algorithm \ref{algo: Lanczos}.

\begin{algorithm}[ht]
\begin{algorithmic}[1]
\caption{Lanczos method \citep[][Algorithm 9.1]{saad2011numerical}}
\label{algo: Lanczos}
\Statex \textbf{Input}: Symmetric matrix pair $(A,B)$ with $B$ positive definite, starting vector $\mathbf{b}$, number of iterations $m$
\Statex \textbf{Output}: $m$ approximate eigenpairs of $(A,B)$
\State Set $\mathbf{v}_0=0$, $\mathbf{v}_1=\mathbf{b}/\|\mathbf{b}\|_B$, $\beta_1=0$
\For{$j=1,2,\cdots, m$}
    \State $\mathbf{v}=A\mathbf{v}_j$
    \State $\alpha_j = (\mathbf{v}, \mathbf{v}_j)$
    \State $\mathbf{w}=B^{-1}\mathbf{v}-\alpha_j \mathbf{v}_j-\beta_j\mathbf{v}_{j-1}$
    \State $\beta_{j+1}=\sqrt{(\mathbf{v}, \mathbf{w})}$
    \State $\mathbf{v}_{j+1}=\mathbf{w}/\beta_{j+1}$
\EndFor
\end{algorithmic}
\end{algorithm}

The vectors $\mathbf{v}_j$ computed during the course of the iterations serve to build a $B$-orthonormal basis for a Krylov subspace from which eigenvalue and eigenvector approximations are extracted. The interested reader may refer to standard textbooks detailing this procedure; e.g. \citep{parlett1998symmetric, saad2011numerical}.

\begin{remark}
Due to the propagation of round-off errors, a reorthogonalization procedure is necessary to restore, at least occasionally, the $B$-orthonormality of the Krylov basis. Implementations of the Lanczos method are further supplemented with restarting procedures that truncate the Krylov basis to avoid its prohibitive growth and potential storage issues. Common guidelines recommend using a basis size of $m=2k$, where $k$ is the number of desired eigenpairs \cite{stewart2002krylov}. Finally, convergence checks are also implemented to serve as stopping criterion. We have voluntarily left aside all those advanced topics to focus on the essential. The interested reader may refer to the extensive literature for a detailed discussion; e.g. \citep{stewart2002krylov, wu2000thick, parlett1998symmetric, saad2011numerical}.
\end{remark}

Although we have implemented an eigensolver from scratch for the sake of writing this paper, it is absolutely not necessary for applying the techniques presented herein. Several efficient implementations are available in software packages and all the user has to worry about is supplying an algorithm for computing matrix-vector multiplications with the (implicitly defined) stiffness matrix and solving linear systems with the (lumped) mass matrix, which depends on the nature of the problem (e.g. single-patch, multipatch,...). These two operations are reminiscent of the central difference method and as a matter of fact, if the number of iterations $m$ remains relatively small, which is ensured through restarting procedures, each iteration of the Lanczos method costs nearly as much as an iteration of the central difference method. Moreover, the Lanczos method is known to converge very fast to eigenvalues that are well separated from the rest of the spectrum \citep{saad1980rates, parlett1998symmetric, saad2011numerical}, which is precisely a distinctive feature of outliers. However, the time span of the simulation must be sufficiently large to amortize the cost of computing outlier eigenpairs. In general, the shorter the simulation, the smaller the number of outliers we can afford computing. In practice, their number is further limited by computer resources. For very large problems, simply storing all outlier eigenvectors is plainly infeasible. Following common guidelines relating the size of the Krylov subspace to the number of eigenpairs \cite{stewart2002krylov}, if one can afford storing $2k$ vectors, then one might compute roughly $k$ eigenpairs. Finally, although our method scales down outlier frequencies, it does not remove the corresponding outlier modes that might negatively impact the solution.

The deflation procedure proposed in this section is very general and can in principle be applied to any type of problem, including nontrivial (multipatch) geometries. In the multipatch setting, it is possible to locally scale the single-patch system matrices, similarly to local (elementwise) mass scaling techniques \cite{tkachuk2014local, eisentrager2024eigenvalue}. Indeed, Lemma \ref{lem: bounds_generalized_eig} reveals that the largest eigenvalues of  $(\mathcal{K},\mathcal{P}_i)$ could be controlled by the largest eigenvalues of $(\mathcal{K}_r,\mathcal{P}_{r,i})$ and suggests a scaling strategy directly targeting the origin of the issue: at the patch level. While this strategy is generally cheaper than scaling $(\mathcal{K},\mathcal{P}_i)$ globally, it has three shortcomings: firstly, the assembly into global matrices generally affects the smallest eigenvalues; secondly it cannot remove outliers introduced by the $C^0$ coupling of patch interfaces and finally, since the inverse of the global mass matrix \emph{is not} given by the assembly of the local inverses, we cannot directly apply \eqref{eq: woodbury_identity} locally. To resolve the third issue, we suggest locally scaling the stiffness matrix instead. Denoting $\bar{\mathcal{K}}_r$ the locally scaled stiffness matrix and recalling that the perturbation is negative semidefinite (see Section \ref{se: deflation}), $\bar{\mathcal{K}}_r \preceq \mathcal{K}_r$ and
\begin{equation*}
\bar{\mathcal{K}}:=\sum_{r=1}^{N_p} R_r^T\bar{\mathcal{K}}_rR_r \preceq \sum_{r=1}^{N_p} R_r^T\mathcal{K}_rR_r = \mathcal{K}.
\end{equation*}
Consequently, $\lambda_k(\bar{\mathcal{K}},\mathcal{P}_i) \leq \lambda_k(\mathcal{K},\mathcal{P}_i) \leq \lambda_k(\mathcal{K},\mathcal{M})$. Thus, our strategy essentially boils down to computing a few of the largest eigenpairs of a sequence of generalized eigenproblems on single patches, which is also well suited for parallel computations. We will better assess the numerical properties of this method in Section \ref{se: numerical_experiments}.

\begin{remark}
For some special cases it is yet far more advantageous to exploit the structure of the problem. In particular, for problems featuring a Kronecker product structure, all outliers can be removed by separately scaling the 1D factor matrices. For maximally smooth spline discretizations of 1D problems, the number of outliers only depends on the spline order, differential operator and type of boundary conditions \cite{hiemstra2021removal, manni2022application}. Consequently, the number of eigenvalues computed scales \emph{linearly} with the dimension and does not depend on the mesh size. For instance, based on the upper bounds provided in \cite{manni2022application}, for a uniform $C^{p-1}$ discretization of the Laplace on the hypercube $(0,1)^d$ with homogeneous Dirichlet boundary conditions, at most $dp$ eigenpairs are required to remove $O(n^{d-1}(p-1))$ outliers, where $n$ is the dimension of the univariate spline space \cite{hiemstra2021removal}. This method also offers some advantages over the ones presented in \citep{hiemstra2021removal, manni2022application}. Firstly, its implementation is straightforward: it does not require a change of basis, the standard B-spline basis is sufficient. Secondly, our strategy only relies on algebraic properties and not geometric ones, contrary to the method presented in \cite{hiemstra2021removal}, which generally stumbles on domains with curved boundaries, even if the mass matrix is a Kronecker product.  
\end{remark}

\section{Numerical experiments}
\label{se: numerical_experiments}
This section gathers a few numerical experiments designed to verify our theoretical results and demonstrate the usefulness of our strategies in the context of explicit dynamics. All experiments in this section are done using GeoPDEs \cite{vazquez2016new}, an open source Matlab/Octave software package for isogeometric analysis. 

\subsection{Single-patch geometries}
\begin{example}
We consider a cubic discretization of the 2D Laplace on two nontrivial single-patch domains shown in Figures \ref{fig: stretched_square} and \ref{fig: plate_with_hole}: a stretched square and a quarter of a plate with a hole, represented by a (near) singular NURBS patch. These geometries are discretized with $(20,20)$ and $(40,20)$ subdivisions, respectively. Here, $\mathcal{M} \in \mathcal{S}_{(n_1,n_2)}^+$ and we construct the block lumped matrices $\mathcal{P}_i$ for $i=1,2,3$. The spectrum of $(\mathcal{K},\mathcal{M})$ and $(\mathcal{K},\mathcal{P}_i)$, for $i=1,2,3$, is shown in Figures \ref{fig: 2D_Laplace_stretched_square_block_LM_p3_n20} and \ref{fig: 2D_Laplace_plate_with_hole_block_LM_p3_n20} for the stretched square and the plate with a hole, respectively. As predicted, the generalized eigenvalues of the matrix pairs $(\mathcal{K},\mathcal{P}_i)$ monotonically converge to the eigenvalues of $(\mathcal{K},\mathcal{M})$ from below for increasing values of $i$. This property holds for all eigenvalues, including the ``outliers'', now characterized by a sharp but smooth increase of the spectrum rather than a stepwise increase. For this reason, the distinction between ``outlier'' and ``regular'' eigenvalues is rather ambiguous. For nontrivial problems, ``outlier'' eigenvalues are merely large eigenvalues, which can be removed using deflation techniques such as those presented in Section \ref{se: deflation}. In order to assess the practical gains of the procedure, we solve the wave equation on the plate geometry shown in Figure \ref{fig: plate_with_hole} over the time span $T=[0,6]$ with the manufactured solution $u(x,y,t)=xy(x+4)(y-4)(x^2+y^2-1)(2+\sin(2\pi t))$. The numerical solutions are computed with the central difference method using the critical time step \eqref{eq: CFL_central_difference} multiplied by a safeguarding factor of $0.85$. The results are shown in Figure \ref{fig: plate_with_hole_dynamics} for a small time ($t=0.65$) and a larger time ($t=2.65$). Figure \ref{fig: 2D_Laplace_plate_with_hole_dynamics_rel_error_LM_p3_n20} represents the evolution of the $L^2$ error over time. As one could expect, increasing the block bandwidth improves the accuracy of the lumping techniques. Unfortunately, the improvement also holds true for the outliers, which must be removed. In order to evaluate the computational savings of the outlier removal technique, we consider the ratio $(N_s+N_i)/N_w$ for fixed rank values $r$, where $N_s$ and $N_w$ are the number of iterations with and without scaling, respectively, and $N_i$ is the number of iterations needed by the eigensolver for computing the $r+1$ largest eigenpairs. As we have seen in Section \ref{se: eigenvalue_comp}, an iteration of Lanczos costs nearly as much as an iteration of the central difference method and justifies adding $N_i$ to $N_s$. For conciseness, Figure \ref{fig: 2D_Laplace_plate_with_hole_LR_pert_ratio_p4_n20} only reports the results for the block diagonal matrix $\mathcal{P}_1$ and rank values $r=10,20,40$ for a convergence tolerance of $10^{-3}$. We could verify over the range of ranks tested that the scaling did not have any significant effect on the quality of the solution. The horizontal line for $r=0$ in Figure \ref{fig: 2D_Laplace_plate_with_hole_LR_pert_ratio_p4_n20} indicates the absence of scaling and serves as comparison. A ratio strictly larger than $1$ indicates a deficit: the saving due to the scaling could not offset the cost for computing it. This situation commonly occurs for short-time simulations. However, the workload for computing a few eigenpairs is quickly amortized over longer simulations and may eventually save more than $50 \%$ iterations. In general, one should favor smaller ranks for shorter simulations and larger ranks for longer simulations. In practice, fixing the total number of Lanczos iterations and a loose target rank (defining the size of the Krylov subspace) should help keep control of the computational overhead.

\begin{figure}[htbp]
     \centering
     \begin{subfigure}[t]{0.35\textwidth}
    \centering
    \includegraphics[width=\textwidth]{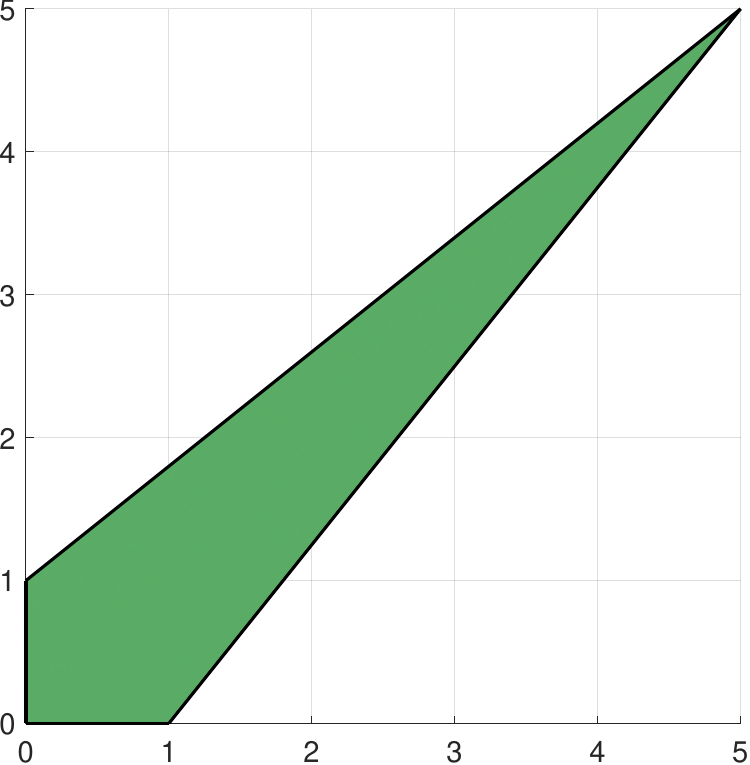}
    \caption{Geometry}
    \label{fig: stretched_square}
     \end{subfigure}
     \hfill
     \begin{subfigure}[t]{0.48\textwidth}
    \centering
    \includegraphics[width=\textwidth]{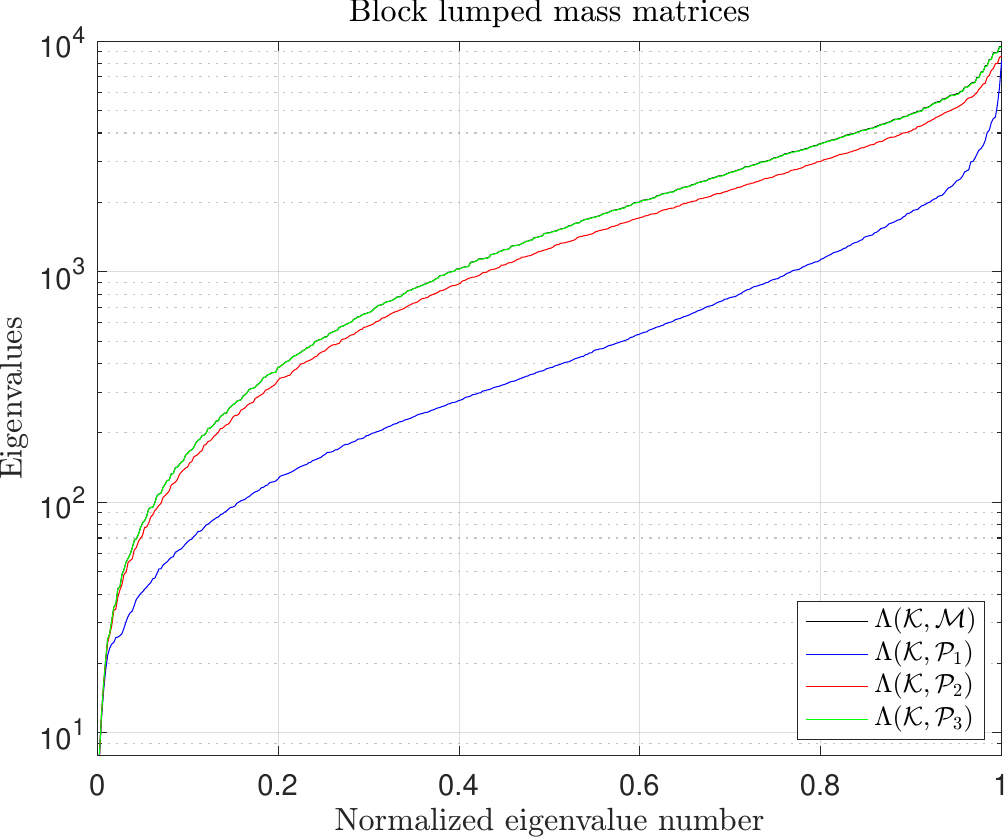}
    \caption{Spectrum}
    \label{fig: 2D_Laplace_stretched_square_block_LM_p3_n20}
     \end{subfigure}
     \hfill
    \caption{Stretched square}
    \label{fig: 2D_tretched_square_eig}
\end{figure}

\begin{figure}[htbp]
     \centering
     \begin{subfigure}[t]{0.35\textwidth}
    \centering
    \includegraphics[width=\textwidth]{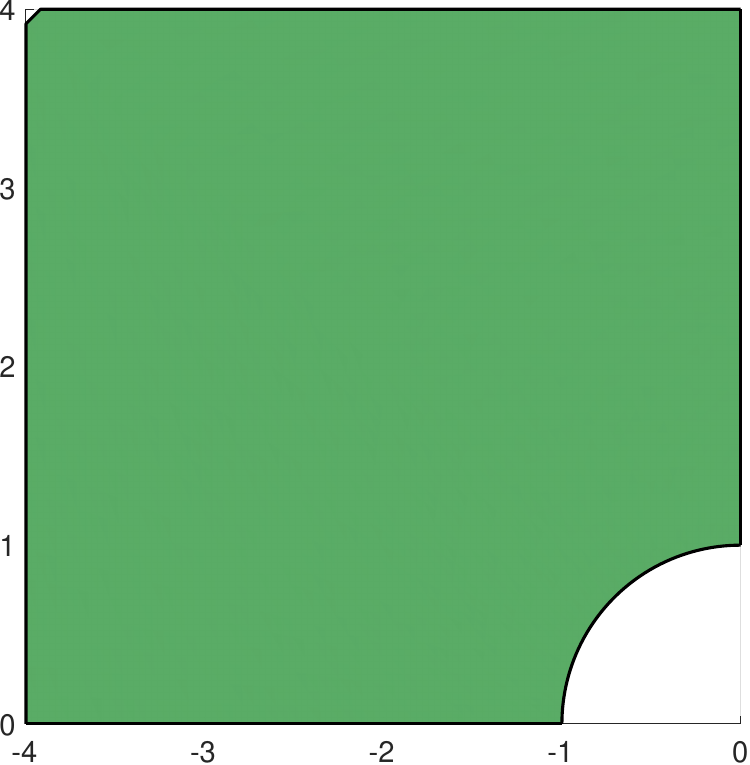}
    \caption{Geometry}
    \label{fig: plate_with_hole}
     \end{subfigure}
     \hfill
     \begin{subfigure}[t]{0.48\textwidth}
    \centering
    \includegraphics[width=\textwidth]{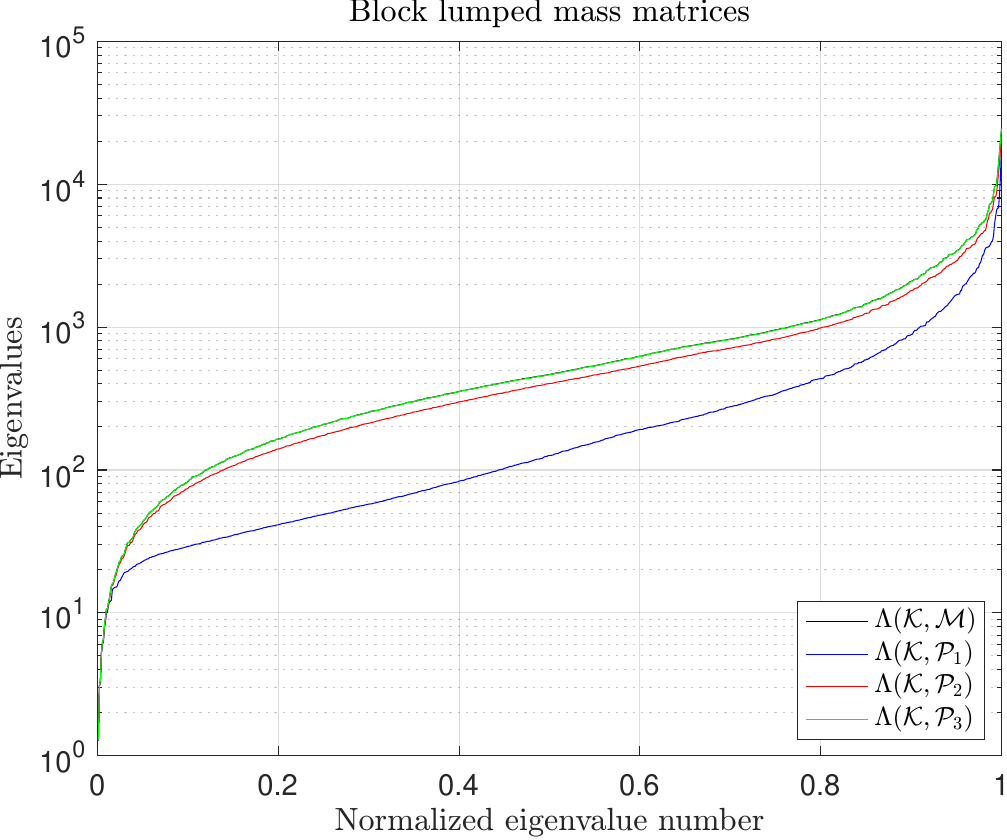}
    \caption{Spectrum}
    \label{fig: 2D_Laplace_plate_with_hole_block_LM_p3_n20}
     \end{subfigure}
     \hfill
    \caption{Quarter of a plate with a hole}
    \label{fig: 2D_plate_with_hole_eig}
\end{figure}

\begin{figure}[htbp]
    \centering
    \begin{subfigure}[t]{0.95\textwidth}
    \centering
    \includegraphics[width=\textwidth]{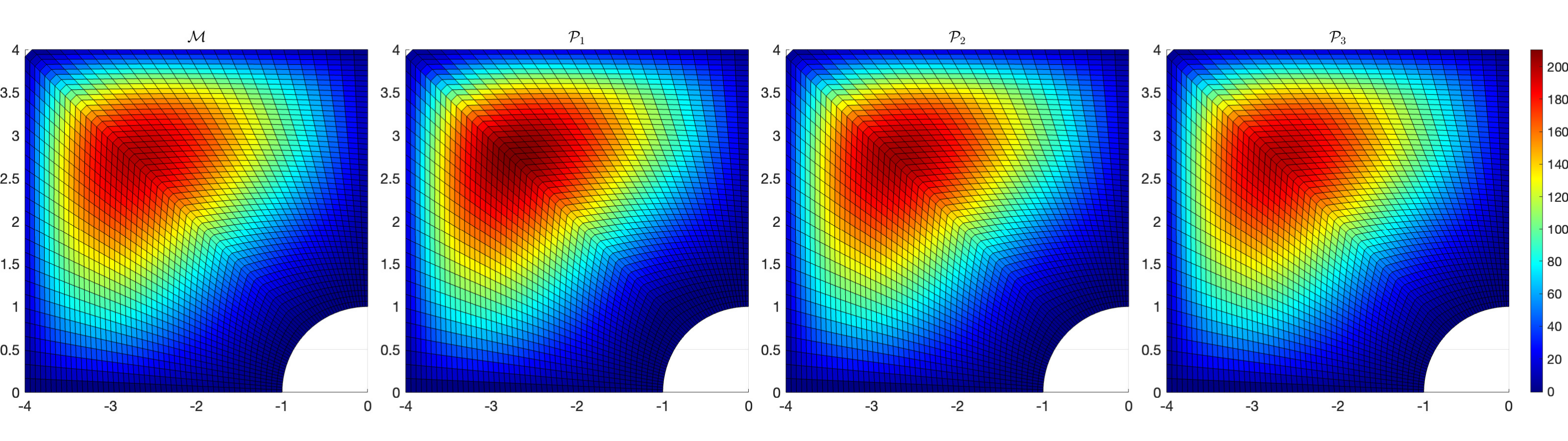}
    \caption{Solution at time $t=0.65$}
    \label{fig: 2D_Laplace_plate_with_hole_dynamics_solution_block_LM_time_0_65_p3_n20}
    \end{subfigure}
    \hfill
    \begin{subfigure}[t]{0.95\textwidth}
    \centering
    \includegraphics[width=\textwidth]{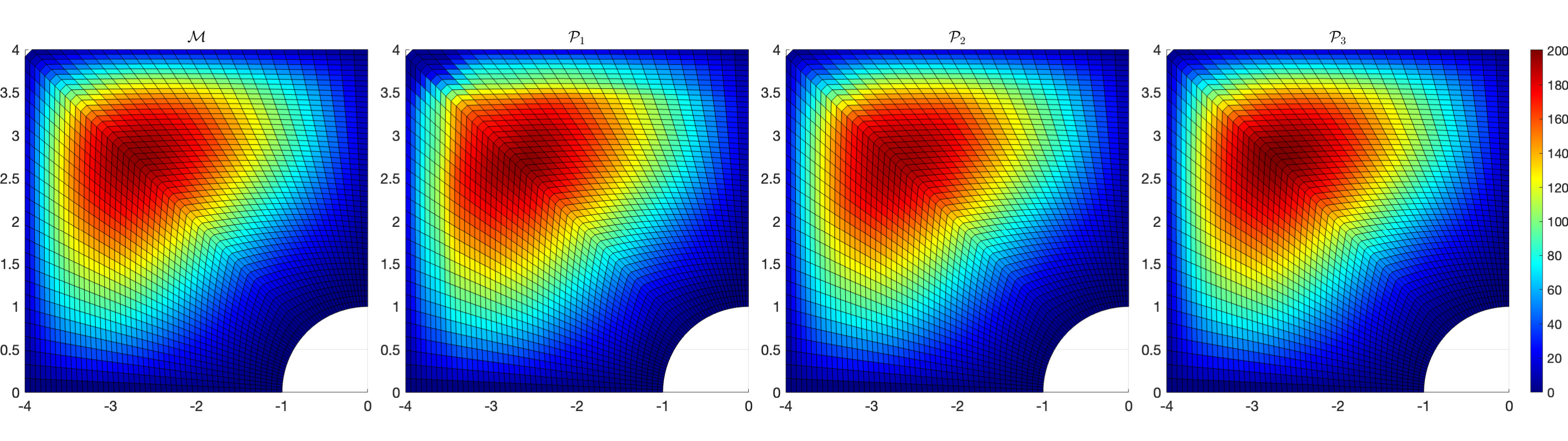}
    \caption{Solution at time $t=2.65$}
    \label{fig: 2D_Laplace_plate_with_hole_dynamics_solution_block_LM_time_2_65_p3_n20}
    \end{subfigure}
    \hfill
    \caption{Numerical solutions for the plate geometry}
    \label{fig: plate_with_hole_dynamics}
\end{figure}

\begin{figure}[htbp]
    \centering
    \includegraphics[scale=0.5]{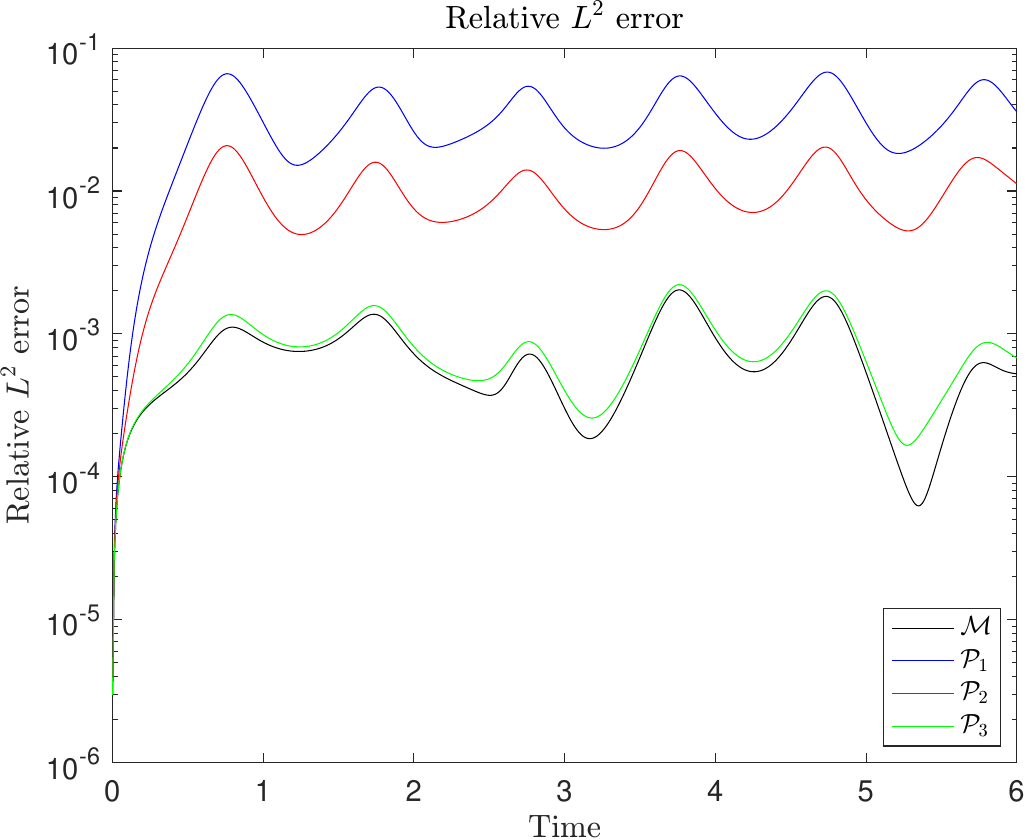}
    \caption{Relative $L^2$ error for the plate geometry}
    \label{fig: 2D_Laplace_plate_with_hole_dynamics_rel_error_LM_p3_n20}
\end{figure}

\begin{figure}[htbp]
    \centering
    \includegraphics[scale=0.5]{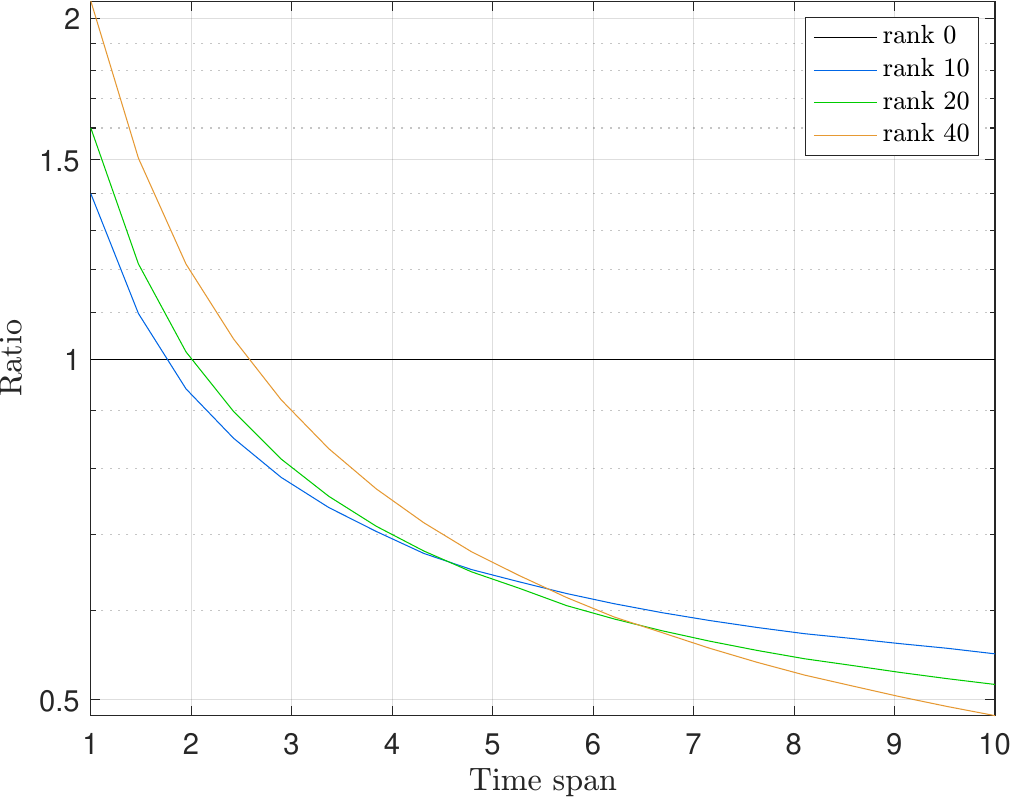}
    \caption{Ratio of number of iterations for the scaled and unscaled methods using $\mathcal{P}_1$}
    \label{fig: 2D_Laplace_plate_with_hole_LR_pert_ratio_p4_n20}
\end{figure}
\end{example}

\begin{remark}
It is important to bear in mind that mass lumping generally depends on the labeling of the parametric directions, which is specific to each software. GeoPDEs, for instance, uses a reverse labeling; i.e. labels $1$, $2$ and $3$ are associated to the $z$, $y$ and $x$ directions, respectively. It might sometimes be advantageous to reorder the system matrices and relabel the parametric directions but we have not experimented with it.
\end{remark}

\begin{example}[Convergence test]
\label{ex: convergence_test_2D}
We now check the convergence of the smallest eigenfrequency for the approximations introduced in Section \ref{se: mass_lumping}. We consider two test cases designed to evaluate the effect of a coefficient function and geometry mapping. The first one is the unit square with a non-separable (but continuous) density function $\rho(x,y)=|\sin(xy)|+x+y+1$ and the second one is the quarter of a plate with a hole (see Figure \ref{fig: plate_with_hole}). Figures \ref{fig: 2D_Laplace_unit_square_coeff_rel_error_smallest_eig_block_ML_p3} and \ref{fig: 2D_Laplace_plate_with_hole_rel_error_smallest_eig_block_ML_p3} show the relative error $\frac{\omega_1-\omega_{h,1}}{\omega_1}$ for the first eigenfrequency and a cubic discretization. Due to the lack of closed form solutions, the reference eigenfrequency $\omega_1$ is a high order approximation computed with the consistent mass on a very fine mesh. The smallest eigenfrequency of $(\mathcal{K}, \mathcal{M})$ converges at the expected rate of $2p$, while the smallest eigenfrequency of $(\mathcal{K}, \mathcal{P}_i)$ converges at a reduced quadratic rate. This observation is in agreement with the well-known fact that the row-sum technique converges at a reduced quadratic rate, independently of the spline order \cite{cottrell2006isogeometric}.

\begin{figure}[htbp]
    \centering
    \begin{subfigure}[t]{0.48\textwidth}
    \centering
    \includegraphics[width=\textwidth]{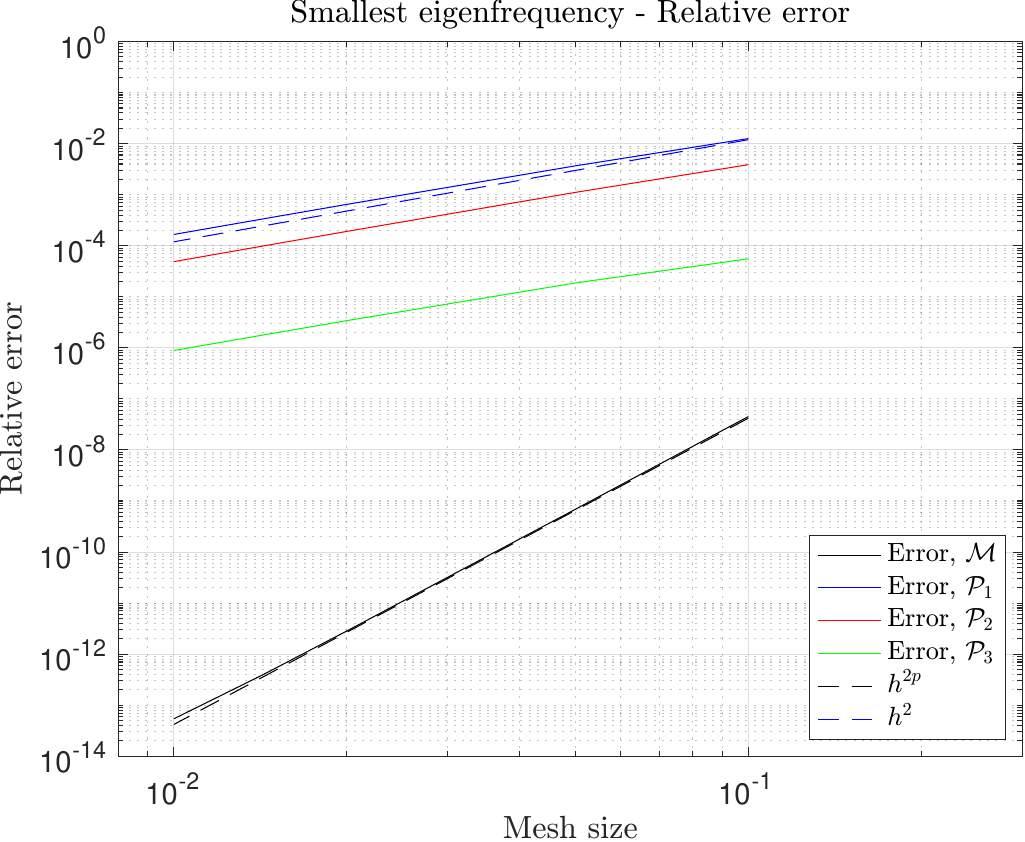}
    \caption{Unit square with a non-separable density function}
    \label{fig: 2D_Laplace_unit_square_coeff_rel_error_smallest_eig_block_ML_p3}
    \end{subfigure}
    \hfill
    \begin{subfigure}[t]{0.48\textwidth}
    \centering
    \includegraphics[width=\textwidth]{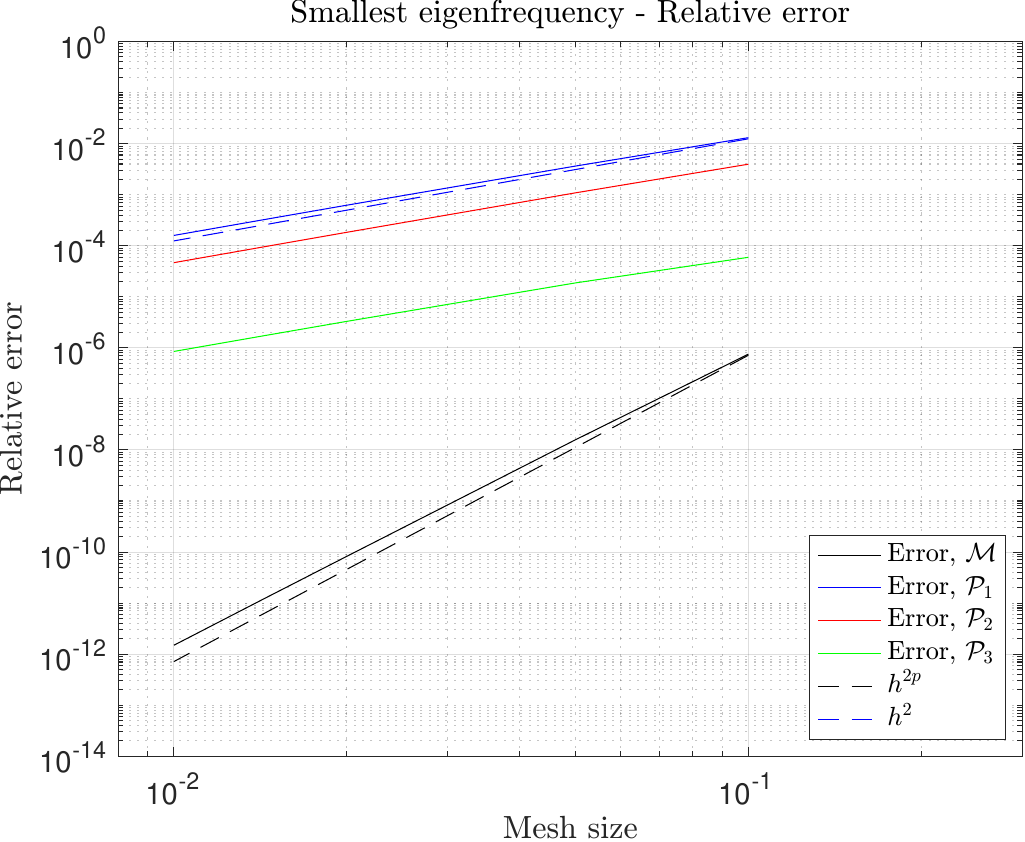}
    \caption{Quarter of a plate with a hole}
    \label{fig: 2D_Laplace_plate_with_hole_rel_error_smallest_eig_block_ML_p3}
    \end{subfigure}
    \hfill
    \caption{Relative error $\frac{\omega_1-\omega_{h,1}}{\omega_1}$}
    \label{fig: rel_error_smallest_eig_block_ML_p3}
\end{figure} 
\end{example}

\begin{example}[Hierarchical mass lumping]
\label{ex: hierarchical_3D_single_patch}
Now we consider a quadratic discretization of the 3D Laplace on the magnet domain shown in Figure \ref{fig: magnet} and test the hierarchical lumped mass matrices described in Section \ref{se: hierarchical_mass_lumping}. Their sparsity pattern is shown in Figure \ref{fig: magnet_sparsity_pattern_hierarchical_ML_n6_p2} together with the consistent mass for $N=12$ subdivisions in each parametric direction. Hierarchical mass lumping leads to a significant reduction of the bandwidth and number of nonzero entries, which drastically speeds up sparse direct solvers. For assessing the performance of mass lumping in explicit dynamics, we solve a sequence of $1000$ linear systems with the consistent mass $\mathcal{M}$ and hierarchical lumped mass matrices $\mathcal{H}_k$ for $k=1,2,3$ on increasingly fine meshes. The solver relies on sparse Cholesky factorizations computed on the reordered matrices using nested dissection. According to Table \ref{tab: computing_times_fixed_time_step}, mass lumping may save several orders of magnitude of computing time, even for relatively small systems. We also noticed that the reordering only slightly reduced the number of nonzero entries in the Cholesky factors and was not the main driver for the enhanced performance. Table \ref{tab: computing_times_fixed_time_span} shows the computing time for a simulation spanning $50$ seconds. While Table \ref{tab: computing_times_fixed_time_step} only accounts for the effect of the linear system solver, Table \ref{tab: computing_times_fixed_time_span} additionally accounts for the saving in the number of time steps thanks to the increase of the critical time step computed with \eqref{eq: CFL_central_difference}.

The impact of hierarchical mass lumping on the accuracy of the solution is (partly) determined by the generalized eigenpairs of $(\mathcal{K},\mathcal{H}_k)$. The eigenvalues are shown in Figure \ref{fig: 3D_Laplace_magnet_hierarchical_ML_p2_n6} alongside those of $(\mathcal{K},\mathcal{M})$ for $N=12$ subdivisions. Interestingly, $\mathcal{H}_2$ seems much more accurate than $\mathcal{H}_3$ and yet barely increases its associated CFL condition. This encouraging result indicates that improved accuracy is possible with only a marginal increase in computational cost. Moreover, similarly to Example \ref{ex: convergence_test_2D}, we verified that hierarchical mass lumping delivered second order convergent eigenvalues.

\begin{figure}[htbp]
     \centering
     \begin{subfigure}[t]{0.48\textwidth}
    \centering
    \includegraphics[width=\textwidth]{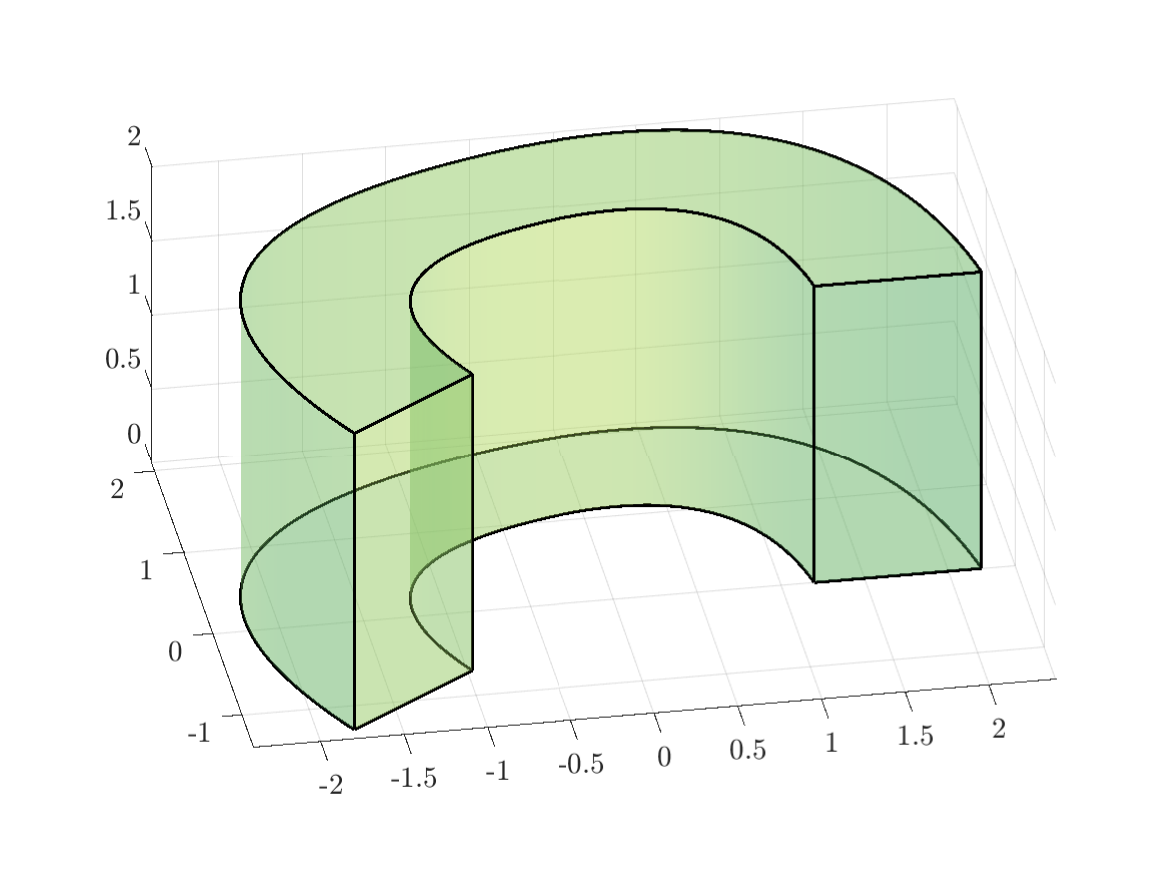}
    \caption{Geometry}
    \label{fig: magnet}
     \end{subfigure}
     \hfill
     \begin{subfigure}[t]{0.48\textwidth}
    \centering
    \includegraphics[width=\textwidth]{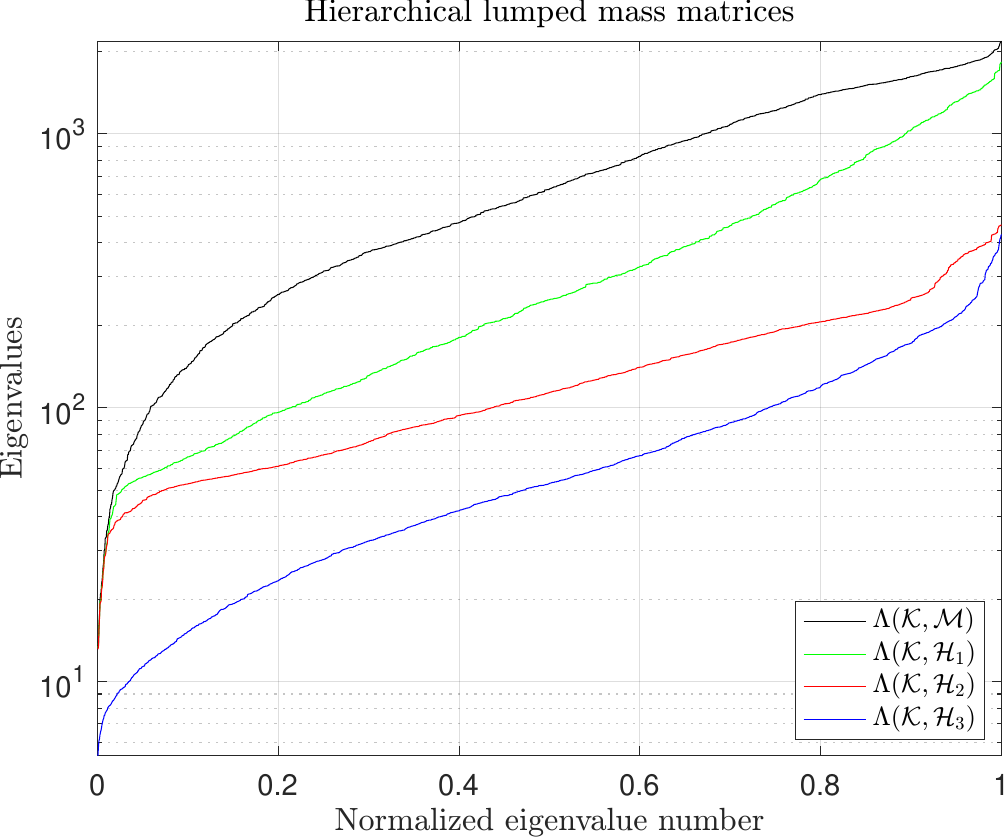}
    \caption{Spectrum}
    \label{fig: 3D_Laplace_magnet_hierarchical_ML_p2_n6}
     \end{subfigure}
     \hfill
    \caption{Magnet}
    \label{fig: 3D_magnet_eig}
\end{figure}

\begin{figure}[htbp]
     \centering
     \begin{subfigure}[t]{0.22\textwidth}
    \centering
    \includegraphics[width=\textwidth]{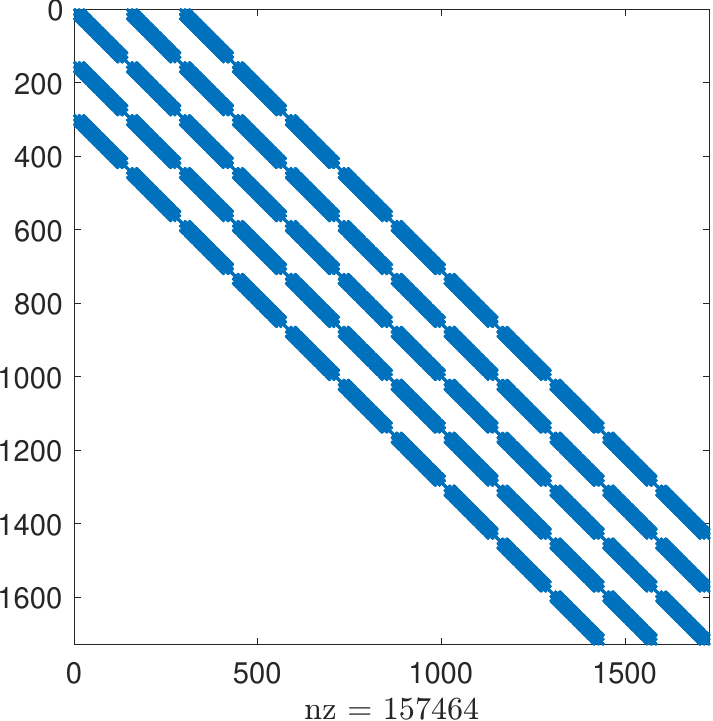}
    \caption{$\mathcal{M}$}
    \label{fig: 3D_Laplace_magnet_sparsity_M_n6_p2}
     \end{subfigure}
     \hfill
     \begin{subfigure}[t]{0.22\textwidth}
    \centering
    \includegraphics[width=\textwidth]{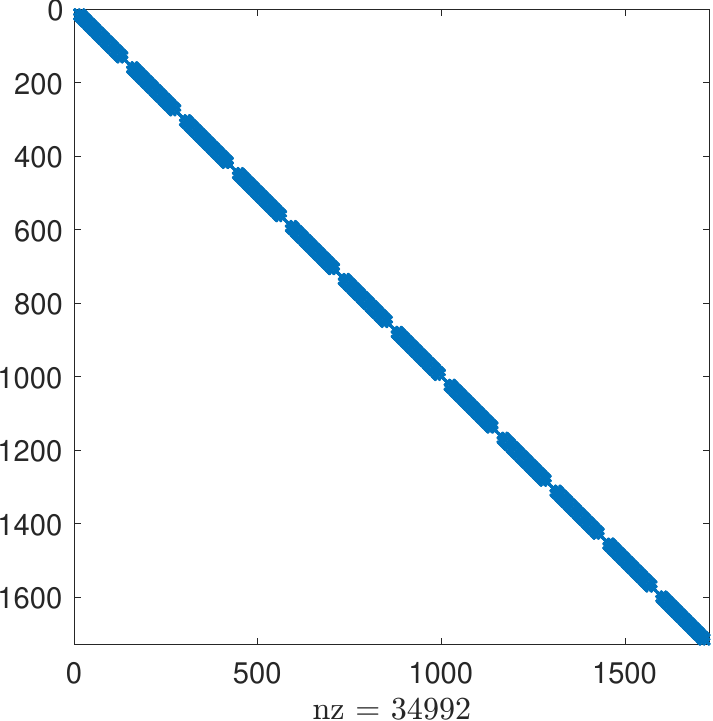}
    \caption{$\mathcal{H}_1$}
    \label{fig: 3D_Laplace_magnet_sparsity_H1_n6_p2}
     \end{subfigure}
     \hfill
    \begin{subfigure}[t]{0.22\textwidth}
    \centering
    \includegraphics[width=\textwidth]{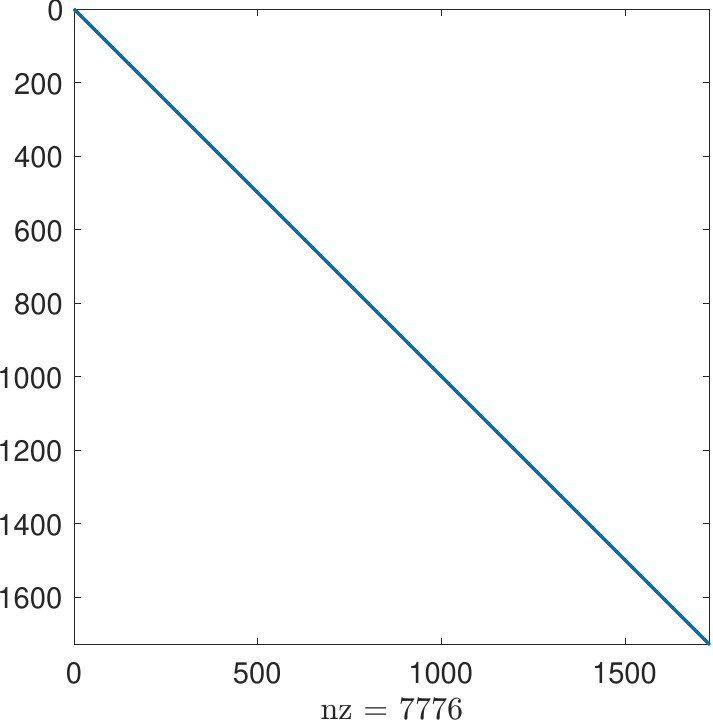}
    \caption{$\mathcal{H}_2$}
    \label{fig: 3D_Laplace_magnet_sparsity_H2_n6_p2}
     \end{subfigure}
     \hfill
    \begin{subfigure}[t]{0.22\textwidth}
    \centering
    \includegraphics[width=\textwidth]{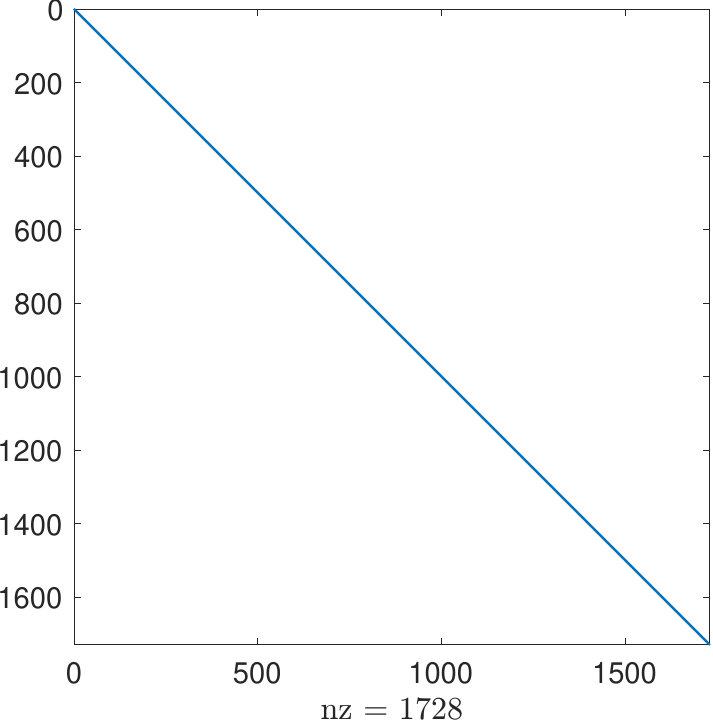}
    \caption{$\mathcal{H}_3$}
    \label{fig: 3D_Laplace_magnet_sparsity_H3_n6_p2}
     \end{subfigure}
     \hfill
    \caption{Sparsity patterns}
    \label{fig: magnet_sparsity_pattern_hierarchical_ML_n6_p2}
\end{figure}

\begin{table}[htbp]
    \centering
    \begin{subtable}[t]{1.0\textwidth}
    \centering
    \begin{tabular}{c c|c c c c|}  
    \cline{3-6}
     & & \multicolumn{4}{|c|}{Time [s]} \\
    \hline
    \multicolumn{1}{|c|}{$N$} & \multicolumn{1}{|c|}{Size} & $\mathcal{M}$ & $\mathcal{H}_1$ &  $\mathcal{H}_2$ & $\mathcal{H}_3$ \\
    \hline
    \multicolumn{1}{|l|}{6} & \multicolumn{1}{|l|}{216} & 0.02 & 0.008 & 0.006 & 0.005 \\
    \multicolumn{1}{|l|}{12} & \multicolumn{1}{|l|}{1728} & 0.56 & 0.07 & 0.03 & 0.02 \\
    \multicolumn{1}{|l|}{18} & \multicolumn{1}{|l|}{5832} & 3.64 & 0.33 & 0.11 & 0.07 \\
    \multicolumn{1}{|l|}{24} & \multicolumn{1}{|l|}{13824} & 13.5 & 0.97 & 0.27 & 0.17 \\
    \multicolumn{1}{|l|}{30} & \multicolumn{1}{|l|}{27000} & 36.38 & 3.22 & 0.53 & 0.33 \\
    \hline
    \end{tabular}
    \caption{Sequence of $1000$ linear systems}
    \label{tab: computing_times_fixed_time_step}
    \end{subtable}    
    \hfill
    \vspace{10pt}
    \begin{subtable}[t]{1.0\textwidth}
    \centering
    \begin{tabular}{c c|c c c c|c c c c|}  
    \cline{3-10}
     & & \multicolumn{4}{|c|}{Time [s]} & \multicolumn{4}{|c|}{Number of time steps} \\
    \hline
    \multicolumn{1}{|c|}{$N$} & \multicolumn{1}{|c|}{Size} & $\mathcal{M}$ & $\mathcal{H}_1$ &  $\mathcal{H}_2$ & $\mathcal{H}_3$ & $\mathcal{M}$ & $\mathcal{H}_1$ &  $\mathcal{H}_2$ & $\mathcal{H}_3$ \\
    \hline
    \multicolumn{1}{|l|}{6} & \multicolumn{1}{|l|}{216} & 0.012 & 0.004 & 0.002 & 0.001 & 578 & 530 & 270 & 257 \\
    \multicolumn{1}{|l|}{12} & \multicolumn{1}{|l|}{1728} & 0.67 & 0.08 & 0.02 & 0.01 & 1167 & 1069 & 539 & 518 \\
    \multicolumn{1}{|l|}{18} & \multicolumn{1}{|l|}{5832} & 6.37 & 0.54 & 0.09 & 0.06 & 1756 & 1608 & 814 & 785 \\
    \multicolumn{1}{|l|}{24} & \multicolumn{1}{|l|}{13824} & 31.7 & 2.06 & 0.29 & 0.18 & 2345 & 2148 & 1090 & 1052 \\
    \multicolumn{1}{|l|}{30} & \multicolumn{1}{|l|}{27000} & 106.3 & 6.46 & 0.73 & 0.45 & 2935 & 2689 & 1367 & 1320 \\
    \hline
    \end{tabular}
    \caption{Simulation of $50$ seconds}
    \label{tab: computing_times_fixed_time_span}
    \end{subtable}
    \hfill
    \caption{System size, computing times (in seconds) and number of time steps for the consistent mass and hierarchical lumped mass matrices. $N=h^{-1}$ denotes the number of subdivisions in each parametric direction.}
    \label{fig: computing_times_hierarchical}
\end{table}

We consider now a more realistic situation by solving a linear elasticity problem on the magnet domain of Figure \ref{fig: magnet}. Homogeneous Dirichlet boundary conditions are prescribed on its base and homogeneous Neumann boundary conditions on its side faces. A slowly oscillating traction force, given by
\begin{equation*}
    \bm{\tau}(\mathbf{x},t)=
    \begin{pmatrix}
        0 \\
        0 \\
        -q\sin(\frac{8 \pi t}{T})
    \end{pmatrix}
\end{equation*}
is applied on its top face, where $q=20 \ \text{MPa}$ is the pressure's magnitude and $T=10^{-2} \ \text{s}$ is the final time. We assume the magnet is made out of steel (elastic modulus $E=207 \ \text{GPa}$, Poisson's ratio $\nu = 0.3$ and density $\rho = 7800 \ \text{kg/m\textsuperscript{3}}$). The problem is discretized in space using quadratic B-splines with $N=6$ subdivisions in each parametric direction and approximated in time using the central difference method. The critical time step, given by \eqref{eq: CFL_central_difference}, and multiplied by a safeguarding factor of $0.85$ leads to $2305$, $2213$, $1107$ and $1057$ time steps for $\mathcal{M}$, $\mathcal{H}_1$, $\mathcal{H}_2$ and $\mathcal{H}_3$, respectively. Figure \ref{fig: 3D_Elasticity_magnet_dynamics_error_CD_N2305_p2_n6} shows the $L^2$ error committed with respect to the solution for the consistent mass. In this figure, we have used the same step sizes to ensure valid comparison of the discrete solutions. Whereas the row-sum lumped mass matrix $\mathcal{H}_3$ induces a significant error, the hierarchical lumped mass matrices $\mathcal{H}_2$ and $\mathcal{H}_1$ provide a much better approximation. Moreover, the computing times for solving linear systems scaled similarly as those reported in Tables \ref{tab: computing_times_fixed_time_step} and \ref{tab: computing_times_fixed_time_span} for $N=6$. These two observations make a compelling case for $\mathcal{H}_2$ as it (nearly) provides the accuracy of $\mathcal{H}_1$ but at the cost of $\mathcal{H}_3$. In practice, the industry's concern is speed more than accuracy and we would advise choosing the hierarchical level $k$ (or block bandwidth $i$) based on the estimated slowdown with respect to the row-sum technique.

\begin{figure}[htbp]
    \centering
    \includegraphics[scale=0.5]{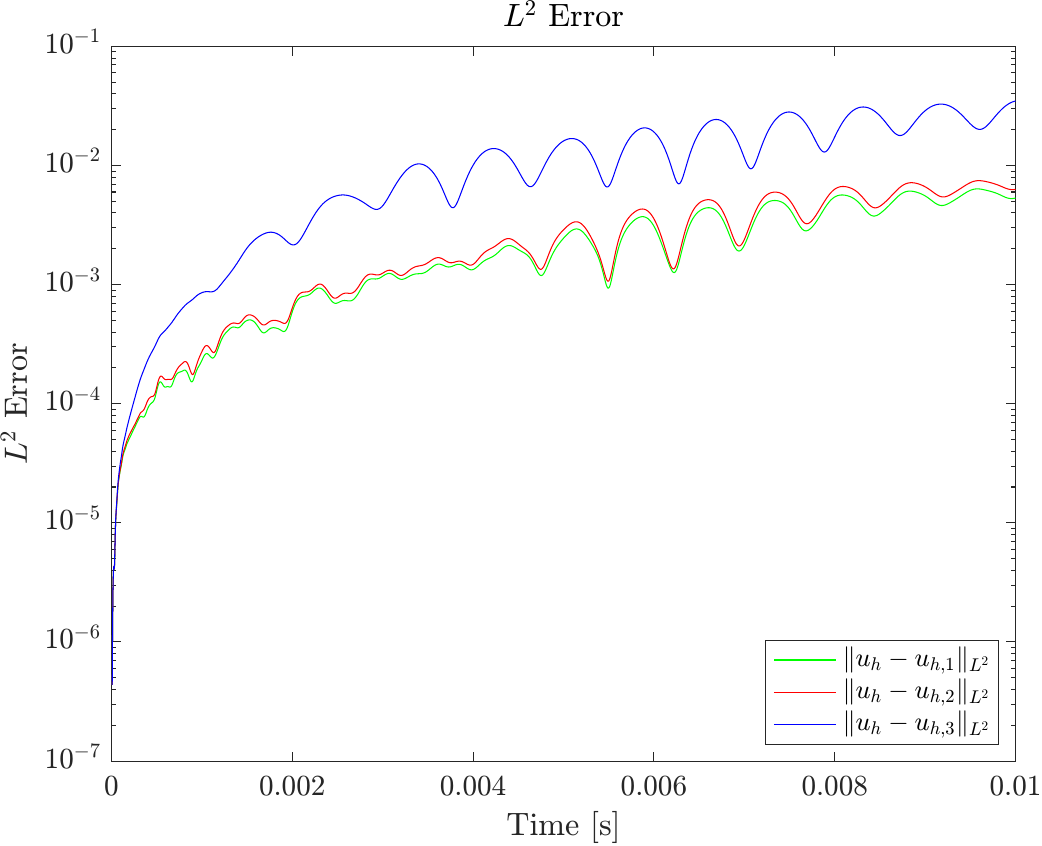}
    \caption{$L^2$ error with respect to the solution for the consistent mass}
    \label{fig: 3D_Elasticity_magnet_dynamics_error_CD_N2305_p2_n6}
\end{figure}
\end{example}

\subsection{Multipatch geometries}

\begin{example}
We consider a cubic discretization of the Laplace on the quarter of a plate with a hole, as shown in Figure \ref{fig: plate_with_hole}, and split into two patches to remove sources of singularity. The sparsity patterns of the consistent mass and lumped mass matrices $\mathcal{P}_i$ for $i=1,2$ are shown in Figure \ref{fig: sparsity_patterns_mp} for $15$ subdivisions in each parametric direction and each patch. Mass lumping for multipatch problems does not completely remove the coupling but instead focuses on reducing the bandwidth of the diagonal blocks in order to easily form the Schur complement. While increasing the bandwidth generally improves the accuracy, forming the Schur complement also typically takes a heavier toll. This tradeoff must be carefully balanced on a case-by-case basis. As explained in Section \ref{se: outlier_removal}, we try scaling the local stiffness matrices before assembling them into a global matrix. For clarity, we denote $\bar{\mathcal{K}}_{\textrm{loc}}$ the locally scaled stiffness matrix. Given its heuristic nature, this scaling strategy might affect the smallest eigenvalues. Nevertheless, Figure \ref{fig: 2D_Laplace_plate_with_hole_P2_LR_pert_p3_n15_r20_mp}, obtained for a rank 20 patchwise scaling and a tolerance of $10^{-3}$, reveals that this effect is very mild. By comparing the results with a rank 40 global mass scaling, we notice that the local scaling method removes fewer outliers. This is expected given that it cannot remove interior outliers, arising from the $C^0$ coupling of patch interfaces. The convergence test carried out in Figure \ref{fig: 2D_Laplace_plate_with_hole_P2_LR_rel_error_smallest_eig_p3_mp} further indicates that the smallest eigenfrequency converges at a second order rate, as it does in the purely lumped mass case (using $P_1$). For the sake of clarity, we have only reported the results for $\mathcal{P}_2$ but they seem to hold more generally.

\begin{figure}[htbp]
    \centering
    \begin{subfigure}[t]{0.32\textwidth}
    \centering
    \includegraphics[width=\textwidth]{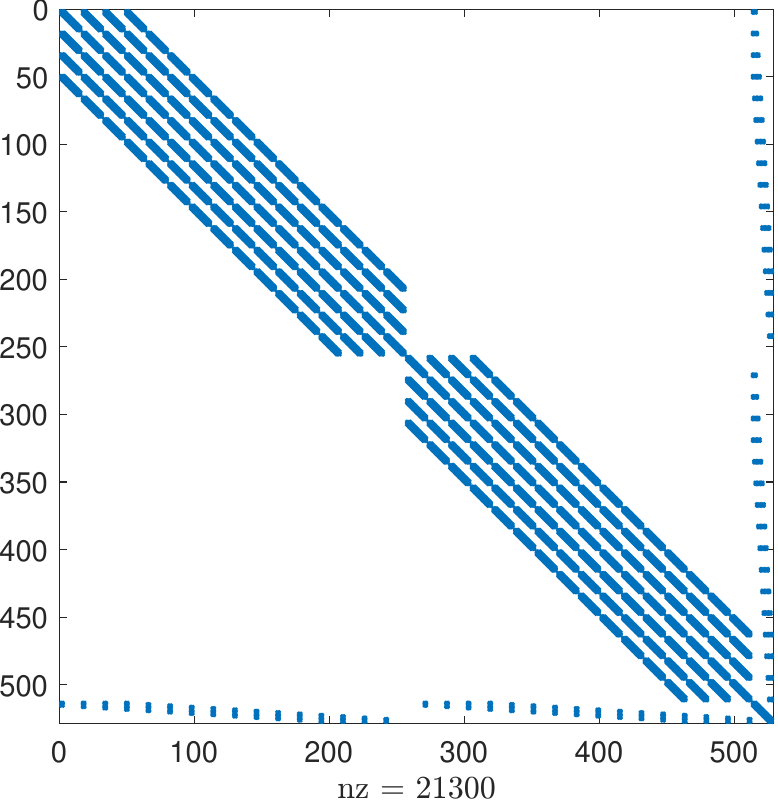}
    \caption{$\mathcal{M}$}
    \label{fig: 2D_Laplace_plate_with_hole_sparsity_M_p3_n15_mp}
    \end{subfigure}
    \hfill
    \begin{subfigure}[t]{0.32\textwidth}
    \centering
    \includegraphics[width=\textwidth]{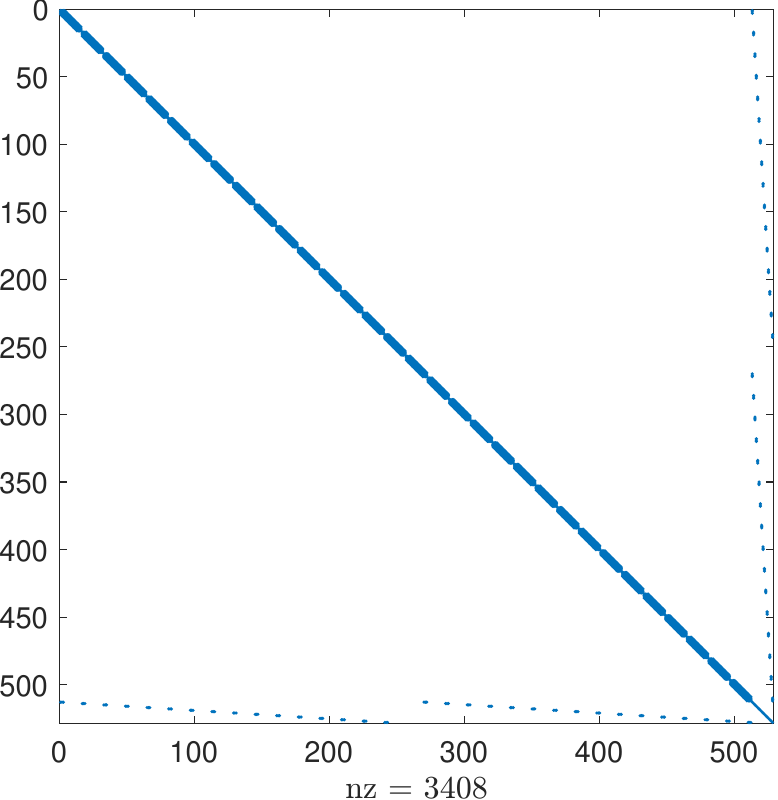}
    \caption{$\mathcal{P}_1$}
    \label{fig: 2D_Laplace_plate_with_hole_sparsity_P1_p3_n15_mp}
    \end{subfigure}    
    \hfill
    \begin{subfigure}[t]{0.32\textwidth}
    \centering
    \includegraphics[width=\textwidth]{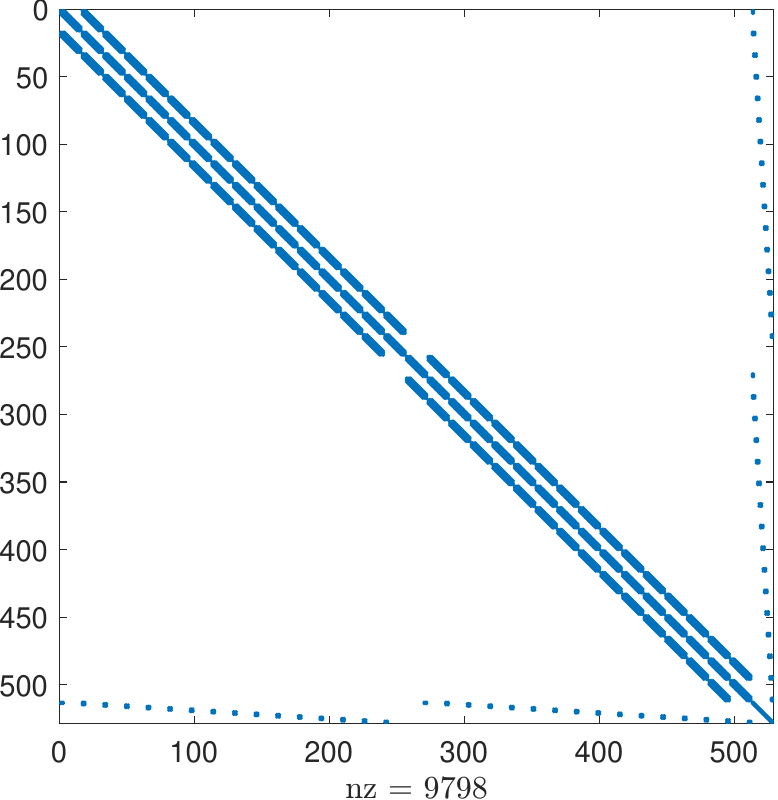}
    \caption{$\mathcal{P}_2$}
    \label{fig: 2D_Laplace_plate_with_hole_sparsity_P2_p3_n15_mp}
    \end{subfigure}    
    \hfill
    \caption{Sparsity patterns}
    \label{fig: sparsity_patterns_mp}
\end{figure}

\begin{figure}[htbp]
    \centering
    \begin{subfigure}[t]{0.48\textwidth}
    \centering
    \includegraphics[width=\textwidth]{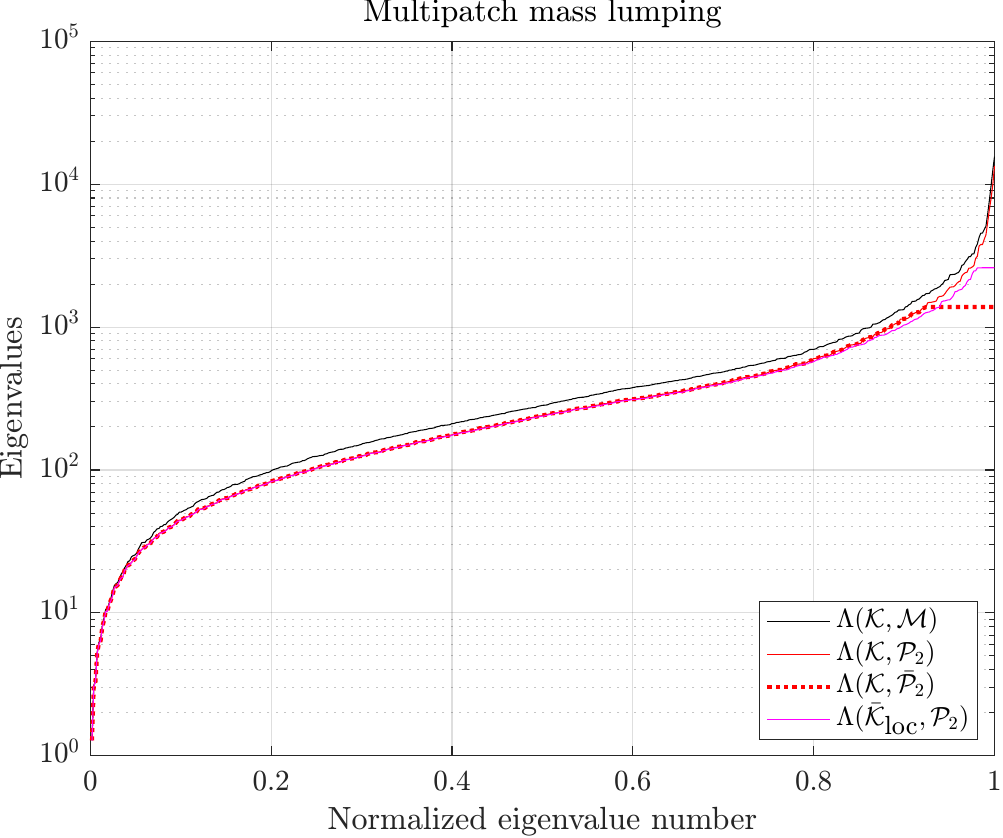}
    \caption{Spectrum}
    \label{fig: 2D_Laplace_plate_with_hole_P2_LR_pert_p3_n15_r20_mp}
    \end{subfigure}    
    \hfill
    \begin{subfigure}[t]{0.48\textwidth}
    \centering
    \includegraphics[width=\textwidth]{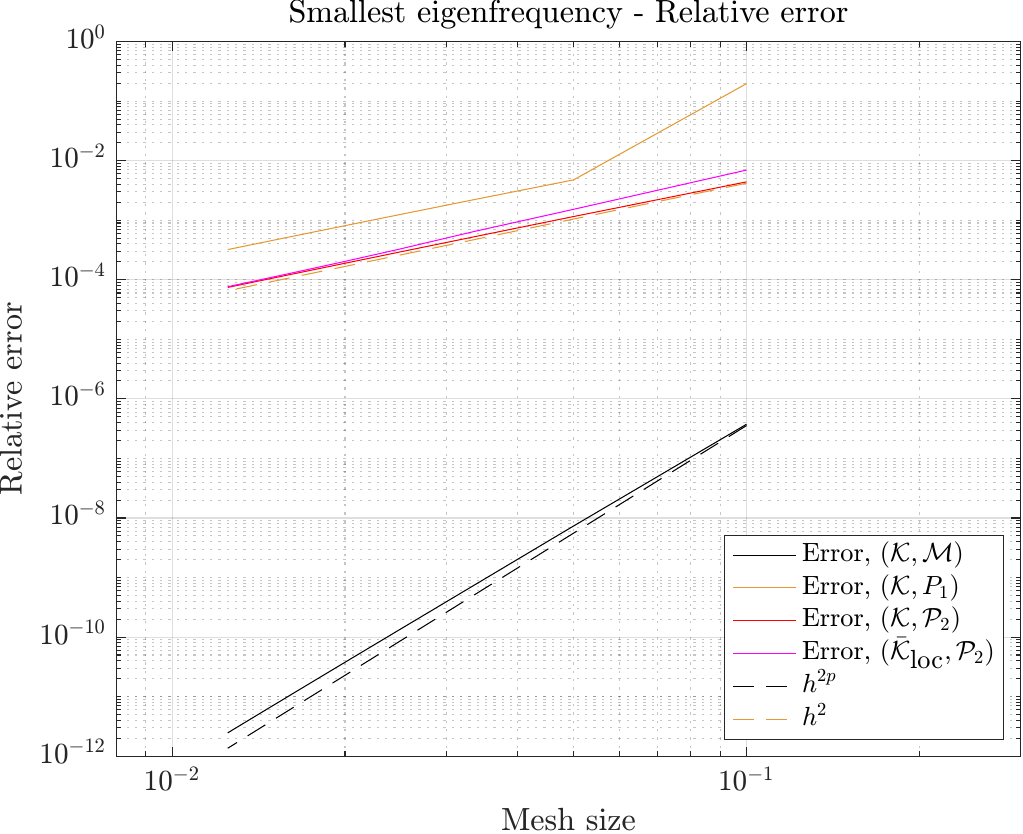}
    \caption{Relative error $\frac{\omega_1-\omega_{h,1}}{\omega_1}$ for the consistent mass and lumped mass with and without scaling}
    \label{fig: 2D_Laplace_plate_with_hole_P2_LR_rel_error_smallest_eig_p3_mp}
    \end{subfigure}
    \hfill
    \caption{Spectrum and convergence test}
    \label{fig: multipatch_geometries}
\end{figure} 
\end{example}

\begin{example}
This experiment aims at better understanding how patch configurations affect the scaling outcome. Our benchmarks consist of the unit square split into a $4 \times 4$ grid of patches and a rectangle of length $4$ and width $0.25$ split into a $1 \times 16$ grid. Both configurations are discretized using quadratic B-splines and $N=8$ subdivisions in each direction and each patch. The size and number of patches of both configurations are the same (16 patches of size $0.25 \times 0.25$) but are coupled very differently. Although homogeneous Dirichlet boundary conditions are prescribed along the entire boundary in both cases, the first one features a number of interior patches and consequently more degrees of freedom. Figures \ref{fig: 2D_Laplace_1x1_grid_4x4_n8_p2_rg160_rl10} and \ref{fig: 2D_Laplace_4x0.25_grid_1x16_n8_p2_rg160_rl10} compare for both configurations a local patchwise scaling of rank $10$ with a global scaling of rank $160$. Surprisingly, our local scaling strategy performs poorly even for the second, weakly coupled, configuration. In both cases, the global scaling is again more efficient at removing outliers. However, this strategy is also more expensive than our local (sequential) patchwise scaling, as shown in Tables \ref{tab: computing_times_global_vs_local_1x1_grid_4x4_p2_rg160_rl10} and \ref{tab: computing_times_global_vs_local_4x0.25_grid_1x16_p2_rg160_rl10} for increasingly fine meshes. The global scaling strategy is generally slightly faster on the second configuration, probably owing to the smaller system size. Nevertheless, the outcome generally depends on the patch configuration, which affects the spectral properties of the system matrices and consequently the convergence speed of the Lanczos method. Therefore, a direct comparison of global and local scaling is not trivial, even while neglecting potential for parallelism. A more thorough study is needed before drawing definite conclusions. In the standard tensor product case, mesh refinement and degree elevation widens the gap between ``outlier'' and ``regular'' eigenvalues and benefits Lanczos while in these multipatch examples, the number of Lanczos iterations (for a tolerance of $10^{-3}$) remains roughly constant beyond a certain system size and may occasionally exceed the system size (especially for small systems) due to restarting. 

\begin{figure}[htbp]
    \centering
    \begin{subfigure}[t]{0.48\textwidth}
    \centering
    \includegraphics[width=\textwidth]{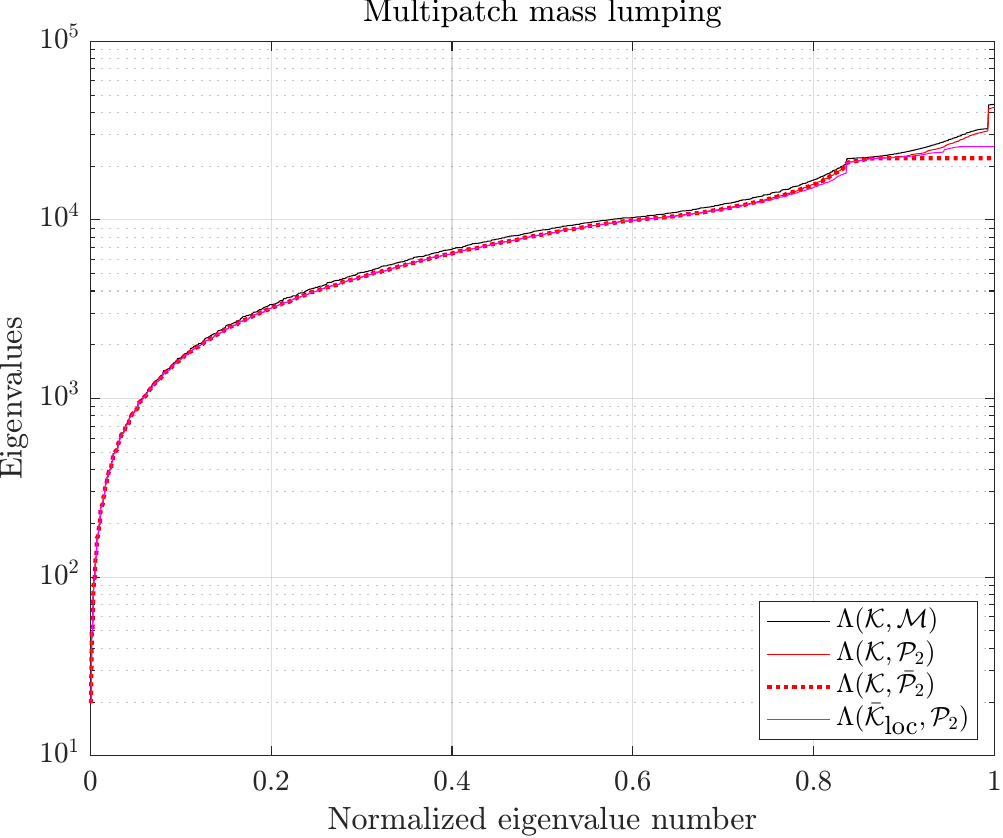}
    \caption{Unit square split in a $4 \times 4$ grid of patches}
    \label{fig: 2D_Laplace_1x1_grid_4x4_n8_p2_rg160_rl10}
    \end{subfigure}    
    \hfill
    \begin{subfigure}[t]{0.48\textwidth}
    \centering
    \includegraphics[width=\textwidth]{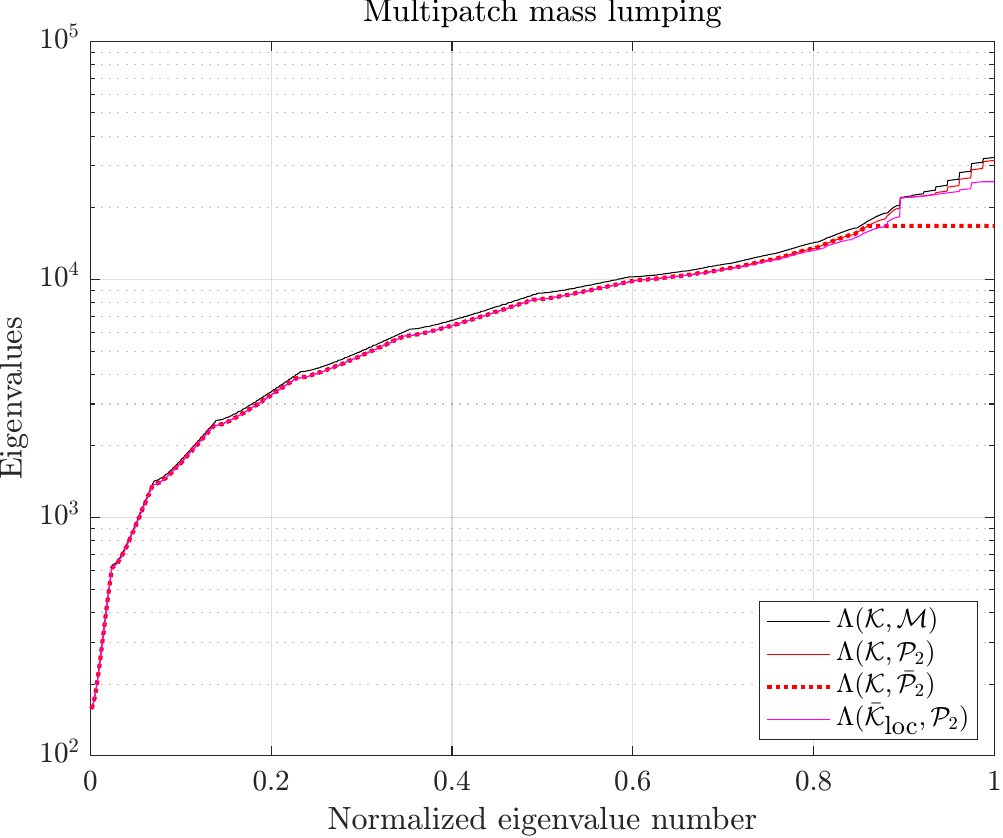}
    \caption{Rectangle of length $4$ and width $0.25$ split into a $1 \times 16$ grid of patches}
    \label{fig: 2D_Laplace_4x0.25_grid_1x16_n8_p2_rg160_rl10}
    \end{subfigure}
    \hfill
    \caption{Comparison of global and local scaling for different patch configurations}
    \label{fig: multipatch_configurations}
\end{figure}

\begin{table}[htbp]
    \centering
    \begin{subtable}[t]{1.0\textwidth}
    \centering
    \begin{tabular}{c c|c c|c c|}  
    \cline{3-6}
     & & \multicolumn{2}{|c|}{Time [s]} & \multicolumn{2}{|c|}{Iterations} \\
    \hline
    \multicolumn{1}{|c|}{$N$} & \multicolumn{1}{|c|}{Size} & Global & Local & Global & Local \\
    \hline
    \multicolumn{1}{|l|}{4} & \multicolumn{1}{|l|}{361} & 0.05 & 0.03 & 480 & 160 \\
    \multicolumn{1}{|l|}{8} & \multicolumn{1}{|l|}{1225} & 0.18 & 0.03 & 832 & 80 \\
    \multicolumn{1}{|l|}{12} & \multicolumn{1}{|l|}{2601} & 0.45 & 0.04 & 800 & 112 \\
    \multicolumn{1}{|l|}{16} & \multicolumn{1}{|l|}{4489} & 0.76 & 0.09 & 800 & 228 \\
    \multicolumn{1}{|l|}{20} & \multicolumn{1}{|l|}{6889} & 1.21 & 0.12 & 832 & 145 \\
    \hline
    \end{tabular}
    \caption{Unit square split into a $4 \times 4$ grid of patches.}
    \label{tab: computing_times_global_vs_local_1x1_grid_4x4_p2_rg160_rl10}
    \end{subtable}    
    \hfill
    \vspace{10pt}
    \begin{subtable}[t]{1.0\textwidth}
    \centering
    \begin{tabular}{c c|c c|c c|}  
    \cline{3-6}
     & & \multicolumn{2}{|c|}{Time [s]} & \multicolumn{2}{|c|}{Iterations} \\
    \hline
    \multicolumn{1}{|c|}{$N$} & \multicolumn{1}{|c|}{Size} & Global & Local & Global & Local \\
    \hline
    \multicolumn{1}{|l|}{4} & \multicolumn{1}{|l|}{316} & 0.02 & 0.02 & 316 & 157 \\
    \multicolumn{1}{|l|}{8} & \multicolumn{1}{|l|}{1144} & 0.10 & 0.04 & 480 & 81 \\
    \multicolumn{1}{|l|}{12} & \multicolumn{1}{|l|}{2484} & 0.37 & 0.05 & 960 & 112 \\
    \multicolumn{1}{|l|}{16} & \multicolumn{1}{|l|}{4336} & 0.67 & 0.09 & 1120 & 236 \\
    \multicolumn{1}{|l|}{20} & \multicolumn{1}{|l|}{6700} & 1.29 & 0.11 & 1088 & 147 \\
    \hline
    \end{tabular}
    \caption{Rectangle of length $4$ and width $0.25$ split into a $1 \times 16$ grid of patches.}
    \label{tab: computing_times_global_vs_local_4x0.25_grid_1x16_p2_rg160_rl10}
    \end{subtable}
    \hfill
    \caption{System size, computing times (in seconds) and number of Lanczos iterations for the global and local (sequential) scaling strategies. Computing times are averaged out over $5$ independent runs. Local iteration counts are averaged out across patches and rounded to the nearest integer. $N=h^{-1}$ denotes the number of subdivisions in each parametric direction and each patch.}
    \label{fig: computing_times_global_vs_local_scaling}
\end{table}
\end{example}

\begin{example}
As we have noted in Section \ref{se: multipatch_mass_lumping}, any suitable mass lumping technique may be applied patchwise before assembling the global multipatch lumped mass matrix. In dimension $d \geq 3$, hierarchical mass lumping stands out as the natural candidate. When the context is clear, we do not distinguish the single-patch matrices from their multipatch counterpart. We solve the Laplace eigenvalue problem on the $3$-patch twisted box geometry shown in Figure \ref{fig: twisted_box}, discretized with quadratic B-splines and $N=12$ subdivisions in each parametric direction and each patch. The spectra of $(\mathcal{K},\mathcal{M})$ and $(\mathcal{K},\mathcal{H}_k)$ for $k=1,2,3$ are shown in Figure \ref{fig: 3D_Laplace_twisted_box_hierarchical_ML_p2_n6_mp} and confirm the improved accuracy on the low-frequency part of the spectrum with respect to the row-sum technique (i.e. $\mathcal{H}_3$). Similarly to Example \ref{ex: hierarchical_3D_single_patch}, hierarchical mass lumping yields a drastic reduction of the bandwidth and number of nonzero entries, as shown in Figure \ref{fig: twisted_box_sparsity_pattern_hierarchical_ML_n6_p2_mp}.

\begin{figure}[htbp]
     \centering
     \begin{subfigure}[t]{0.48\textwidth}
    \centering
    \includegraphics[width=\textwidth]{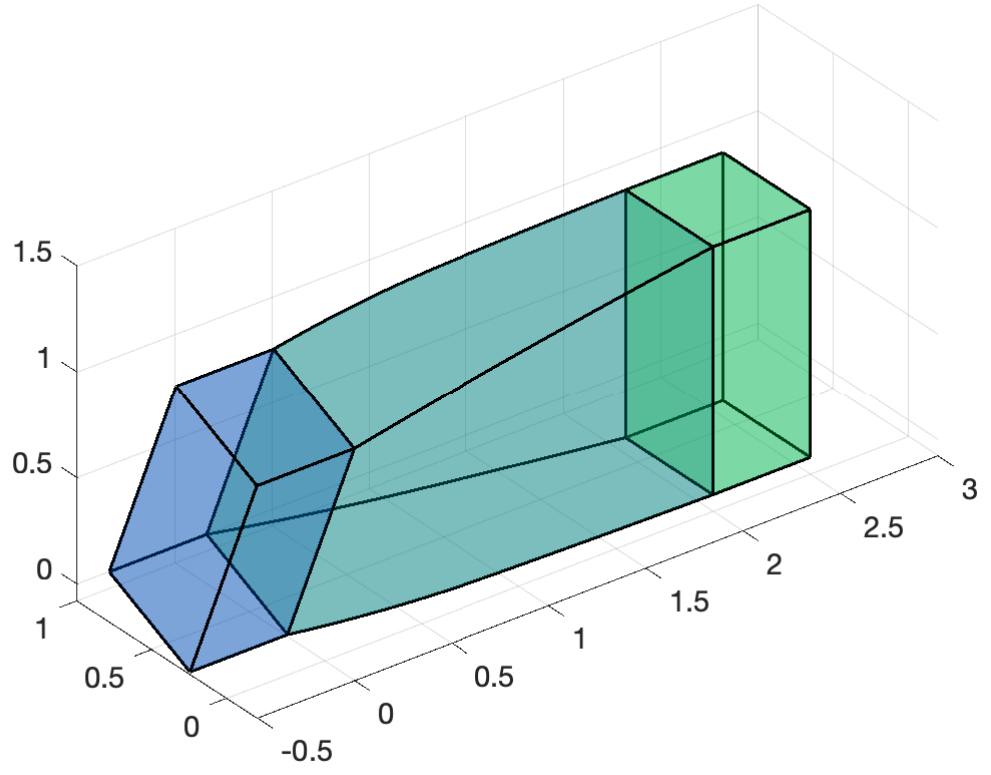}
    \caption{Geometry. Different colors identify different patches}
    \label{fig: twisted_box}
     \end{subfigure}
     \hfill
     \begin{subfigure}[t]{0.48\textwidth}
    \centering
    \includegraphics[width=\textwidth]{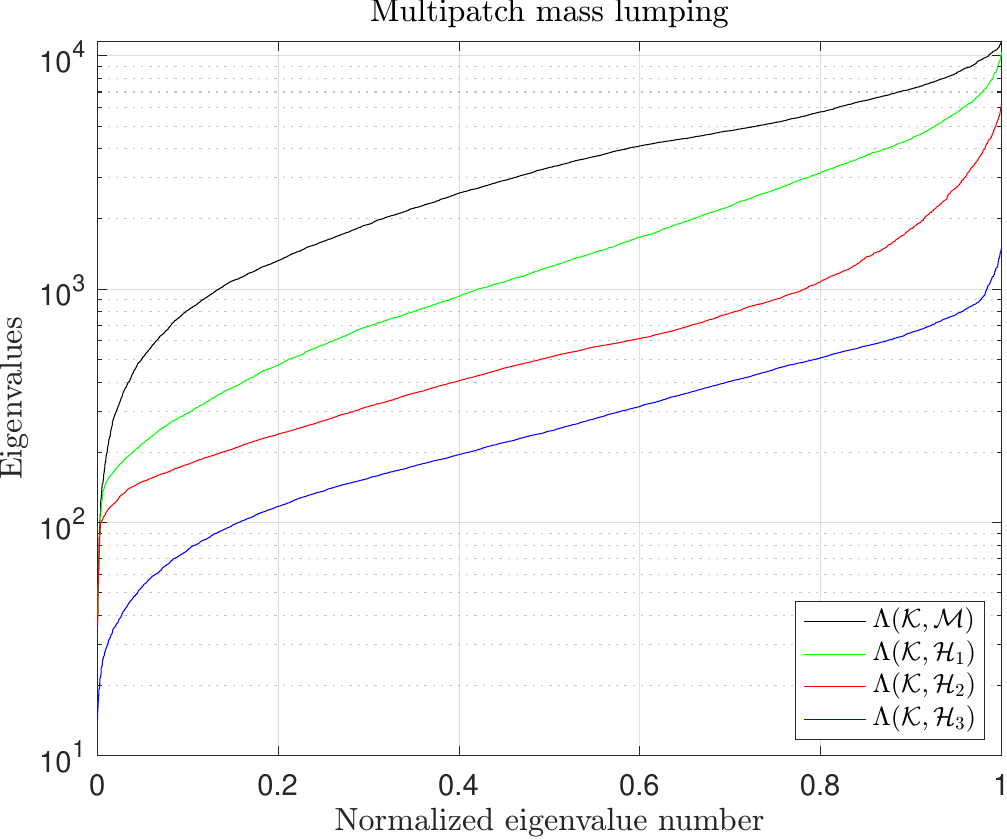}
    \caption{Spectrum}
    \label{fig: 3D_Laplace_twisted_box_hierarchical_ML_p2_n6_mp}
     \end{subfigure}
     \hfill
    \caption{Twisted box}
    \label{fig: 3D_multipatch_twisted_box}
\end{figure}

\begin{figure}[htbp]
     \centering
     \begin{subfigure}[t]{0.22\textwidth}
    \centering
    \includegraphics[width=\textwidth]{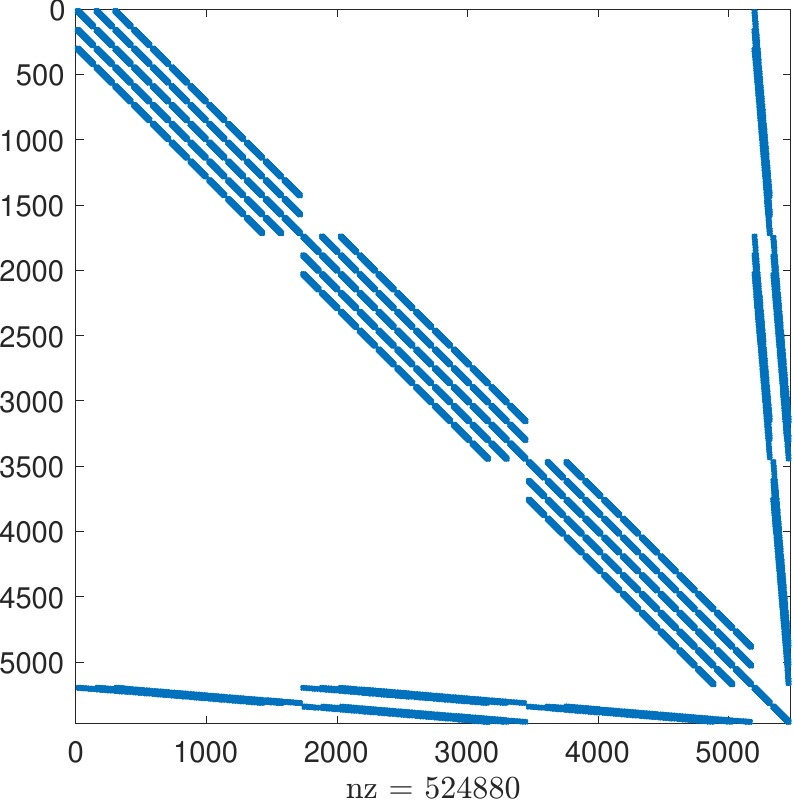}
    \caption{$\mathcal{M}$}
    \label{fig: 3D_Laplace_twisted_box_sparsity_M_n6_p2_mp}
     \end{subfigure}
     \hfill
     \begin{subfigure}[t]{0.22\textwidth}
    \centering
    \includegraphics[width=\textwidth]{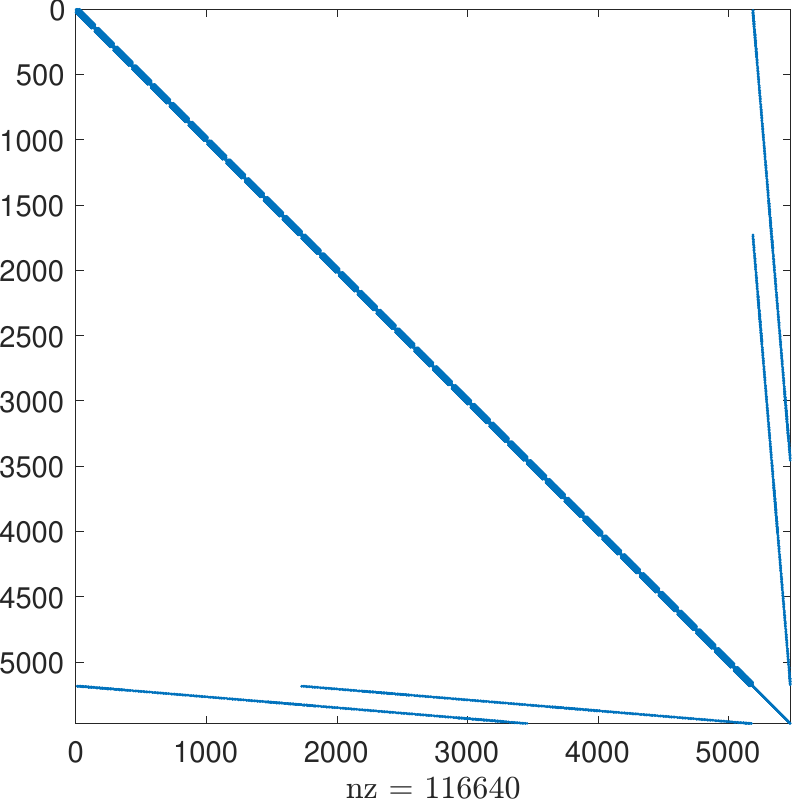}
    \caption{$\mathcal{H}_1$}
    \label{fig: 3D_Laplace_twisted_box_sparsity_H1_n6_p2_mp}
     \end{subfigure}
     \hfill
    \begin{subfigure}[t]{0.22\textwidth}
    \centering
    \includegraphics[width=\textwidth]{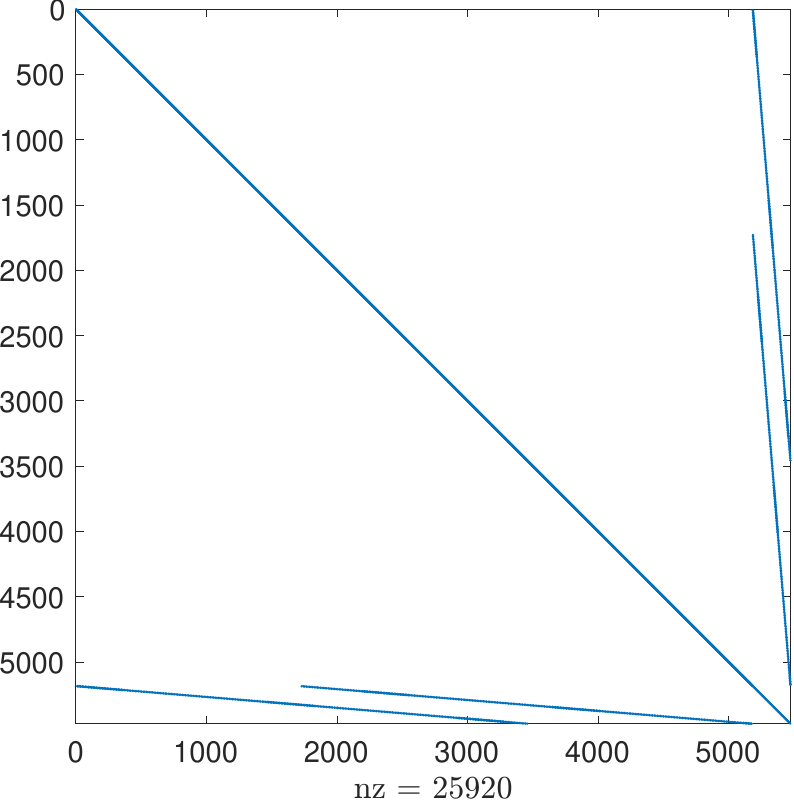}
    \caption{$\mathcal{H}_2$}
    \label{fig: 3D_Laplace_twisted_box_sparsity_H2_n6_p2_mp}
     \end{subfigure}
     \hfill
    \begin{subfigure}[t]{0.22\textwidth}
    \centering
    \includegraphics[width=\textwidth]{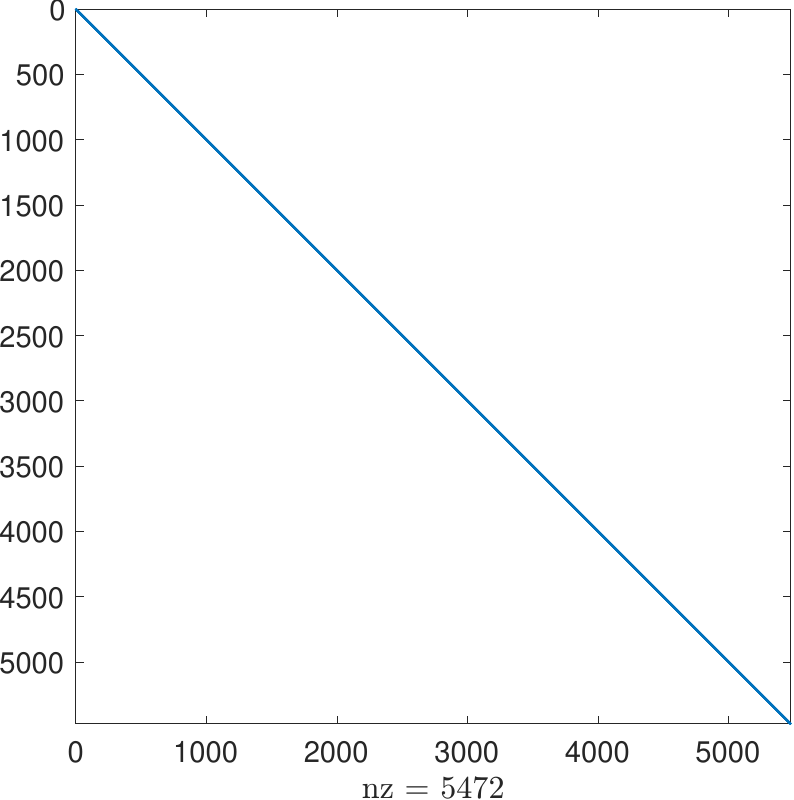}
    \caption{$\mathcal{H}_3$}
    \label{fig: 3D_Laplace_twisted_box_sparsity_H3_n6_p2_mp}
     \end{subfigure}
     \hfill
    \caption{Sparsity patterns}
    \label{fig: twisted_box_sparsity_pattern_hierarchical_ML_n6_p2_mp}
\end{figure}
\end{example}

\subsection{Trimmed geometries}

\begin{example}
In this example, the Laplace eigenvalue problem with pure Neumann boundary conditions is solved on a shifted and rotated square. The physical domain is embedded in a regular unit square discretized with $N=80$ subdivisions in each direction and trimmed to fit the domain's boundaries. Figure \ref{fig: 2D_rotating_square_geometry} provides an illustration for a coarser mesh with only $20$ subdivisions in each direction. We test our block lumping strategies, combined with a rank $80$ scaling, and compare them to the standard row-sum technique for quadratic and cubic order spline discretizations. System matrices for trimmed domains generally do not feature a Kronecker structure, not even in their sparsity. However, by padding them with zero entries, they may be embedded in larger matrices whose structure allows applying the block lumping operator. Once the operator has been applied, they are trimmed back to their original size by removing the artificial rows and columns. Thus, our lumping strategies also apply to trimmed geometries. The resulting system size was $1244$ and $1321$ for $p=2$ and $p=3$, respectively. The scaled and unscaled spectra of $(\mathcal{K}, \mathcal{M})$, $(\mathcal{K}, P_1)$ and $(\mathcal{K}, \mathcal{P}_i)$ for $i=1,2$ is reported in Figures \ref{fig: 2D_Laplace_rotating_square_block_LM_p2} and \ref{fig: 2D_Laplace_rotating_square_block_LM_p3} for $p=2$ and $p=3$, respectively, after discarding the zero eigenvalue. It is well-known that small trimmed elements heavily deteriorate the conditioning of the system matrices (among other issues) \cite{de2023stability,de2017condition}. Beyond a certain threshold, the computations altogether are no longer accurate. Our rotation angle, while still being unfavorable, avoids such extreme situations. Nevertheless, one should exercise caution when computing eigenvalues of matrix pairs $(A,B)$ with a heavily ill-conditioned matrix $B$. In our numerical experiments, we have computed the eigenvalues of the equivalent matrix pair $(DAD,DBD)$, where $D=\diag(d_1,\dots,d_n)$ with $d_i=1/\sqrt{b_{ii}}$ for $i=1,\dots,n$ is a Jacobi preconditioner for $B$ \cite{de2017condition}. In the context of trimming, the conditioning of $DBD$ is generally orders of magnitude better than the one of $B$, which improves the stability of eigensolvers.

As shown in Figures \ref{fig: 2D_Laplace_rotating_square_block_LM_p2} and \ref{fig: 2D_Laplace_rotating_square_block_LM_p3}, the outlier eigenvalues for this problem are particularly pronounced, as evidenced by the sharp change of curvature in the spectrum. Although the row-sum technique is most effective at improving the CFL condition, it is also most inaccurate, even for moderate spline degrees, as highlighted in \cite{radtke2024analysis}. Our block lumping method drastically improves the accuracy but also leads to more restrictive CFL conditions. However, combining it with scaling strongly mitigates this effect. The fast increase of the outlier eigenvalues also speeds up the convergence of the Lanczos method, which in turn improves the efficiency of the method for explicit dynamics. Indeed, for $p=2$ and a tolerance of $10^{-3}$, Lanczos needed $486$ iterations for the row-sum technique $P_1$ while it needed only $243$ iterations for the block lumped mass matrices $\mathcal{P}_1$ and $\mathcal{P}_2$. The number of iterations rose to $891$ for the row-sum while it remained constant for $\mathcal{P}_1$ and $\mathcal{P}_2$ when increasing the spline degree to $p=3$. The widely different behavior of the outliers for the row-sum technique and their block counterpart explains the stunning difference in iteration count. While the outliers for $\mathcal{P}_i$ sharply increase as they approach those for the consistent mass, those for $P_1$ barely stand out and Lanczos struggles to converge to these ``outliers''. The spectra for $40$ uniformly spaced rotation angles between $0$ and $2 \pi$ are shown in Figures \ref{fig: 2D_Laplace_rotating_square_block_LM_angle_p2} and \ref{fig: 2D_Laplace_rotating_square_block_LM_angle_p3} for $p=2$ and $p=3$, respectively, and consistently capture the same trend.

\begin{figure}[htbp]
    \centering
    \includegraphics[scale=0.5]{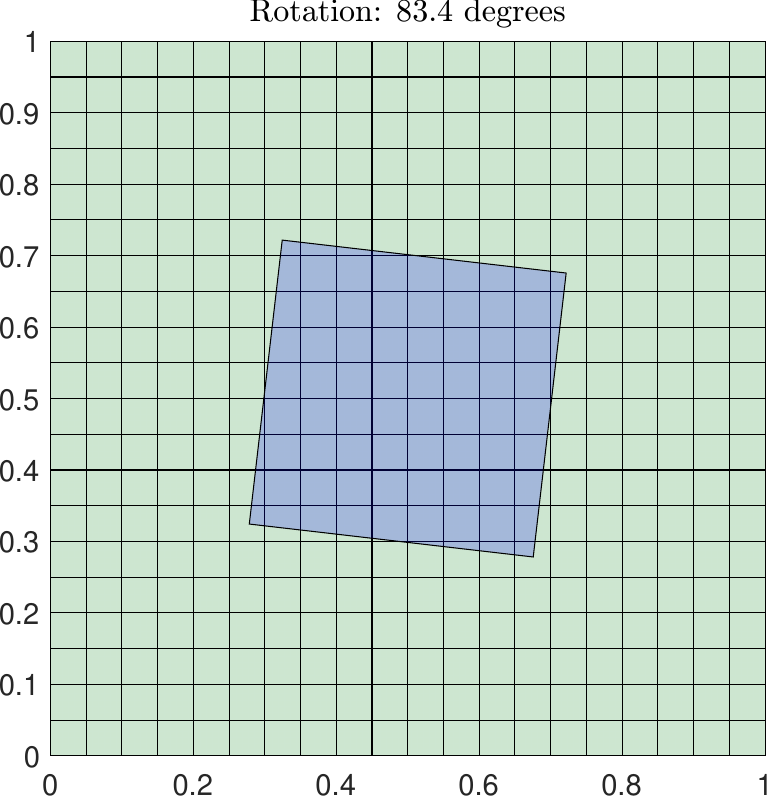}
    \caption{Shifted and rotated square}
    \label{fig: 2D_rotating_square_geometry}
\end{figure}

\begin{figure}[htbp]
    \centering
    \begin{subfigure}[t]{0.48\textwidth}
    \centering
    \includegraphics[width=\textwidth]{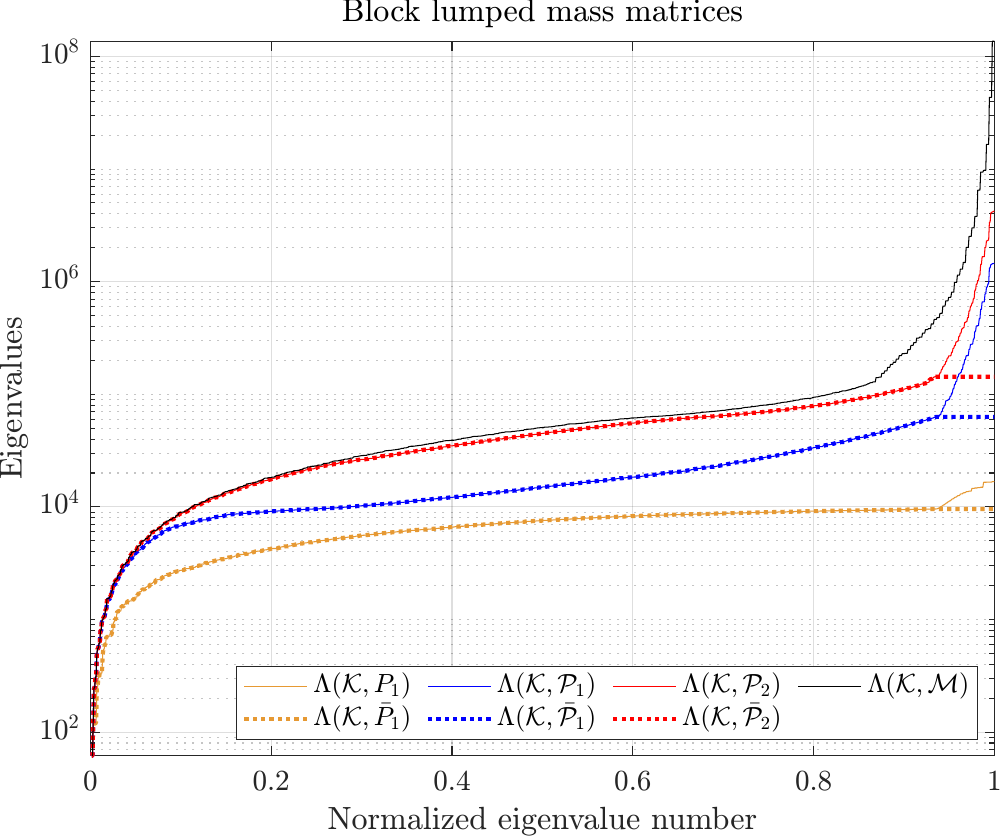}
    \caption{Spectrum for $p=2$}
    \label{fig: 2D_Laplace_rotating_square_block_LM_p2}
    \end{subfigure}    
    \hfill
    \begin{subfigure}[t]{0.48\textwidth}
    \centering
    \includegraphics[width=\textwidth]{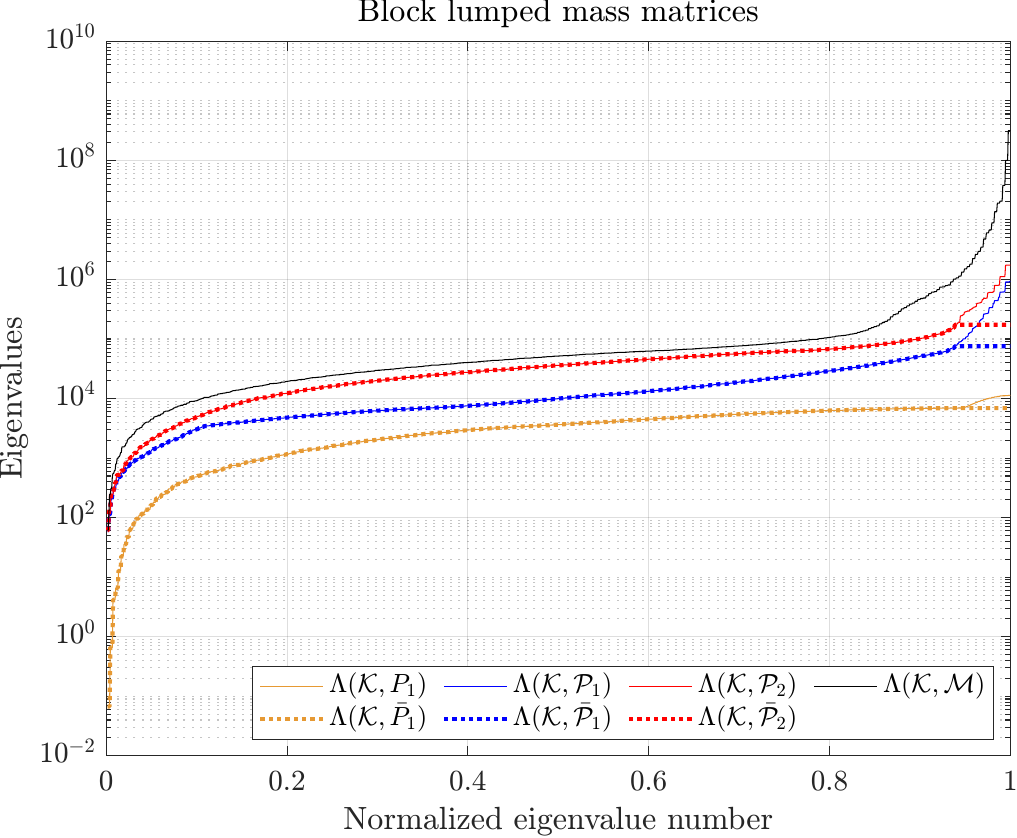}
    \caption{Spectrum for $p=3$}
    \label{fig: 2D_Laplace_rotating_square_block_LM_p3}
    \end{subfigure}
    \hfill
    \caption{Spectrum}
    \label{fig: spectrum_rotating_square}
\end{figure}

\begin{figure}[htbp]
    \centering
    \begin{subfigure}[t]{0.48\textwidth}
    \centering
    \includegraphics[width=\textwidth]{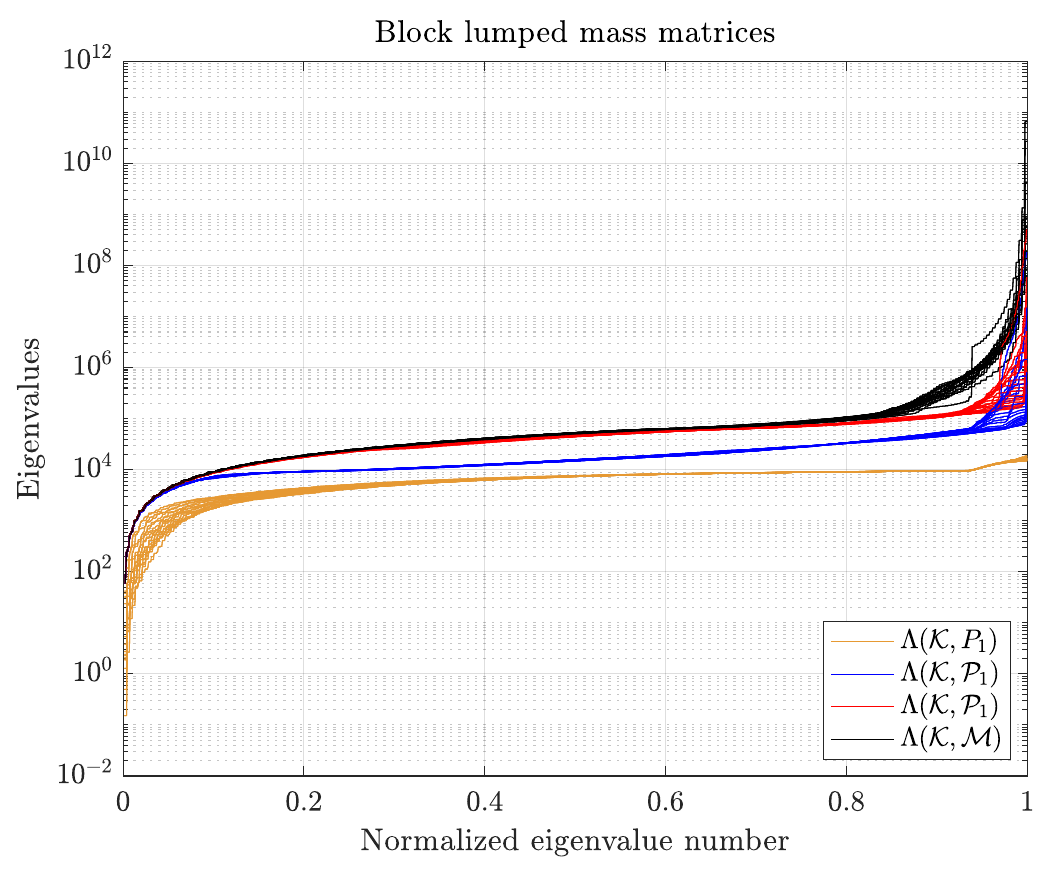}
    \caption{Spectrum for $p=2$}
    \label{fig: 2D_Laplace_rotating_square_block_LM_angle_p2}
    \end{subfigure}    
    \hfill
    \begin{subfigure}[t]{0.48\textwidth}
    \centering
    \includegraphics[width=\textwidth]{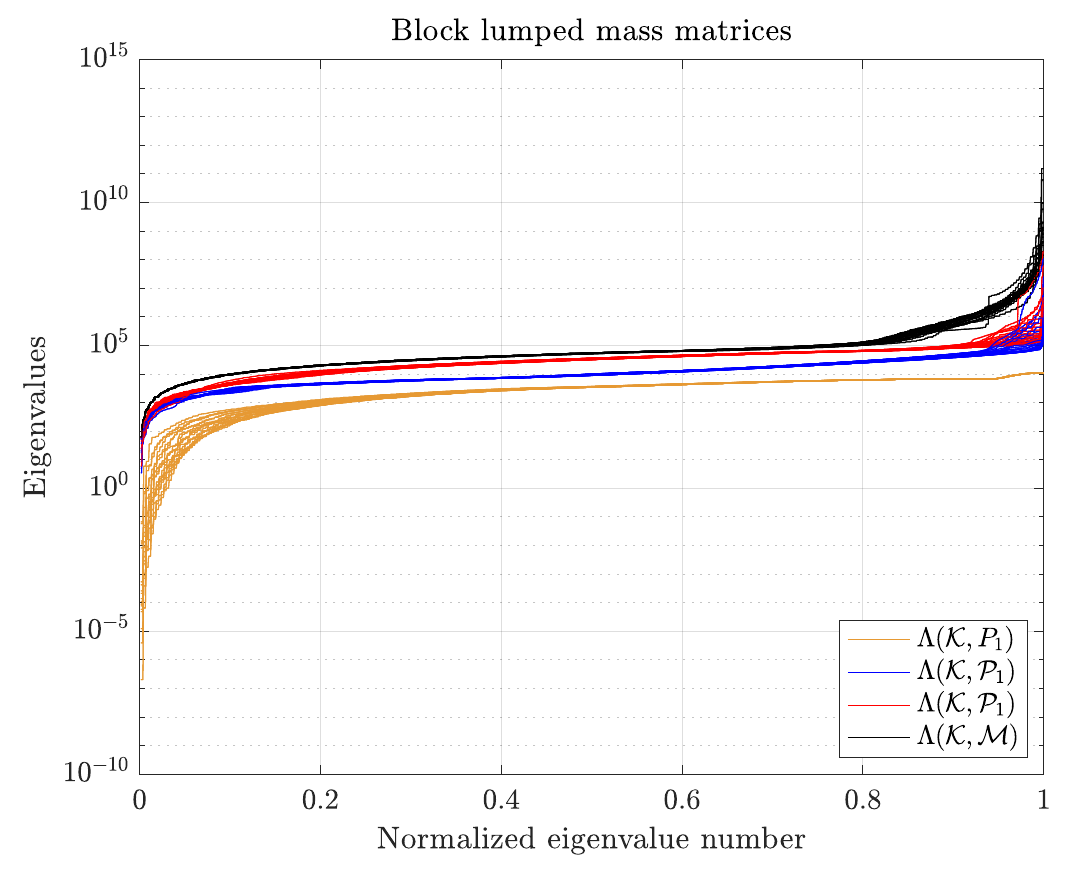}
    \caption{Spectrum for $p=3$}
    \label{fig: 2D_Laplace_rotating_square_block_LM_angle_p3}
    \end{subfigure}
    \hfill
    \caption{Spectrum for $40$ uniformly spaced rotation angles between $0$ and $2 \pi$}
    \label{fig: spectrum_rotating_square_angle}
\end{figure}
\end{example}

\section{Conclusion}
\label{se: conclusion}
In this article, we have proposed mass lumping and outlier removal techniques for nontrivial isogeometric discretizations, including multipatch and trimmed geometries. Our mass lumping techniques focus on reducing the bandwidth of the consistent mass and therefore significantly accelerate sparse direct solvers. Moreover, they provably do not deteriorate the CFL condition of the original problem and oftentimes improve it, thereby extending the methods proposed in \cite{voet2023mathematical} to more realistic settings. For a significant increase in step size, we suggest purging persistent outliers by deflating the spectrum. This outlier removal technique is independent of the lumping strategies proposed herein and contrary to other existing techniques, it only relies on standard eigensolvers, whose cost per iteration is comparable to explicit time integration methods and whose efficiency is enhanced by the large eigenvalue gaps characterizing outliers. Numerical experiments have shown that the cost for computing a few eigenpairs is rapidly amortized by the subsequent increase in critical time step.

\appendix
\section{Appendix}
It was shown in \cite{bressan2019approximation} that maximally smooth spline spaces provide better approximation per degree of freedom than $C^0$ finite element spaces in almost all cases of practical interest. In this section we shall see how this improved approximation guarantees that there are fewer highly inaccurate outlier eigenvalues in the case of smooth spline approximations than for $C^0$ FEM for any elliptic PDE. 
In this section we use the standard Sobolev spaces
\begin{equation*}
H^r(\Omega):= \{u\in L^2(\Omega) : \partial_1^{\alpha_1}\cdots\partial_d^{\alpha_d} u \in L^2(\Omega),\, 1\leq\alpha_1+\cdots+\alpha_d\leq r, \, \alpha_i \geq 0\},
\end{equation*}
with corresponding norms $\|\cdot\|_{H^r}$ given by
\begin{equation*}
  \|u\|^2_{H^r} = \sum_{0\leq\alpha_1+\cdots+\alpha_d\leq r} \|\partial_1^{\alpha_1}\cdots\partial_d^{\alpha_d} u\|^2_{L^2}.
\end{equation*}
Furthermore, $H^{-1}(\Omega)$ is the usual dual space to $H^1_0(\Omega)=\{u\in H^1: u|_{\partial\Omega}=0\}$ with corresponding dual norm $\|\cdot\|_{H^{-1}}$.
For simplicity, we will only consider a second-order elliptic problem with zero Dirichlet boundary conditions. The variational form of such a PDE can be stated as: given $f\in H^{-1}(\Omega)$ find $u\in H^1_0(\Omega)$ such that
\begin{align}\label{bvp}
a(u,v) = (f,v) \qquad \forall v\in H^1_0(\Omega).
\end{align}
The problem is well-posed if
\begin{equation}\label{well-posed}
\begin{aligned}
\|u\|^2_{H^1}&\lesssim a(u,u) &&\forall u\in H^1_0(\Omega),
\\
a(u,v)&\lesssim\|u\|_{H^1}\|v\|_{H^1} &&\forall u,v\in H^1_0(\Omega).
\end{aligned}
\end{equation}
Here we use the notation $a\lesssim b$ as shorthand for the inequality $a\leq Cb$ for some constant $C>0$ independent of $u$.
For a well-posed PDE the solution $u$ in \eqref{bvp} satisfies
\begin{align}\label{pde-bound}
\|u\|_{H^1}\lesssim \|f\|_{H^{-1}}.
\end{align}
If the coefficients in $a(\cdot,\cdot)$ are sufficiently smooth and the domain $\Omega$ is either convex or its boundary $\partial\Omega$ is also sufficiently smooth then it follows from the elliptic regularity theorem that \eqref{pde-bound} can be improved to
\begin{align}\label{pde-ell-reg}
\|u\|_{H^2}\lesssim \|f\|_{L^2},
\end{align}
see \cite[Chapter 6.3]{evans2022partial} and \cite[Chapter 3.2]{grisvard2011elliptic} for the details. In fact, for any nonnegative integer $m$, if the coefficients in $a(\cdot,\cdot)$ are $C^{m+1}$ and the boundary $\partial\Omega$ is $C^{m+2}$ then we have the estimate \cite[Chapter 6.3]{evans2022partial}
\begin{align}\label{pde-better-ell-reg}
\|u\|_{H^{m+2}}\lesssim \|f\|_{H^m}.
\end{align}
Let us now consider the eigenvalue problem: find $\mu_j\in \mathbb{R}$ and $u_j\in H^1_0(\Omega)$, $j=1,2,\ldots$, such that
\begin{align}\label{eig}
a(u_j,v) = \mu_j(u_j,v) \qquad \forall v\in H^1_0(\Omega).
\end{align}
As explained in Section \ref{se: model_problem}, problem \eqref{eig} can be discretized using B-splines to obtain the problem:
 find $\lambda_j\in \mathbb{R}$ and $\mathbf{u}_j\in \mathbb{R}^n$, $j=1,2,\ldots,n$, 
such that
\begin{align}\label{eig-disc}
K\mathbf{u}_j=\lambda_jM\mathbf{u}_j.
\end{align}
To simplify the analysis we assume that $p=p_1=\ldots=p_d$, $k=k_1=\ldots=k_d$ and that the mesh is uniform. Consider the (pushforward) $L^2$ projection $\Pi^k_{p,n}$ onto the (pushforward) spline space $\mathcal{S}^{\mathbf{k}}_{\mathbf{p},\mathbf{\Xi}}$. This projection is stable in $H^1(\Omega)$ since the mesh is uniform. For a non-uniform mesh a different projection should be considered in the following analysis (ideally, the so-called Ritz projection); see \cite{manni2022application} for the details in the case of the Laplacian in one space dimension.
It follows from the min-max theory of Strang and Fix \cite{strang2008analysis} that the error between the discrete eigenvalue $\lambda_j$ in \eqref{eig-disc} and the true eigenvalue $\mu_j$ in \eqref{eig} is bounded by the error $\|u_j-\Pi^k_{p, n}u_j\|_{L^2}$, where $u_j$ is the corresponding eigenfunction in \eqref{eig}.

Following \cite{sande2020explicit}, we define the constant $C_{p,k,r}$ as follows.
If $k=p-1$, we let
\begin{equation*}
C_{p,p-1,r}:=\left(\frac{1}{\pi}\right)^r
\end{equation*} 
and if 
$k\leq p-2$, we~let 
\begin{equation*}
C_{p,k,r}:=\begin{cases}
\left(\dfrac{1}{2}\right)^{r}\left(\dfrac{1}{\sqrt{(p-k)(p-k+1)}}\right)^r, &k\geq r-2,
\\[0.5cm]
\left(\dfrac{1}{2}\right)^{r}\left(\dfrac{1}{\sqrt{(p-k)(p-k+1)}}\right)^{k+1}\sqrt{\dfrac{(p+1-r)!}{(p-1+r-2k)!}}, & k<r-2.
\end{cases}
\end{equation*}
For any $u\in H^r(\Omega)$ it is shown in \cite{sande2020explicit} that we have the explicit error estimate
\begin{align}\label{ineq:gen}
\|u-\Pi^k_{p, n}u\|_{L^2} \leq C_{\text{Geo}}C_{p,k,r}h^r\|u\|_{H^r}, \quad r\leq p+1
\end{align}
where the constant $C_{\text{Geo}}$ only depends on the geometry maps $F_i$, $i=1,\ldots,N_p$ and is explicitly given in \cite{sande2020explicit}. In classical finite element methods, the smoothness $k=0$ and in this case the constant $C_{p,0,r}$ satisfies, for $p\geq2$, the following inequalities \cite[Remark~3]{sande2020explicit} 
\begin{align*}
    C_{p,0,r}=
\left(\frac{1}{2\sqrt{p(p+1)}}\right)^r
\leq \left(\frac{1}{2p}\right)^r, \quad r=1,2
\end{align*}
and 
\begin{align*}
C_{p,0,r}=
\left(\frac{1}{2}\right)^{r}\left(\dfrac{1}{\sqrt{p(p+1)}}\right)\sqrt{\frac{(p+1-r)!}{(p-1+r)!}}
\leq  \left(\frac{e}{4p}\right)^r, \quad r>2.
\end{align*}
It follows from the spline dimension formula that the mesh size $h \sim (p-k)/n^{1/d}$ and thus $C_{p,k,r}h^r\sim C_{p,k,r}((p-k)/n^{1/d})^r$. More explicitly, with the upper bounds previously obtained:
\begin{itemize}
    \item For $r=1,2$,
    \begin{equation*}
        C_{p,k,r}\left(\frac{p-k}{n^{1/d}}\right)^r \lesssim \frac{1}{n^{r/d}}
        \begin{cases}
            \left(\frac{1}{2}\right)^r & \text{if } k=0, \\
            \left(\frac{1}{\pi}\right)^r & \text{if } k=p-1.
        \end{cases}
    \end{equation*}
    \item For $r>2$,
    \begin{equation*}
        C_{p,k,r}\left(\frac{p-k}{n^{1/d}}\right)^r \lesssim \frac{1}{n^{r/d}}
        \begin{cases}
            \left(\frac{e}{4}\right)^r & \text{if } k=0, \\
            \left(\frac{1}{\pi}\right)^r & \text{if } k=p-1.
        \end{cases}
    \end{equation*}
\end{itemize}
Although these estimates are only upper bounds, the difference between the $C^0$ and $C^{p-1}$ cases is readily appreciated. For instance, for $p=3$ and $r=4$ the upper bound is about $20$ times smaller in the $C^{p-1}$ case than in the $C^0$ case. The reader may refer to \cite[Figs. 1 and 2]{sande2020explicit} for a graphical comparison of the constants for different values of $p$ and $r$. In fact, by using a lower bound on the best approximation constant in the $C^0$ case, it is shown in \cite{bressan2019approximation} that the maximally smooth approximation constant $C_{p,p-1,r}$ becomes exponentially better than the best achievable approximation constant for $k=0$ as the degree $p \geq 3$ and $r$ increase.

\begin{theorem}\label{theorem-eig1}
Let $n$ be the dimension of $\mathcal{S}^{\mathbf{k}}_{\mathbf{p},\mathbf{\Xi}}$. For any  $j=1,\ldots,n$ let $u_j$ be the $j$th eigenfunction of \eqref{eig} with corresponding eigenvalue $\mu_j$. Then, for all $0\leq k\leq p-1$, we have 
\begin{align}\label{ineq-eig1}
\frac{\|u_j-\Pi^k_{p, n} u_j\|_{L^2}}{\|u_j\|_{L^2}} \leq C_{\text{PDE}}C_{\text{Geo}}C_{p,k,1}h\sqrt{\mu_j}.
\end{align}
\end{theorem}
\begin{proof}
Let $u=u_j$ in \eqref{ineq:gen} with $r=1$ and use \eqref{well-posed} together with $a(u_j,u_j)=\mu_j\|u_j\|_{L^2}^2$.
\end{proof}
From our previous discussion, it follows that for fixed $n$ and $r$ and for a given tolerance, maximally smooth splines allow to approximate a larger fraction of the eigenvalues than $C^0$ finite element spaces. Moreover, if the PDE satisfies elliptic regularity then the error estimate in \eqref{ineq-eig1} can be further improved.

\begin{theorem}\label{theorem-eig2}
Let $r\leq p+1$ and assume the coefficients in $a(\cdot,\cdot)$ are $C^{r-1}$ and the boundary $\partial\Omega$ is $C^{r}$.
Let $n$ be the dimension of $\mathcal{S}^{\mathbf{k}}_{\mathbf{p},\mathbf{\Xi}}$. For any $j=1,\ldots,n$ let $u_j$ be the $j$th eigenfunction of \eqref{eig} with corresponding eigenvalue $\mu_j$. Then, for all $0\leq k\leq p-1$, we have 
\begin{align*}
 \frac{\|u_j-\Pi^k_{p, n} u_j\|_{L^2}}{\|u_j\|_{L^2}} \leq C_{\text{PDE}}C_{\text{Geo}}C_{p,k,r}h^r\mu_j^{r/2},
\end{align*}
\end{theorem}
\begin{proof}
If $r$ is even then iterate the elliptic regularity result in \eqref{pde-better-ell-reg} with $f=\mu_ju_j$, $r/2$ times and use \eqref{ineq:gen}. If $r$ is odd then additionally use the argument of Theorem \ref{theorem-eig1} once.
\end{proof}
The improved approximation for maximally smooth splines compared with $C^0$ finite element spaces will only get better as $r$ increases in Theorem \ref{theorem-eig2}, and we are guaranteed good approximation of a larger fraction of the eigenvalues than for $C^0$ FEM. Using the arguments presented in \cite{strang2008analysis}, the statement is also valid for eigenfunctions.

\end{document}